\documentclass{article}
\usepackage[english]{babel}
\usepackage{amsmath,amssymb,stmaryrd,enumerate,bbm,latexsym,theorem}

\usepackage[top=2.5cm,bottom=2.5cm,right=2cm,left=2cm]{geometry}
\usepackage{authblk}
\numberwithin{equation}{section}

\newcommand{\assign}{:=}
\newcommand{\comma}{{,}}
\newcommand{\nobracket}{}
\newcommand{\tmop}[1]{\ensuremath{\operatorname{#1}}}
\newcommand{\tmstrong}[1]{\textbf{#1}}
\newcommand{\tmtextit}[1]{{\itshape{#1}}}
\newenvironment{itemizedot}{\begin{itemize} }{\end{itemize}}
\newenvironment{proof}{\noindent\textbf{Proof\ }}{\hspace*{\fill}$\Box$\medskip}
\newcounter{nnacknowledgments}

{\theorembodyfont{\rmfamily}\newtheorem{acknowledgments*}[nnacknowledgments]{Acknowledgments}}
\newtheorem{theorem}{Theorem}[section]

\newtheorem{lemma}[theorem]{Lemma}
\newtheorem{proposition}[theorem]{Proposition}

\usepackage[hidelinks]{hyperref}

\date{}
\author{Tej-Eddine Ghoul}
\author{Nader Masmoudi}
\author{Eliot Pacherie}
\affil{NYUAD Research Institute, New York University Abu Dhabi}

\begin{document}

\title{Localisation of perturbations of a constant state in a traffic flow
model}

\maketitle

\begin{abstract}
  We consider, in the Aw-Rascle-Zhang traffic flow model, the problem of the
  asymptotic stability of constant flows. By using a perturbative approach, we
  show the stability in a larger space of perturbation than previous results.
  Furthermore, we are able to compute where the perturbation is mainly
  localised in space for a given time, based on the localisation of the
  perturbation initially. These new ideas can be applied to various other
  models of hyperbolic conservation laws with relaxations.
\end{abstract}

\section{Introduction and presentation of the results}

\subsection{The Aw-Rascle-Zhang model}

\ We look here at the Aw-Rascle-Zhang traffic flow model introduced in
{\cite{MR1750085}} and {\cite{articleZhangZhang}}, that is composed of two
equations on unknowns $\rho, u : \mathbb{R}_x \times \mathbb{R}^+_t
\rightarrow \mathbb{R}$. The quantity $\rho (x, t)$ for $x \in \mathbb{R}, t
\geqslant 0$ is the density of cars on an infinite highway at position $x$ and
time $t$, and $u (x, t)$ is the speed at that same point, and they are
solutions of the system
\begin{equation}
  (\tmop{ARZ}) \left\{\begin{array}{l}
    \partial_t \rho + \partial_x (\rho u) = 0\\
    \partial_t (u + h (\rho)) + u \partial_x (u + h (\rho)) = \frac{1}{\tau}
    (U (\rho) - u) .
  \end{array}\right. \label{maineq}
\end{equation}
The first equation encode the fact that the cars move at speed $u$, and
implies the conservation of mass. The second equation describes the variation
of the speed of the cars. Two effects are taken into account in this model, in
addition to the fact that the speed moves with the cars. First, a pressure $h
(\rho)$, a strictly increasing function of $\rho$, that makes cars entering an
area with higher density slow down. Secondly, an equilibrium speed $U (\rho) =
U_f (1 - \rho)$ at which cars try to move at, and they relax to this speed at
a caracteristic time $\tau > 0$. $U_f > 0$ is the free speed, at which cars
move if alone on the road, and finally $\rho_{\max} = 1$ is the maximum
density, at which $U (\rho_{\max}) = 0$, and $h (\rho)$ blows up when $\rho
\rightarrow \rho_{\max}$.

This model is an improvement of previous traffic flow models like
{\cite{MR72606}}, {\cite{doi:10.1287/opre.4.1.42}} that captures many physical
phenomenons, see {\cite{MR1750085}}.

\

We consider equation (\ref{maineq}) with $\tau > 0, U (\rho) = U_f (1 -
\rho), U_f > 0$ and $h \in C^{\infty} (] 0, 1 [, \mathbb{R})$ with $h' > 0$
and $h (\rho) \rightarrow + \infty$ when $\rho \rightarrow 1$. There exists
constant solutions to this equation: $\rho = \rho_0 \in] 0, 1 [$ and $u = U
(\rho_0)$. We are interested here in their stability.

\subsection{Previous stability results on the constant flow}

The Aw-Rascle-Zhang model is part of a large class of hyperbolic systems with
relaxation of the form
\[ \left\{\begin{array}{l}
     \partial_t \rho + \partial_x (f (\rho, w)) = 0\\
     \partial_t w + \partial_x (g (\rho, w)) = r (\rho, w)
   \end{array}\right. \]
on the unknown $(\rho, w)$, where $f, g, r$ are known function. the system
$(\tmop{ARZ})$ is of this form, with the quantity for any $c \in] 0, 1 [$,
\[ G (\rho) \assign \rho \int_c^{\rho} \frac{h' (\nu)}{\nu} d \nu, w \assign u
   + G (\rho), \]
if we take
\[ f (\rho, w) = \rho w - \rho G (\rho), g (\rho, w) = \frac{w^2}{2} - \frac{g
   (\rho)^2}{2}, r (\rho, w) = \frac{1}{\tau} (U (\rho) + G (\rho) - w) . \]
In this kind of model and natural generalisations, the question of the
stability of constant flows and small travelling waves have been studied
extensively using energy methods. See for instance {\cite{MR3355354}} and
reference therein, as well as {\cite{MR1346371}}, {\cite{MR899951}} for the
stability of constant flows, {\cite{MR872145}}, {\cite{MR1409022}} for the
stability of travelling waves, and {\cite{MR2349349}}, {\cite{MR1755074}} for
generalisations. See also {\cite{MR4195556}} for a recent work an another
model.

\

Let us summarize the results of {\cite{MR1346371}} here in the particular
case of the system (ARZ). Consider an initial data $(\rho_0 + \rho_i, U
(\rho_0) + u_i)$ and let us write the solution of (\ref{maineq}) on the form
$(\rho_0 + \rho, U (\rho_0) + u)$. In {\cite{MR1346371}}, it is shown that at
leading order, $\rho$ satisfies an equation of the form
\[ \partial_t \rho - \nu \partial_x^2 \rho + \partial_x (f_{\ast} (\rho)) = 0,
\]
where $\nu \in \mathbb{R}$ and $f_{\ast}$ is a smooth function. For this
equation to have a solution, it is necessary that $\nu > 0$. This is called
the subcaracteristic condition, and in the case of $(\tmop{ARZ})$, it is
equivalent to $s_{\tmop{cc}} (\rho_0) \assign h' (\rho_0) - U_f > 0$. Then,
considering that $\rho$ is small, this equation is developed as
\[ \partial_t \rho - \nu \partial_x^2 \rho + \lambda_{\ast} \partial_x \rho +
   \frac{f_{\ast}'' (0)}{2} \partial_x (\rho^2) = 0. \]
In the case of $(\tmop{ARZ})$, we can compute that $\lambda_{\ast} \assign U
(\rho_0) - \rho_0 U_f$. This Burgers type equation admit self similar
solutions $\theta (x, t)$, that satisfy
\begin{equation}
  \left\{\begin{array}{l}
    \partial_t \theta - \nu \partial_x^2 \theta + \lambda_{\ast} \partial_x
    \theta + \frac{f_{\ast}'' (0)}{2} \partial_x (\theta^2) = 0\\
    \theta (x, - 1) = m \delta (x)
  \end{array}\right. \label{theta0}
\end{equation}
where $m = \int_{\mathbb{R}} \rho_i (x) d x$ is the mass of the perturbation
(which is conserved with time) and $\delta$ is the dirac function. In the case
$f''_{\ast} (0) = 0$, this solution is
\[ \theta (x, t) = \frac{m}{4 \pi \nu \sqrt{1 + t}} \exp \left( - \frac{(x -
   \lambda_{\ast} (1 + t))^2}{4 \nu (1 + t)} \right), \]
that is the Gaussian. In the case $f''_{\ast} (0) \neq 0$ the solution is also
explicit and has similar properties (see {\cite{MR1346371}} below equation
(3.1)).

The main result of {\cite{MR1346371}} is the asymptotic stability of this
diffusion wave.

\begin{theorem}[{\cite{MR1346371}}, Theorem 1.1]
  \label{th4}Consider equation (\ref{maineq}) with an initial data $(\rho_0 +
  \rho_i, U (\rho_0) + u_i)$, $\rho_i, u_i \in C^2_{\tmop{loc}}$, with $\rho_0
  \in] 0, 1 [$ and $s_{\tmop{cc}} (\rho_0) = h' (\rho_0) - U_f > 0$. With $m =
  \int_{\mathbb{R}} \rho_i (x) d x$, we consider $\theta$ the solution of
  (\ref{theta0}) and $\theta_0$ its value at $t = 0$. Then, there exists
  $\delta_0 > 0$ depending on $\rho_0, h' (\rho_0), U_f$ such that, if the
  functions $z_0, w_0$ defined by
  \[ \rho_i = \theta_0 + \partial_x z_0, u_i = U (\theta_0) + w_0 \]
  are such that
  \[ | m | + \sum_{l = 0}^3 \| \partial_x^l z_0 \|_{L^2 (\mathbb{R})} +
     \sum_{l = 0}^2 \| \partial_x^l w_0 \|_{L^2 (\mathbb{R})} + \| z_0 \|_{L^1
     (\mathbb{R})} + \| w_0 \|_{L^1 (\mathbb{R})} \leqslant \delta_0, \]
  then (\ref{maineq}) admit a classical solution for this initial condition
  for all positive time. Furthermore, if we decompose this solution as
  $(\rho_0 + \rho, U (\rho_0) + u)$, then for $l \in \{ 0, 1, 2 \}$,
  \[ \| \partial_x^l (\rho - \theta) (., t) \|_{L^p} = O_{t \rightarrow
     \infty} (1 + t)^{- 1 - \frac{l}{2} + \frac{1}{2 p}} \tmop{for} p \in [1,
     + \infty] \]
  and
  \[ \| \partial_x^l (u - U (\theta)) (., t) \|_{L^p} = O_{t \rightarrow
     \infty} (1 + t)^{- 1 - \frac{l}{2} + \frac{1}{2 p}} \tmop{for} p \in [2,
     + \infty] . \]
\end{theorem}

This stability result is proven using energy methods. Remark that it requires
the initial data to be well localized (the condition $\| z_0 \|_{L^1
(\mathbb{R})} < + \infty$ is a lot stronger than $\rho_i \in L^1
(\mathbb{R})$). For instance, take $\chi$ a smooth positive function with
value $1$ in $[- 1, 1]$ and $0$ outside of $[- 2, 2]$. Then Theorem \ref{th4}
can not be applied to the case $\rho_i = \varepsilon (\chi (x - A) + \chi (x +
A)), u_i = 0$ for $\varepsilon > 0$ small, uniformly in $A$ (it would require
$A \ll \frac{1}{\varepsilon}$). The stability has been shown for a larger
space of perturbation than Theorem \ref{th4}, but for simplier hyperbolic
systems, see {\cite{MR4088950}} and {\cite{MR2303479}}.

This localisation on the initial data is necessary to get the asymptotic
profile. Our goal here is to use a new approach, that will allow us to extend
the space of perturbations. We also want to give a different type of precision
on the behavior of the perturbation. Broadly speaking, we want to know where
the perturbation is mainly localized, given some informations on its position
initially, even if it's not close to a diffusion wave, but without necessarly
having an asymptotic profile. First, let us start by revisiting the linear
problem around a constant flow, on which such properties can be made explicit.

\subsection{The linear stability}

Consider the equation $\left( \ref{maineq} \right)$ for the functions $\rho_0
+ \rho, U (\rho_0) + u$, and keep only the linear terms in $(\rho, u)$. Then,
the equations become
\begin{equation}
  \partial_t \left(\begin{array}{c}
    \rho\\
    u
  \end{array}\right) + \left(\begin{array}{cc}
    U (\rho_0) & \rho_0\\
    0 & U (\rho_0) - \rho_0 h' (\rho_0)
  \end{array}\right) \partial_x \left(\begin{array}{c}
    \rho\\
    u
  \end{array}\right) + \frac{1}{\tau} \left(\begin{array}{cc}
    0 & 0\\
    U_f & 1
  \end{array}\right) \left(\begin{array}{c}
    \rho\\
    u
  \end{array}\right) = 0 \label{linconst} .
\end{equation}
This is the linearized equation around a constant flow $(\rho_0, U (\rho_0))$
for $\rho_0 \in] 0, 1 [$. A usefull remark about this linear problem is that
the equations on $\rho$ and $u$ can be decoupled. In particular, with the
speed $\lambda_{\ast} = U (\rho_0) - \rho_0 U_f$, the functions
\[ q (x, t) \assign \rho (x + \lambda_{\ast} t, t), v (x, t) \assign u (x +
   \lambda_{\ast} t, t) \]
satisfy the equations
\[ \partial_t^2 q + (\lambda_1^0 + \lambda^0_2) \partial_{x t}^2 q +
   \lambda^0_1 \lambda^0_2 \partial_x^2 q + \frac{1}{\tau} \partial_t q = 0 \]
and
\[ \partial_t^2 v + (\lambda_1^0 + \lambda_2^0) \partial_{x t}^2 v +
   \lambda_1^0 \lambda^0_2 \partial_x^2 v + \frac{1}{\tau} \partial_t v = 0,
\]
where $\lambda_1^0 \assign \rho_0 U_f, \lambda^0_2 \assign \rho_0 (U_f - h'
(\rho_0))$. See Lemma \ref{lineq31} for the proof of these computations. The
subcaracteristic condition $s_{\tmop{cc}} (\rho_0) > 0$ implies that
$\lambda_1^0 > 0, \lambda_2^0 < 0$, and in particular $\lambda^0_1 \lambda^0_2
< 0$. The equation satisfied by $q$ (and $v$) is called the damped wave
equation. It turns out that we can compute explicitely its solution.

\subsubsection{The damped wave equation}

In this section, we are interested in the properties of a solution to the
problem
\begin{equation}
  \left\{\begin{array}{l}
    \partial_t^2 f + (\lambda_1 + \lambda_2) \partial_{x t}^2 f + \lambda_1
    \lambda_2 \partial_x^2 f + \delta \partial_t f = S (x, t)\\
    f_{| t = 0 \nobracket} = f_0\\
    \partial_t f_{| t = 0 \nobracket} = f_1
  \end{array}\right. \label{dwe}
\end{equation}
where $\delta > 0, \lambda_1 > 0, \lambda_2 < 0$ are constants, because $q$
and $v$ satisfy this equation with $S = 0 ; \delta = \frac{1}{\tau}, \lambda_1
= \lambda_1^0$ and $\lambda_2 = \lambda_2^0$. We can compute explicitely its
solution.

\begin{proposition}
  \label{dwesol}For \ $\delta > 0, \lambda_1 > 0, \lambda_2 < 0$, the solution
  to the problem (\ref{dwe}) with $f_0, f_1, S \in C^2_{\tmop{loc}}$ satisfies
  \begin{eqnarray*}
    f (x, t) & = & \int_{\lambda_2 t}^{\lambda_1 t} V (y, t) (\delta f_0 +
    (\lambda_1 + \lambda_2) f_0' + f_1) (x - y) d y\\
    & + & \int_{\lambda_2 t}^{\lambda_1 t} \partial_t V (y, t) f_0 (x - y) d
    y\\
    & + & \lambda_1 \frac{e^{\frac{- \delta \lambda_1 t}{\lambda_1 -
    \lambda_2}}}{\lambda_1 - \lambda_2} f_0 (x - \lambda_1 t) - \lambda_2
    \frac{e^{\frac{\delta \lambda_2 t}{\lambda_1 - \lambda_2}}}{\lambda_1 -
    \lambda_2} f_0 (x - \lambda_2 t)\\
    & + & \int_0^t \int_{\lambda_2 (t - s)}^{\lambda_1 (t - s)} V (y, t - s)
    S (x - y, s) d y d s,
  \end{eqnarray*}
  where
  \[ V (y, t) \assign \frac{e^{- \frac{2 \delta}{(\lambda_1 - \lambda_2)^2}
     \left( - \lambda_1 \lambda_2 t + \frac{\lambda_1 + \lambda_2}{2} y
     \right)}}{\lambda_1 - \lambda_2} I_0 \left( \frac{2 \delta \sqrt{-
     \lambda_1 \lambda_2}}{(\lambda_1 - \lambda_2)^2} \sqrt{- (y - \lambda_1
     t) (y - \lambda_2 t)} \right), \]
  $I_0$ being the modified Bessel function of the first kind and of order $0$.
\end{proposition}

We can also show some estimates on the kernel $V$:

\begin{lemma}
  \label{Vest}For $k, n \in \mathbb{N}, \lambda_1 > 0, \lambda_2 < 0, \delta >
  0$, the function
  \[ V (y, t) = \frac{e^{- \frac{2 \delta}{(\lambda_1 - \lambda_2)^2} \left( -
     \lambda_1 \lambda_2 t + \frac{\lambda_1 + \lambda_2}{2} y
     \right)}}{\lambda_1 - \lambda_2} I_0 \left( \frac{2 \delta \sqrt{-
     \lambda_1 \lambda_2}}{(\lambda_1 - \lambda_2)^2} \sqrt{- (y - \lambda_1
     t) (y - \lambda_2 t)} \right) \]
  is smooth on $t \geqslant 0, y \in [\lambda_2 t, \lambda_1 t]$ and there
  exists $C_{k, n} (\lambda_1, \lambda_2, \delta), a_0 (\lambda_1, \lambda_2,
  \delta) > 0$ such that, for any $t \geqslant 0, y \in [\lambda_2 t,
  \lambda_1 t]$, we have
  \[ | \partial_x^k \partial_t^n V (y, t) | \leqslant \frac{C_{k, n}
     (\lambda_1, \lambda_2, \delta) e^{- a_0 \frac{y^2}{1 + t}}}{(1 +
     t)^{\frac{1}{2} + \frac{k}{2} + n}} . \]
\end{lemma}

Section \ref{sdampedwave} is devoted to the proof of these two results. They
follow from the computation of the kernel of the simplier and well known
equation $\partial_t^2 u - \partial_x^2 u + \partial_t u = 0$ (see
{\cite{courant2008methods}}).

\subsubsection{Statement of the linear stability}

With the explicit solution coming from Proposition \ref{dwesol} and estimate
in Lemma \ref{Vest} on the kernel that appears, we can show a precise estimate
on the solution of the linearized problem (\ref{linconst}).

\begin{proposition}
  \label{constlinstab}Consider the equation (\ref{linconst}) with $\rho_0 \in]
  0, 1 [$ and $s_{\tmop{cc}} (\rho_0) = h' (\rho_0) - U_f > 0$ and some
  initial data $\rho_{| t = 0 \nobracket} = \rho_0 + \rho_i, u_{| t = 0
  \nobracket} = U (\rho_0) + u_i$ with $\rho_i, u_i \in C^j_{\tmop{loc}}
  (\mathbb{R}, \mathbb{R})$ for some $j \geqslant 2$. Then, the solution
  $(\rho_0 + \rho, U (\rho_0) + u)$ of (\ref{linconst}) is well defined in
  $C^j_{\tmop{loc}} (\mathbb{R}_x \times \mathbb{R}^+_t, \mathbb{R})$ and we
  have the following estimates. There exists $a_L, b_L > 0$ depending on
  $\rho_0, U_f, h' (\rho_0)$ such that, for any $k, n \in \mathbb{N}$ with $k
  + n \leqslant j,$defining
  \[ \lambda_1^0 = \rho_0 U_f > 0, \lambda_2^0 = \rho_0 (U_f - h' (\rho_0)) <
     0, \lambda_{\ast} = U (\rho_0) - \rho_0 U_f, \]
  there exists $K_{n, k} > 0$ depending on $n, k, \rho_0, U_f, h' (\rho_0)$
  such that
  \begin{eqnarray*}
    &  & | \partial_t^n \partial_x^k (u (x - \lambda_{\ast} t, t)) | + |
    \partial_t^n \partial_x^k (\rho (x - \lambda_{\ast} t, t)) |\\
    & \leqslant & \frac{K_{n, k}}{(1 + t)^{\frac{1}{2} + \frac{k}{2} + n}}
    \int_{\lambda_2^0 t}^{\lambda_1^0 t} e^{- a_L \frac{y^2}{1 + t}} (| u_i (x
    - y) | + | \rho_i (x - y) |) d y\\
    & + & K_{n, k} e^{- b_L t} \left( \sum_{j = 0}^{n + k} (| \rho_i^{(j)} |
    + | u_i^{(j)} |) (x - \lambda_1^0 t) + (| \rho_i^{(j)} | + | u_i^{(j)} |)
    (x - \lambda_2^0 t) \right)
  \end{eqnarray*}
  for all $x \in \mathbb{R}, t \geqslant 0$.
\end{proposition}

Let us first make a few remarks on this result.
\begin{itemizedot}
  \item Suppose that $\rho_i, u_i \in C^j \cap L^1 (\mathbb{R})$. We check
  that this result implies
  \begin{eqnarray*}
    &  & | \partial_t^n \partial_x^k (u (x - \lambda_{\ast} t, t)) | + |
    \partial_t^n \partial_x^k (\rho (x - \lambda_{\ast} t, t)) |\\
    & \leqslant & \frac{K_{n, k} (\| u_i \|_{C^{n + k} \cap L^1 (\mathbb{R})}
    + \| \rho_i \|_{C^{n + k} \cap L^1 (\mathbb{R})})}{(1 + t)^{\frac{1}{2} +
    \frac{k}{2} + n}} .
  \end{eqnarray*}
  \item This formulation allows us to keep information on the localisation of
  the perturbation when time evolves. That is, to estimate $u (x -
  \lambda_{\ast} t, t)$, we can only look at the initial data in a
  neighborhood of $x$ of size $\sqrt{t}$, and the error committed by doing so
  is almost exponentially small in time, while this main term is, if the
  initial data is $L^1 (\mathbb{R})$, of size $\frac{1}{\sqrt{t}}$. This means
  for instance that if our initial perturbation is compactly supported in $[-
  1, 1]$ at $t = 0$, then for any $\varepsilon > 0$, at time $t \geqslant 1$
  the perturbation outside $[(\lambda_{\ast} - \varepsilon) t, (\lambda_{\ast}
  + \varepsilon) t]$ is exponentially small in time. This is something that
  can not be shown easily using energy methods. By linearity of
  (\ref{linconst}) this can also be applied to a sum of localized
  perturbations.
  
  \item Another way to say this is the following. Amongst all the speeds
  between $\lambda_2^0 < 0$ and $\lambda_1^0 > 0$ that the perturbation could
  take, it only, up to exponentially small error, goes at the speed
  $\lambda_{\ast}$. That is, $\rho (\lambda t, t)$ is exponentially small in
  time for all $\lambda \neq \lambda_{\ast}$.
  
  \item If the initial condition is in $L^1 (\mathbb{R})$, we get the same
  decay in time as Theorem \ref{th4}. Remark also that this estimate make
  sense even if the initial condition are not in $L^1 (\mathbb{R})$, but in
  that case we do not have necessarly a decay in time of the solution.
  
  \item The main term in the estimate looks like the one we would get for the
  heat equation with a transport term, that is the solution of $\left(
  \partial_t - \frac{1}{a_L} \partial_x^2 + \lambda^{\ast} \partial_x \right)
  \rho = 0$. It is in fact the dominating effect at first order (which makes
  sense in regards of Theorem \ref{th4} from {\cite{MR1346371}}). We see two
  main difference with the simplier heat + transport system: the integral is
  only on $[\lambda_2^0 t, \lambda_1^0 t]$ instead of $\mathbb{R}$, and we
  have an additional exponentially small error in time.
  
  \item Remark that here, since we only look at the linear problem, we do not
  need any smallness on the initial data. In fact it does not need to decay at
  infinity for this result to hold, thanks to the finite speed of propagation.
  We can also go above the threshold of maximum density $\rho_{\max} = 1$, but
  this is only true for the linear problem, since in that case the function
  $h'$, that blows up at $1$, is taken in $\rho_0$ and not $\rho_0 + \rho$.
\end{itemizedot}

This proposition is a corollary of Proposition \ref{dwesol} and Lemma
\ref{Vest}, see section \ref{sec5vf} for its proof.

\subsection{The nonlinear stability in $L^1 (\mathbb{R})$}

Our goal here is to show a similar estimate as Proposition \ref{constlinstab}
but in the nonlinear case. Taking equation $\left( \ref{maineq} \right)$ for
the functions $\rho_0 + \rho, U (\rho_0) + u$, and then introducing $q (x, t)
= \rho (x - \lambda_{\ast} t, t), v (x, t) = u (x - \lambda_{\ast} t, t)$ and
the non constant caracteristic speeds $\lambda_1 \assign \rho_0 U_f + v,
\lambda_2 \assign \rho_0 U_f - (\rho_0 + q) h' (\rho_0 + q) + v$, we can check
(see Lemma \ref{lem32} for the computation) that the functions $q, v$ satisfy
the equations
\[ \partial_t^2 v + (\lambda_1 + \lambda_2) \partial_{x t}^2 v + \lambda_1
   \lambda_2 \partial_x^2 v + \frac{1}{\tau} \partial_t v + \Omega_1
   \partial_x v = 0 \]
and
\[ \partial_t^2 q + (\lambda_1 + \lambda_2) \partial_{x t}^2 q + \lambda_1
   \lambda_2 \partial_x^2 q + \left( \frac{1}{\tau} + \Omega_2 \right)
   \partial_t q + \Omega_3 \partial_x q = 0, \]
where $\Omega_1, \Omega_2, \Omega_3$ are small if $q, v$ are small. In
particular, it is no longer possible to decouple the two equations.

\

There are several obstacles to do an estimation similar to the linear case,
that will force us to reduce the quality of the estimates. The first one is
the fact that the caracteristic speeds are no longer constant. In the linear
setting, they were $\lambda_1^0 = \rho_0 U_f, \lambda_2^0 = \rho_0 (U_f - h'
(\rho_0))$, and in the nonlinear setting, they are $\lambda_1 = \rho_0 U_f +
v, \lambda_2 = \rho_0 U_f - (\rho_0 + q) h' (\rho_0 + q) + v.$ Concerning
$\lambda_{\ast}$, if we would define it in the nonlinear setting, it would not
even have the same value for $q$ and $v$ (it would be $\lambda_{\ast} +
\Omega_1$ for $v$ and $\lambda_{\ast} + \Omega_3$ for $q$). A usual method to
deal with nonconstant caracteristic speed is to do a change of coordonate to
make them constant. But here, because there are at least three speeds we would
want to make constant at the same time, this is difficult.

The second one is simply to show that the nonlinear terms does not infer
growth or blow up. They will be consider as source terms of the linear
problem. This is why in Proposition \ref{dwesol} we did the computations with
a generic source term $S$. This iself pose a difficulty, since terms in the
source term will have as many derivatives as linear terms (for instance the
equation on $q$ will contain $(\lambda_1 \lambda_2 - \lambda_1^0 \lambda_2^0)
\partial_x^2 q$ in $S$).

Let us infer our main stability result before discussing on how to solve these
issues. We recall the norm on the space $W^{2, 1} (\mathbb{R}) :$
\[ \| f \|_{W^{2, 1} (\mathbb{R})} \assign \sum_{j = 0}^2 \| f^{(j)} \|_{L^1
   (\mathbb{R})} \]
and the norm on the space $C^2 (\mathbb{R})$:
\[ \| f \|_{C^2 (\mathbb{R})} \assign \sum_{j = 0}^2 \| f^{(j)} \|_{L^{\infty}
   (\mathbb{R})} . \]
\begin{theorem}
  \label{NLconstlocTheorem}Consider the equation (\ref{maineq}) with $\rho_0
  \in] 0, 1 [$ and $s_{\tmop{cc}} (\rho_0) = h' (\rho_0) - U_f > 0$ and
  initial data $\rho_{| t = 0 \nobracket} = \rho_0 + \rho_i, u_{| t = 0
  \nobracket} = u (\rho_0) + u_i$ with $\rho_i, u_i \in C^2_{\tmop{loc}}
  (\mathbb{R}, \mathbb{R})$. Suppose furthermore that $\rho_i, u_i \in W^{2,
  1} (\mathbb{R}) \cap C^2 (\mathbb{R})$ and define
  \[ \varepsilon \assign \| \rho_i \|_{W^{2, 1} (\mathbb{R})} + \| u_i
     \|_{W^{2, 1} (\mathbb{R})} + \| \rho_i \|_{C^2 (\mathbb{R})} + \| u_i
     \|_{C^2 (\mathbb{R})} \]
  as well as
  \[ w_i = \sum_{j = 0}^2 | \rho_i^{(j)} | + | u_i^{(j)} | . \]
  Then, for any $\nu > 0$there exists $\varepsilon_0 > 0$ depending on
  $\rho_0, h' (\rho_0), U_f, \nu,$ such that if $\varepsilon \leqslant
  \varepsilon_0$, then the solution $(\rho_0 + \rho, U (\rho_0) + u)$ of
  (\ref{maineq}) is well defined in $C^2_{\tmop{loc}} (\mathbb{R}^+_t \times
  \mathbb{R}_x, \mathbb{R})$ and we have the following estimates. There exists
  $a, b > 0$ depending on $\rho_0, U_f, h,$ and $\delta > 0$ depending on
  $\rho_0, U_f, h, \varepsilon$ with $\delta = o_{\varepsilon \rightarrow 0}
  (1)$ such that, for any $k, n \in \mathbb{N}$ with $k + n \leqslant
  2,$defining
  \[ \lambda_1^0 = \rho_0 U_f > 0, \lambda_2^0 = \rho_0 (U_f - h' (\rho_0)) <
     0, \lambda_{\ast} = U (\rho_0) - \rho_0 U_f, \]
  \[ \gamma (0, 0) = \frac{1}{2}, \gamma (1, 0) = 1 - \nu, \gamma (0, 1) =
     \frac{3}{2} - \nu, \gamma (k, n) = 1 - \nu \]
  otherwise, there exists $K_{n, k} > 0$ depending on $n, k, \rho_0, U_f, h'
  (\rho_0), \nu$ such that
  \begin{eqnarray*}
    &  & | \partial_t^n \partial_x^k (u (x - \lambda_{\ast} t, t)) | + |
    \partial_t^n \partial_x^k (\rho (x - \lambda_{\ast} t, t)) |\\
    & \leqslant & \frac{K_{n, k}}{(1 + t)^{\gamma (k, n)}} \int_{(\lambda_2^0
    - \delta) t}^{(\lambda_1^0 + \delta) t} e^{- a \frac{y^2}{1 + t}} w_i (x -
    y) d y\\
    & + & K_{n, k} e^{- b t} \sup_{y \in [- \delta t, \delta t]} (w_i (x -
    \lambda_1^0 t + y) + w_i (x - \lambda_2^0 t + y))
  \end{eqnarray*}
  for all $x \in \mathbb{R}, t \geqslant 0$.
\end{theorem}

We start by introducing some quantities that we want to use to estimate $u$
and $\rho$. For $\mu_1 > 0, \mu_2 < 0, \gamma > 0, a > 0$ and $w \in
L^1_{\tmop{loc}} (\mathbb{R}, \mathbb{R}^+)$, we introduce the quantity
\[ F_{a, \gamma, w}^{\mu_1, \mu_2} (x, t) \assign \frac{1}{(1 + t)^{\gamma}}
   \int_{\mu_2 t}^{\mu_1 t} e^{- a \frac{y^2}{1 + t}} w (x - y) d y. \]
Remark that in Proposition \ref{constlinstab}, the functions $\partial_t^n
\partial_x^k (u (x - \lambda_{\ast} t, t))$ and $\partial_t^n \partial_x^k
(\rho (x - \lambda_{\ast} t, t))$ are estimated in part by $F_L \assign
F^{\lambda_1^0, \lambda_2^0}_{a_L, \frac{1}{2} + \frac{k}{2} + n, | u_i | + |
\rho_i |}$. The quantity $F_L$ can not estimate these functions in the
nonlinear case. For instance, it only ``see'' the initial data on
$[\lambda_2^0 t, \lambda_1^0 t]$, whereas the true caracteristic speeds
$\lambda_2, \lambda_1$ can for instance be such that $\lambda_2 < \lambda_2^0,
\lambda_1 > \lambda_1^0$ and thus the light cone is larger than $[\lambda_2^0
t, \lambda_1^0 t]$. It also does not take into account the fact that the speed
$\lambda_{\ast}$ is no longer constant. Despite all that, to make the estimate
work, we will only need to change the parameters of $F$. In Theorem
\ref{NLconstlocTheorem}, the quantity $F^{\lambda_1^0, \lambda_2^0}_{a_L,
\frac{1}{2} + \frac{k}{2} + n, | u_i | + | \rho_i |}$ is replaced by
$F^{\lambda_1^0 + \delta, \lambda_2^0 - \delta}_{a, \gamma (k, n), w_i}$ where
$\delta > 0$ is small, $a_L > a > 0,$ $\frac{1}{2} \leqslant \gamma (k, n)
\leqslant \frac{1}{2} + \frac{k}{2} + n$ and $w_i$ incode the initial value of
the perturbation. Taking $\lambda_1^0 + \delta, \lambda_2^0 - \delta$ for the
caracteristic allows us to have a bigger light cone than the real (non
straight) one, and having $a < a_L$ will allow us to absorb the error commited
by following the speed $\lambda_{\ast}$. There will also be a loss in decay in
time ($\gamma (k, n)$ will not be $\frac{1}{2} + \frac{k}{2} + n$ for all
values of $(k, n)$), on which we will come back later on.

The second quantity we introduce is, for $\mu \in \mathbb{R}, b > 0, \delta >
0, w \in L^1_{\tmop{loc}} (\mathbb{R}, \mathbb{R})$,
\[ G_{b, \delta, w}^{\mu} (x, t) \assign e^{- b t} \sup_{y \in [- \delta t,
   \delta t]} w (x - \mu t + y) . \]
Remark that the estimate of Proposition \ref{constlinstab} can be written as
\begin{eqnarray*}
  &  & | \partial_t^n \partial_x^k (u (x - \lambda_{\ast} t, t)) | + |
  \partial_t^n \partial_x^k (\rho (x - \lambda_{\ast} t, t)) |\\
  & \leqslant & K_{n, k} \left( F^{\lambda_1^0, \lambda_2^0}_{a_L,
  \frac{1}{2} + \frac{k}{2} + n, w} + G_{b_L, 0, w}^{\lambda_1^0} + G_{b_L, 0,
  w}^{\lambda_2^0} \right) (x, t)
\end{eqnarray*}
with
\[ w = \sum_{j = 0}^{n + k} | \rho_i^{(j)} | + | u_i^{(j)} | . \]
As for $F$, we will not be able to take $\delta = 0$ in $G_{b, \delta,
w}^{\mu}$ for the estimate in Theorem \ref{NLconstlocTheorem} for similar
reasons. We will also need to take $0 < b < b_L$. To simplify the notation, we
introduce finally
\[ F_{a, \gamma, \delta} \assign F^{\lambda_1^0 + \delta, \lambda_2^0 -
   \delta}_{a, \gamma, w_i} \]
and
\[ G_{b, \delta} \assign G_{b, \delta, w_i}^{\lambda_1^0} + G_{b, \delta,
   w_i}^{\lambda_2^0} . \]
The estimate of Theorem \ref{NLconstlocTheorem} can be written with these
notations as
\[ | \partial_t^n \partial_x^k (u (x - \lambda_{\ast} t, t)) | + |
   \partial_t^n \partial_x^k (\rho (x - \lambda_{\ast} t, t)) | \leqslant K
   (F_{a, \gamma, \delta} (x, t) + G_{b, \delta} (x, t)) \]
for $n + k \leqslant 2$.

\

Let us make a few remarks on Theorem \ref{NLconstlocTheorem}.
\begin{itemizedot}
  \item We keep informations on the localisation of the perturbation. The
  remarks done in the linear problem also apply here.
  
  \item In the case $k = n = 0$, since $w_i \in L^1 (\mathbb{R}) \cap
  L^{\infty} (\mathbb{R})$, we have that
  \[ | \rho (x, t) | + | u (x, t) | \leqslant \frac{K \varepsilon}{\sqrt{1 +
     t}} . \]
  This is the optimal decay, that was already shown in {\cite{MR1346371}}. For
  first derivatives, we have a small loss in decay with this method (we loose
  $t^{\nu}$ but $\nu$ can be made as small as we want when $\varepsilon
  \rightarrow 0$). For second order derivatives, we have a bigger loss. See
  the sketch of the proof in subsection \ref{ss15} for more details about why,
  but we expect to only get a $t^{\nu}$ loss if we ask for the smallness on
  more derivatives on the initial data.
  
  \item We can check easily that Theorem \ref{NLconstlocTheorem} implies that
  $\rho, u \in W^{2, 1} \cap C^2 (\mathbb{R})$ for all positive time and that
  for $n + k \leqslant 2$ we have
  \[ \| \partial_t^n \partial_x^k \rho (., t) \|_{L^{\infty}} + \|
     \partial_t^n \partial_x^k u (., t) \|_{L^{\infty}} \leqslant \frac{K
     \varepsilon}{(1 + t)^{\gamma (n, k)}} \]
  as well as
  \[ \| \partial_t^n \partial_x^k \rho (., t) \|_{L^1} + \| \partial_t^n
     \partial_x^k u (., t) \|_{L^1} \leqslant \frac{K \varepsilon}{(1 +
     t)^{\gamma (n, k) - 1 / 2}} . \]
  By interpolation we can deduce $L^p$ estimates on these quantities.
  
  \item On the initial perturbation, we only require for it to be small in
  $W^{2, 1} \cap C^2$, which is a larger space than the space of perturbations
  in {\cite{MR1346371}}. However, we do not have an equivalent when $t
  \rightarrow + \infty$, only a bound on the error. This allows us to sidestep
  the difficulty of finding the first order, which, for such a general initial
  perturbation, might be difficult.
  
  \item Remark that the estimates are given on $\rho (x - \lambda_{\ast} t,
  t)$ and $u (x - \lambda_{\ast} t, t)$, while the speed $\lambda_{\ast}$ is
  connected to the linear problem and not the nonlinear one. The ``true''
  speed of the perturbation should be $\lambda_{\ast} + \Omega_1$ for $\rho$,
  with $| \Omega_1 | \leqslant \frac{\varepsilon}{\sqrt{1 + t}}$. Such an
  error is absorbed by the form of the estimate we took. Indeed, we can show
  that
  \[ F_{a, \gamma, w}^{\mu_1, \mu_2} (x + \kappa (x, t), t) \leqslant K F_{a',
     \gamma, w}^{\mu_1 + \delta, \mu_2 - \delta} (x, t) \]
  for some $K, a' < a, \delta > 0$ if $| \kappa | \leqslant \varepsilon
  \sqrt{1 + t}$. We can check that $a > 0$ defined in Theorem
  \ref{NLconstlocTheorem} will be smaller than $a_L > 0$ from the linear case,
  to take into account this kind of error. In other words, we had to change
  the parameters in $F$ and $G$ to take into account the nonlinear
  characteristic speed, but also to absorb the error committed by the fact that
  $\lambda_{\ast}$ is not exactly the right speed of the error.
  
  \item We could be more specific in what derivatives are needed on the
  initial condition to estimate $\rho, u$ and its derivatives. That is,
  instead of considering $w_i$ that regroups all derivatives up to $2$ of the
  initial data, we could consider only a part of it in the case $k + n
  \leqslant 1$. We do not focus on this here, to simplify some part of the
  proof and some notations.
\end{itemizedot}

We will sketch the proof of Theorem \ref{NLconstlocTheorem} in the next
subsection. In this paper, we looked specifically at the Aw-Rascle-Zhang model
rather than the general case for several reason. The model in itself is
interesting and has been studied mathematically with different approach and
question. For instance, the limit $\tau \rightarrow 0$ has been studied in
{\cite{MR3901913}}. This also allow us to simplify some notations without
losing too much generality. For other hyperbolic systems, Theorem
\ref{NLconstlocTheorem} should hold, with some technical conditions. For
instance, it will be important that $\lambda_1, \lambda_2$, the two nonlinear
characteristic speeds, can depend on $\rho, u$ but not on their derivatives.

\subsection{Plan of the proofs}\label{ss15}

We start by solving, in section \ref{sdampedwave}, the equation
\begin{equation}
  \partial_t^2 f + (\lambda_1 + \lambda_2) \partial_{x t}^2 f + \lambda_1
  \lambda_2 \partial_x^2 f + \delta \partial_t f = S, \label{150}
\end{equation}
see Proposition \ref{dwesol} and Lemma \ref{Vest}. The solution of
$\partial_t^2 u - \partial_x^2 u + \partial_t u = S$ is well known, and is
connected to (\ref{150}) by a change of variables. The estimates shown in
Lemma \ref{Vest} are a consequence of usual computations on Bessel functions.

Section \ref{sect3vf} is devoted to estimates on $F_{a, \gamma, w}^{\mu_1,
\mu_2}$ and $G_{b, \delta, w}^{\mu}$, in particular, their interaction with
the kernel of the damped wave equation. The contribution of the source $S$ in
the damped wave equation is
\begin{equation}
  \int_0^t \int_{\lambda_2^0 (t - s)}^{\lambda_1^0 (t - s)} V (y, t - s) S (x
  - y, s) d y d s, \label{160}
\end{equation}
and we compute this quantity and its derivatives with respect to $x$ and $t$
in the case $| S | \leqslant \alpha F_{a, \gamma, w}^{\lambda_1^0 + \delta,
\lambda^0_2 - \delta} + \beta G_{b, \delta, w}^{\mu}$ for $\alpha, \beta > 0,
\mu \in \{ \lambda_1^0, \lambda^0_2 \}$. The key point is that for such a
value of $S$, the quantity (\ref{160}) and its derivatives can be estimates by
terms that are also of the form $\alpha' F_{a, \gamma, w}^{\lambda^0_1 +
\delta, \lambda^0_2 - \delta} + \beta' G_{b, \delta, w}^{\mu}$ with $\alpha',
\beta' > 0$.

\

Then, in section \ref{ss33v4}, we compute a change of variable to deal with
the fact that the caracteristic speeds $\lambda_1$ and $\lambda_2$ are not
constant. It will be used to estimate second derivatives of the perturbation.
However, this will not be used for the estimate of the first and non
derivative, and most likely it is not possible to do so. Since $\lambda_1$ and
$\lambda_2$ remain close to constants $\lambda_1^0, \lambda_2^0$, we will do a
change of coordinate close to the identity, both in time and space, to change
the caracteristic speeds to $\lambda_1^0, \lambda_2^0$.

\

Section \ref{sec5vf} is devoted to the proof of Proposition
\ref{constlinstab}. This is done using Proposition \ref{dwesol} and Lemma
\ref{Vest}.

\

Finally, section \ref{sec6vf} is devoted to the proof of Theorem
\ref{NLconstlocTheorem}. We treat the nonlinear term perturbatively, using
Proposition \ref{dwesol}. We show the estimates by a bootstrap. The result
hold for small time by standard hyperbolic theory. Then, using the results in
section \ref{sect3vf}, we show that they hold for all times for $\rho, u$ and
their first derivatives.

For the second derivative, we use section \ref{ss33v4} to write the equation
on a similar form, but this time the terms in $S$ do not contain second
derivatives of $\rho$ and $u$. Otherwise, we will have a loss of derivatives.
Still using section \ref{sect3vf} we compute that the estimates hold for these
new variables. Then, we show that we can get these estimates back in the
original coordinates with the quantities $F_{a, \gamma, w}^{\mu_1, \mu_2}$ and
$G_{b, \delta, w}^{\mu}$, up to reducing the value of $a$ and $b$ a little.

\

\begin{acknowledgments*}
  The authors are supported by Tamkeen under the NYU Abu Dhabi Research
  Institute grant CG002.
\end{acknowledgments*}

\section{Solution to the general damped wave equation}\label{sdampedwave}

This section is devoted to the proof of Proposition \ref{dwesol} and Lemma
\ref{Vest}. We suppose here that $\lambda_1 > 0, \lambda_2 < 0, \delta > 0$.

\subsection{Proof of Proposition \ref{dwesol}}

We recall the equation we are trying to solve, which is
\[ \left\{\begin{array}{l}
     \partial_t^2 f + (\lambda_1 + \lambda_2) \partial_{x t}^2 f + \lambda_1
     \lambda_2 \partial_x^2 f + \delta \partial_t f = S (x, t)\\
     f_{| t = 0 \nobracket} = f_0\\
     \partial_t f_{| t = 0 \nobracket} = f_1 .
   \end{array}\right. \]

\subsubsection{Solution to the classical damped wave equation}

We recall the definition of $I_0$, the modified Bessel function of the first
kind and of order $0$:
\[ I_0 (y) \assign \sum_{n = 0}^{+ \infty} \frac{1}{n! \Gamma (n + 1)} \left(
   \frac{y}{2} \right)^{2 n} . \]
We will recall some of its properties in subsection \ref{sssBessel}. It
appears in the solution of the classical damped wave equation on $\mathbb{R}$:
\begin{equation}
  \partial_t^2 u - \partial_x^2 u + \partial_t u = S, u_{| t = 0 \nobracket} =
  u_0, \partial_t u_{| t = 0 \nobracket} = u_1 . \label{cdwe}
\end{equation}
\begin{theorem}[{\cite{courant2008methods}}]
  \label{CH}Consider the function
  \begin{equation}
    J (t, y) \assign \frac{e^{- t / 2}}{2} I_0 \left( \frac{1}{2} \sqrt{t^2 -
    y^2} \right) .
  \end{equation}
  The solution to
  \[ \partial_t^2 u - \partial_x^2 u + \partial_t u = S, u_{| t = 0
     \nobracket} = u_0, \partial_t u_{| t = 0 \nobracket} = u_1 \]
  is
  \begin{eqnarray*}
    u (x, t) & = & \int_{- t}^t J (t, y) u_1 (x - y) d y\\
    & + & \int_{- t}^t J (t, y) u_0 (x - y) d y + \partial_t \left( \int_{-
    t}^t J (t, y) u_0 (x - y) d y \right)\\
    & + & \int_0^t \left( \int_{| y | \leqslant t - s} J (t - s, y) S (x - y,
    s) d y \right) d s.
  \end{eqnarray*}
\end{theorem}

We recall briefly the proof of Theorem \ref{CH}. We first check that
\[ \partial_t^2 J + \partial_t J - \partial_z^2 J = 0. \]
Then, we can solve the problem if $u_0 = S = 0$, simply by computing
\[ (\partial_t^2 - \partial_x^2 + \partial_t) \left( \int_{- t}^t J (t, y) u_1
   (x - y) d y \right) = 0. \]
and checking that at $t = 0$, we have the right initial condition. Then, using
this function, we construct the parts coming fom $u_0$ and $S$, and we add
them together by linearity.

\subsubsection{From the classical to the general damped wave equation}

Our goal here is to find the function $V$ that satisfies for the general
damped wave equation (\ref{dwe}) the same role that $J$ satisfies for the
classical damped wave equation (\ref{cdwe}). We define $a > 0, b \in
\mathbb{R}, c > 0$ and the function $V$ by
\[ J (y, t) = V (a y, b y + c t) . \]
We recall that $\partial_t^2 J + \partial_t J - \partial_y^2 J = 0$, and we
compute that
\[ \partial_t J = c \partial_t V, \partial_t^2 J = c^2 \partial_t^2 V \]
\[ \partial_y J = a \partial_y V + b \partial_t V, \partial_y^2 J = a^2
   \partial_y^2 V + 2 a b \partial_{y t}^2 V + b^2 \partial_t^2 V \]
hence
\[ \partial_t^2 V - \frac{2 a b}{c^2 - b^2} \partial_{y t}^2 V -
   \frac{a^2}{c^2 - b^2} \partial_y^2 V + \frac{c}{c^2 - b^2} \partial_t V =
   0. \]
We then look for solutions of
\[ \frac{- 2 a b}{c^2 - b^2} = \beta (= \lambda_1 + \lambda_2), \frac{a^2}{c^2
   - b^2} = \alpha (= - \lambda_1 \lambda_2 > 0), \frac{c}{c^2 - b^2} = \delta
   > 0. \]
We find
\[ a = \frac{1}{\delta} \sqrt{\alpha + \frac{\beta^2}{4}}, b = \frac{-
   \beta}{2 \alpha \delta} \sqrt{\alpha + \frac{\beta^2}{4}}, c =
   \frac{1}{\alpha \delta} \left( \alpha + \frac{\beta^2}{4} \right) . \]
We deduce that the function
\[ V (y, t) = \frac{1}{\lambda_1 - \lambda_2} J \left( \frac{y}{a}, \frac{a t
   - b y}{a c} \right) \]
satisfies the equation
\[ \partial_t^2 V + (\lambda_1 + \lambda_2) \partial_{y t}^2 V + \lambda_1
   \lambda_2 \partial_y^2 V + \delta \partial_t V = 0. \]
The factor $\frac{1}{\lambda_1 - \lambda_2}$ is here to have a cleaner formula
in Proposition \ref{dwesol}. We compute
\[ \frac{a t - b z}{a c} = \frac{\delta}{\alpha + \frac{\beta^2}{4}} \left(
   \alpha t + \frac{\beta}{2} z \right) \]
and therefore
\[ V (y, t) = \frac{e^{- \frac{\delta \left( \alpha t + \frac{\beta}{2} y
   \right)}{2 \left( \alpha + \frac{\beta^2}{4} \right)}}}{\lambda_1 -
   \lambda_2} I_0 \left( \frac{\delta}{2} \sqrt{\left( \frac{1}{\alpha +
   \frac{\beta^2}{4}} \left( \alpha t + \frac{\beta}{2} y \right) \right)^2 -
   \left( \frac{y}{\sqrt{\alpha + \frac{\beta^2}{4}}} \right)^2} \right) . \]
After some simplifications, using $\alpha = - \lambda_1 \lambda_2, \beta =
\lambda_1 + \lambda_2, \alpha + \frac{\beta^2}{4} = \frac{1}{4} (\lambda_1 -
\lambda_2)^2$, we find that
\[ V (y, t) = \frac{e^{- \frac{2 \delta}{(\lambda_1 - \lambda_2)^2} \left( -
   \lambda_1 \lambda_2 t + \frac{\lambda_1 + \lambda_2}{2} y
   \right)}}{\lambda_1 - \lambda_2} I_0 \left( \frac{2 \delta \sqrt{-
   \lambda_1 \lambda_2}}{(\lambda_1 - \lambda_2)^2} \sqrt{- (y - \lambda_1 t)
   (y - \lambda_2 t)} \right) . \]
Let us summarize our computations so far, and give as well some specific
values of $V$.

\begin{lemma}
  \label{L24}The function
  \[ V (y, t) = \frac{e^{- \frac{2 \delta}{(\lambda_1 - \lambda_2)^2} \left( -
     \lambda_1 \lambda_2 t + \frac{\lambda_1 + \lambda_2}{2} y
     \right)}}{\lambda_1 - \lambda_2} I_0 \left( \frac{2 \delta \sqrt{-
     \lambda_1 \lambda_2}}{(\lambda_1 - \lambda_2)^2} \sqrt{- (y - \lambda_1
     t) (y - \lambda_2 t)} \right) \]
  satisfies
  \[ \partial_t^2 V + (\lambda_1 + \lambda_2) \partial_{y t}^2 V + \lambda_1
     \lambda_2 \partial_y^2 V + \delta \partial_t V = 0. \]
  We also compute
  \[ V (\lambda_1 t, t) = \frac{e^{\frac{- \delta \lambda_1 t}{\lambda_1 -
     \lambda_2}}}{\lambda_1 - \lambda_2}, V (\lambda_2 t, t) =
     \frac{e^{\frac{\delta \lambda_2 t}{\lambda_1 - \lambda_2}}}{\lambda_1 -
     \lambda_2} \]
  and
  \[ \partial_t V (\lambda_1 t, t) = \frac{2 \delta \lambda_1
     \lambda_2}{(\lambda_1 - \lambda_2)^2} \left( 1 - \frac{I_0'' (0) \delta
     \lambda_1 t}{\lambda_1 - \lambda_2} \right) V (\lambda_1 t, t), \]
  \[ \partial_t V (\lambda_2 t, t) = \frac{2 \delta \lambda_1
     \lambda_2}{(\lambda_1 - \lambda_2)^2} \left( 1 + \frac{I_0'' (0) \delta
     \lambda_2 t}{\lambda_1 - \lambda_2} \right) V (\lambda_1 t, t) \]
  as well as
  \[ \partial_y V (\lambda_1 t, t) = \left( \frac{- \delta (\lambda_1 +
     \lambda_2)}{(\lambda_1 - \lambda_2)^2} + \frac{2 \delta^2 \lambda_1
     \lambda_2 t I_0'' (0)}{(\lambda_1 - \lambda_2)^3} \right) V (\lambda_1 t,
     t), \]
  \[ \partial_y V (\lambda_2 t, t) = \left( \frac{- \delta (\lambda_1 +
     \lambda_2)}{(\lambda_1 - \lambda_2)^2} - \frac{2 \delta^2 \lambda_1
     \lambda_2 t I_0'' (0)}{(\lambda_1 - \lambda_2)^3} \right) V (\lambda_2 t,
     t) . \]
\end{lemma}

Remark that $\lambda_1 - \lambda_2 > 0$ and $\frac{\delta \lambda_2}{\lambda_1
- \lambda_2} < 0$.

\begin{proof}
  We complete the above computations with the following particular values of
  $V$. We define
  \[ \gamma_1 (y, t) \assign \frac{- 2 \delta}{(\lambda_1 - \lambda_2)^2}
     \left( - \lambda_1 \lambda_2 t + \frac{\lambda_1 + \lambda_2}{2} y
     \right) \]
  and
  \[ \gamma_2 (y, t) \assign \frac{2 \delta \sqrt{- \lambda_1
     \lambda_2}}{(\lambda_1 - \lambda_2)^2} \sqrt{- (y - \lambda_1 t) (y -
     \lambda_2 t)} \]
  so that
  \[ V (y, t) = \frac{e^{\gamma_1 (y, t)}}{\lambda_1 - \lambda_2} I_0
     (\gamma_2 (y, t)) . \]
  We compute
  \[ \partial_t \gamma_1 = \frac{2 \delta \lambda_1 \lambda_2}{(\lambda_1 -
     \lambda_2)^2}, \partial_y \gamma_1 = \frac{- \delta (\lambda_1 +
     \lambda_2)}{(\lambda_1 - \lambda_2)^2} \]
  and
  \[ \partial_t \gamma_2 = \frac{(\lambda_1 + \lambda_2) y - 2 \lambda_1
     \lambda_2 t}{2 \gamma_2} \left( \frac{- 4 \delta^2 \lambda_1
     \lambda_2}{(\lambda_1 - \lambda_2)^4} \right) \]
  as well as
  \[ \partial_y \gamma_2 = \frac{- 2 y + (\lambda_1 + \lambda_2) t}{2
     \gamma_2} \left( \frac{- 4 \delta^2 \lambda_1 \lambda_2}{(\lambda_1 -
     \lambda_2)^4} \right) . \]
  We compute
  \[ \gamma_2 (\lambda_1 t, t) = \gamma_2 (\lambda_2 t, t) = 0, \]
  \[ \gamma_1 (\lambda_1 t, t) = \frac{- \delta \lambda_1 t}{\lambda_1 -
     \lambda_2} \leqslant 0, \gamma_1 (\lambda_2 t, t) = \frac{\delta
     \lambda_2 t}{\lambda_1 - \lambda_2} \leqslant 0, \]
  therefore
  \[ V (\lambda_1 t, t) = \frac{e^{\gamma_1 (\lambda_1 t, t)}}{\lambda_1 -
     \lambda_2} I_0 (\gamma_2 (\lambda_1 t, t)) = \frac{e^{\frac{- \delta
     \lambda_1 t}{\lambda_1 - \lambda_2}}}{\lambda_1 - \lambda_2} \]
  and
  \[ V (\lambda_2 t, t) = \frac{e^{\gamma_1 (\lambda_2 t, t)}}{\lambda_1 -
     \lambda_2} I_0 (\gamma_2 (\lambda_2 t, t)) = \frac{e^{\frac{\delta
     \lambda_2 t}{\lambda_1 - \lambda_2}}}{\lambda_1 - \lambda_2} . \]
  Furthermore, since
  \[ V (y, t) = \frac{e^{\gamma_1 (y, t)}}{\lambda_1 - \lambda_2} I_0
     (\gamma_2 (y, t)), \]
  we have
  \begin{eqnarray*}
    \partial_t V (y, t) & = & \partial_t \gamma_1 (y, t) V (y, t) +
    \frac{e^{\gamma_1 (y, t)}}{\lambda_1 - \lambda_2} \partial_t \gamma_2 (y,
    t) I_0' (\gamma_2 (y, t)),
  \end{eqnarray*}
  and since $\gamma_2 (\lambda_1 t, t) = 0$, we compute that
  \begin{eqnarray*}
    \partial_t V (\lambda_1 t, t) & = & \partial_t \gamma_1 (\lambda_1 t, t) V
    (\lambda_1 t, t) + \frac{e^{\gamma_1 (\lambda_1 t, t)}}{2 (\lambda_1 -
    \lambda_2)} (\partial_t (\gamma_2 (y, t)^2))_{| y = \lambda_1 t
    \nobracket} I_0'' (0)\\
    & = & \frac{2 \delta \lambda_1 \lambda_2}{(\lambda_1 - \lambda_2)^2} V
    (\lambda_1 t, t) + \frac{I_0'' (0)}{2} \left( \frac{- 4 \delta^2
    \lambda_1^2 \lambda_2 t}{(\lambda_1 - \lambda_2)^3} \right) V (\lambda_1
    t, t)\\
    & = & \frac{2 \delta \lambda_1 \lambda_2}{(\lambda_1 - \lambda_2)^2}
    \left( 1 - \frac{I_0'' (0) \delta \lambda_1 t}{\lambda_1 - \lambda_2}
    \right) V (\lambda_1 t, t)
  \end{eqnarray*}
  as well as
  \begin{eqnarray*}
    \partial_t V (\lambda_2 t, t) & = & \partial_t \gamma_1 (\lambda_2 t, t) V
    (\lambda_2 t, t) + \frac{e^{\gamma_1 (\lambda_1 t, t)}}{2 (\lambda_1 -
    \lambda_2)} (\partial_t (\gamma_2 (y, t)^2))_{| y = \lambda_2 t
    \nobracket} I_0'' (0)\\
    & = & \frac{2 \delta \lambda_1 \lambda_2}{(\lambda_1 - \lambda_2)^2} V
    (\lambda_1 t, t) + \frac{I_0'' (0)}{2} \left( \frac{4 \delta^2 \lambda_1
    \lambda^2_2 t}{(\lambda_1 - \lambda_2)^3} \right) V (\lambda_1 t, t)\\
    & = & \frac{2 \delta \lambda_1 \lambda_2}{(\lambda_1 - \lambda_2)^2}
    \left( 1 + \frac{I_0'' (0) \delta \lambda_2 t}{\lambda_1 - \lambda_2}
    \right) V (\lambda_1 t, t) .
  \end{eqnarray*}
  Similarly,
  \[ \partial_y V (\lambda_1 t, t) = \left( \partial_y \gamma_1 + \frac{1}{2}
     \partial_y (\gamma_2^2) I_0'' (0) \right)_{| y = \lambda_1 t \nobracket}
     V (\lambda_1 t, t), \]
  hence
  \[ \partial_y V (\lambda_1 t, t) = \left( \frac{- \delta (\lambda_1 +
     \lambda_2)}{(\lambda_1 - \lambda_2)^2} + \frac{2 \delta^2 \lambda_1
     \lambda_2 t I_0'' (0)}{(\lambda_1 - \lambda_2)^3} \right) V (\lambda_1 t,
     t) \]
  and we also check that
  \[ \partial_y V (\lambda_2 t, t) = \left( \frac{- \delta (\lambda_1 +
     \lambda_2)}{(\lambda_1 - \lambda_2)^2} - \frac{2 \delta^2 \lambda_1
     \lambda_2 t I_0'' (0)}{(\lambda_1 - \lambda_2)^3} \right) V (\lambda_2 t,
     t) . \]
\end{proof}

We are now equipped to compute the solution of the general damped wave
equation (\ref{dwe}).

\begin{lemma}
  \label{L26}Given $u_1 \in C^2_{\tmop{loc}} (\mathbb{R})$, the function
  \[ v (x, t) = \int_{\lambda_2 t}^{\lambda_1 t} V (y, t) u_1 (x - y) d y \]
  satisfies $v (x, 0) = 0, \partial_t v (x, 0) = u_1$ and the equation
  \[ \partial_t^2 v + (\lambda_1 + \lambda_2) \partial_{y t}^2 v + \lambda_1
     \lambda_2 \partial_y^2 v + \delta \partial_t v = 0. \]
\end{lemma}

\begin{proof}
  We compute
  \begin{eqnarray*}
    \partial_t v & = & \lambda_1 V (\lambda_1 t, t) u_1 (x - \lambda_1 t) -
    \lambda_2 V (\lambda_2 t, t) u_1 (x - \lambda_2 t)\\
    & + & \int_{\lambda_2 t}^{\lambda_1 t} \partial_t V (y, t) u_1 (x - y) d
    y
  \end{eqnarray*}
  and
  \begin{eqnarray*}
    \partial^2_t v & = & \lambda_1 (\partial_t (V (\lambda_1 t, t)) +
    \partial_t V (\lambda_1 t, t)) u_1 (x - \lambda_1 t)\\
    & - & \lambda_2 (\partial_t (V (\lambda_2 t, t)) + \partial_t V
    (\lambda_2 t, t)) u_1 (x - \lambda_2 t)\\
    & - & \lambda_1^2 V (\lambda_1 t, t) u_1' (x - \lambda_1 t) + \lambda_2^2
    V (\lambda_2 t, t) u_1' (x - \lambda_2 t)\\
    & + & \int_{\lambda_2 t}^{\lambda_1 t} \partial_t^2 V (y, t) u_1 (x - y)
    d y.
  \end{eqnarray*}
  We also have
  \begin{eqnarray*}
    \partial_{y t}^2 v & = & \lambda_1 V (\lambda_1 t, t) u'_1 (x - \lambda_1
    t) - \lambda_2 V (\lambda_2 t, t) u'_1 (x - \lambda_2 t)\\
    & - & \partial_t V (\lambda_1 t, t) u_1 (x - \lambda_1 t) + \partial_t V
    (\lambda_2 t, t) u_1 (x - \lambda_2 t)\\
    & + & \int_{\lambda_2 t}^{\lambda_1 t} \partial_{t y}^2 V (y, t) u_1 (x -
    y) d y
  \end{eqnarray*}
  and
  \begin{eqnarray*}
    \partial_y^2 v & = & - (V (\lambda_1 t, t) u_1' (x - \lambda_1 t) +
    \partial_z V (\lambda_1 t, t) u_1 (x - \lambda_1 t))\\
    & + & V (\lambda_2 t, t) u_1' (x - \lambda_2 t) + \partial_z V (\lambda_2
    t, t) u_1 (x - \lambda_2 t)\\
    & + & \int_{\lambda_2 t}^{\lambda_1 t} \partial_z^2 V (y, t) u_1 (x - y)
    d y.
  \end{eqnarray*}
  We deduce that
  \begin{eqnarray*}
    &  & \partial_t^2 v + (\lambda_1 + \lambda_2) \partial_{y t}^2 v +
    \lambda_1 \lambda_2 \partial_y^2 v + \delta \partial_t v\\
    & = & u_1 (x - \lambda_1 t) \left( \lambda_1 (\partial_t (V (\lambda_1 t,
    t)) + \partial_t V (\lambda_1 t, t)) - \left( \lambda_1 + \lambda_2
    \right) \partial_t V (\lambda_1 t, t) - \lambda_1 \lambda_2 \partial_z V
    (\lambda_1 t, t) + \delta \lambda_1 V (\lambda_1 t, t) \right)\\
    & + & u_1 (x - \lambda_2 t) \left( - \lambda_2 (\partial_t (V (\lambda_2
    t, t)) + \partial_t V (\lambda_2 t, t)) + \left( \lambda_1 + \lambda_2
    \right) \partial_t V (\lambda_2 t, t) + \lambda_1 \lambda_2 \partial_z V
    (\lambda_2 t, t) - \delta \lambda_2 V (\lambda_2 t, t) \right)\\
    & + & u_1' (x - \lambda_1 t) (- \lambda_1^2 V (\lambda_1 t, t) +
    (\lambda_1 + \lambda_2) \lambda_1 V (\lambda_1 t, t) - \lambda_1 \lambda_2
    V (\lambda_1 t, t))\\
    & + & u_1' (x - \lambda_2 t) (\lambda_2^2 V (\lambda_2 t, t) - (\lambda_1
    + \lambda_2) \lambda_2 V (\lambda_2 t, t) + \lambda_1 \lambda_2 V
    (\lambda_2 t, t)) .
  \end{eqnarray*}
  We check easily with Lemma \ref{L24} that
  \[ - \lambda_1^2 V (\lambda_1 t, t) + (\lambda_1 + \lambda_2) \lambda_1 V
     (\lambda_1 t, t) - \lambda_1 \lambda_2 V (\lambda_1 t, t) = 0 \]
  and
  \[ \lambda_2^2 V (\lambda_2 t, t) - (\lambda_1 + \lambda_2) \lambda_2 V
     (\lambda_2 t, t) + \lambda_1 \lambda_2 V (\lambda_2 t, t) = 0, \]
  and similarly for the other boundaries terms.
\end{proof}

We can now solve the general damped wave equation with a source and any
initial condition.

\begin{proof}[Of Proposition \ref{dwesol}]
  We recall that equation (\ref{dwe}) is linear. We have prove it for the
  contribution in $f_1$ in Lemma \ref{L26}. Now consider $v$ the solution to
  \[ \left\{\begin{array}{l}
       \partial_t^2 v + (\lambda_1 + \lambda_2) \partial_{x t}^2 v + \lambda_1
       \lambda_2 \partial_x^2 v + \delta \partial_t v = 0\\
       v_{| t = 0 \nobracket} = 0\\
       \partial_t v_{| t = 0 \nobracket} = f_0,
     \end{array}\right. \]
  and define $w = \delta v + \partial_t v + (\lambda_1 + \lambda_2) \partial_x
  v$. Then, $w_{| t = 0 \nobracket} = \partial_t v_{| t = 0 \nobracket} = f_0$
  and $\partial_t w = - \lambda_1 \lambda_2 \partial_x^2 v$ hence $\partial_t
  w_{| t = 0 \nobracket} = 0$. We deduce that $w$ solves the problem
  \[ \left\{\begin{array}{l}
       \partial_t^2 w + (\lambda_1 + \lambda_2) \partial_{x t}^2 w + \lambda_1
       \lambda_2 \partial_x^2 w + \delta \partial_t w = 0\\
       w_{| t = 0 \nobracket} = f_0\\
       \partial_t w_{| t = 0 \nobracket} = 0.
     \end{array}\right. \]
  Finally, we define
  \[ F (t, x, s) = \int_{\lambda_2 t}^{\lambda_1 t} V (z, t) S (x - z, s) d z
  \]
  and
  \[ u (x, t) = \int_0^t F (t - s, x, s) d s. \]
  We check easily that $u (x, 0) = \partial_t u (x, 0) = 0$ and since
  \[ F (0, x, t) = \partial_s F (0, x, t) = \partial_x F (0, x, t) = 0 \]
  as well as
  \[ \partial_t F (0, x, t) = \lambda_1 V (0, 0) S (x, t) - \lambda_2 V (0, 0)
     S (x, t) = (\lambda_1 - \lambda_2) V (0, 0) S (x, t) = S (x, t) \]
  we check easily that
  \[ \partial_t^2 u + (\lambda_1 + \lambda_2) \partial_{x t}^2 u + \lambda_1
     \lambda_2 \partial_x^2 u + \delta \partial_t u = S (x, t) . \]
  We conclude by linearity of the equation, taking for $f$ the sum of the
  functions $u, w$ and the solution constructed in Lemma \ref{L26}.
\end{proof}

\subsection{Properties of the function $V$}

We are interested here in estimates on the function
\[ V (y, t) = \frac{e^{- \frac{2 \delta}{(\lambda_1 - \lambda_2)^2} \left( -
   \lambda_1 \lambda_2 t + \frac{\lambda_1 + \lambda_2}{2} y
   \right)}}{\lambda_1 - \lambda_2} I_0 \left( \frac{2 \delta \sqrt{-
   \lambda_1 \lambda_2}}{(\lambda_1 - \lambda_2)^2} \sqrt{- (y - \lambda_1 t)
   (y - \lambda_2 t)} \right), \]
with $\delta > 0, \lambda_1 > 0$ and $\lambda_2 < 0$.

\subsubsection{Properties of the Bessel function $I_0$}\label{sssBessel}

Let us also recall here some properties of $I_0$:

\begin{lemma}[{\cite{MR1349110}}]
  \label{L24vf}The function $I_0 \in C^{\infty} (\mathbb{R}^+, \mathbb{R})$
  satisfies the following properties:
  \begin{itemizedot}
    \item $I_0 (0) = 1, I_0' \geqslant 0$
    
    \item $I_0 (y) = \sum_{n = 0}^{+ \infty} \frac{1}{n! \Gamma (n + 1)}
    \left( \frac{y}{2} \right)^{2 n}$
    
    \item $I_0 (x) = \frac{e^x}{\sqrt{2 \pi x}} \left( 1 + O_{x \rightarrow
    \infty} \left( \frac{1}{x} \right) \right)$.
  \end{itemizedot}
\end{lemma}

\subsubsection{Estimates on the function $V$}

\begin{lemma}
  \label{Vest1}For $\delta > 0, \lambda_1 > 0, \lambda_2 < 0$, there exists
  $C_{0, 0} (\lambda_1, \lambda_2, \delta), a_0 (\lambda_1, \lambda_2, \delta)
  > 0$ such that the function
  \[ V (y, t) = \frac{e^{- \frac{2 \delta}{(\lambda_1 - \lambda_2)^2} \left( -
     \lambda_1 \lambda_2 t + \frac{\lambda_1 + \lambda_2}{2} y
     \right)}}{\lambda_1 - \lambda_2} I_0 \left( \frac{2 \delta \sqrt{-
     \lambda_1 \lambda_2}}{(\lambda_1 - \lambda_2)^2} \sqrt{- (y - \lambda_1
     t) (y - \lambda_2 t)} \right) \]
  satisfied for $t \geqslant 0, y \in [\lambda_2 t, \lambda_1 t]$ that
  \[ | V (y, t) | \leqslant \frac{C_{0, 0} (\lambda_1, \lambda_2, \delta)}{(1
     + t)^{\frac{1}{2}}} e^{- a_0 \frac{y^2}{(1 + t)}} . \]
\end{lemma}

\begin{proof}
  First we check easily the estimate if \ $0 \leqslant t \leqslant 1, y \in
  [\lambda_2 t, \lambda_1 t]$ (which is simply the fact that $V$ is bounded if
  $t$ and $y$ are bounded). We now suppose that $t \geqslant 1$. We define as
  in the proof of Lemma \ref{L24} the quantities
  \[ \gamma_1 (y, t) = \frac{- 2 \delta}{(\lambda_1 - \lambda_2)^2} \left( -
     \lambda_1 \lambda_2 t + \frac{\lambda_1 + \lambda_2}{2} y \right) \]
  and
  \[ \gamma_2 (y, t) = \frac{2 \delta \sqrt{- \lambda_1 \lambda_2}}{(\lambda_1
     - \lambda_2)^2} \sqrt{- (y - \lambda_1 t) (y - \lambda_2 t)}, \]
  so that $V (y, t) = \frac{e^{\gamma_1 (y, t)} I_0 (\gamma_2 (y,
  t))}{\lambda_1 - \lambda_2}$. We also write
  \[ \gamma_2 (y, t) = t G_0 \left( \frac{y}{t} \right) \]
  with
  \[ G_0 (z) \assign \frac{2 \delta \sqrt{- \lambda_1 \lambda_2}}{(\lambda_1 -
     \lambda_2)^2} \sqrt{- (z - \lambda_1) (z - \lambda_2)} . \]
  We also define
  \[ H (x) \assign e^{- x} I_0 (x) \]
  which is a smooth decreasing function on $\mathbb{R}^+$ that satisfies,
  using Lemma \ref{L24vf}, that
  \[ H (0) = 1, H (x) \sim \frac{1}{\sqrt{2 \pi x}} \]
  when $x \rightarrow + \infty$. Remark that $0 \leqslant G_0 (x) \leqslant K
  (\lambda_1, \lambda_2, \delta)$ for $x \in [\lambda_2, \lambda_1]$, and that
  for $x \in \left[ \frac{\lambda_2}{2}, \frac{\lambda_1}{2} \right]$ we have
  \[ G_0 (x) \geqslant C_1 (\lambda_1 \comma \lambda_2, \delta) > 0, \]
  as well as
  \[ V (y, t) = \frac{e^{\gamma_1 (y, t) + \gamma_2 (y, t)}}{\lambda_1 -
     \lambda_2} H \left( t G_0 \left( \frac{y}{t} \right) \right) . \]
  We compute that
  \[ \gamma_1 (y, t) + \gamma_2 (y, t) = \frac{- \delta}{(\lambda_1 -
     \lambda_2)^2} \frac{y^2}{t} - t F_0 \left( \frac{y}{t} \right) \]
  where
  \[ F_0 (x) \assign \frac{2 \delta \lambda_1 \lambda_2}{(\lambda_1 -
     \lambda_2)^2} \left( - 1 + \frac{1}{2} \left( \frac{\lambda_1 +
     \lambda_2}{\lambda_1 \lambda_2} x - \frac{x^2}{\lambda_1 \lambda_2}
     \right) + \sqrt{1 - \frac{(\lambda_1 + \lambda_2)}{\lambda_1 \lambda_2} x
     + \frac{x^2}{\lambda_1 \lambda_2}} \right) . \]
  Indeed, we have
  \begin{eqnarray*}
    &  & \gamma_1 (y, t) + \gamma_2 (y, t)\\
    & = & \frac{2 \delta}{(\lambda_1 - \lambda_2)^2} \left( \lambda_1
    \lambda_2 t - \frac{\lambda_1 + \lambda_2}{2} y + \sqrt{\lambda_1
    \lambda_2 (y - \lambda_1 t) (y - \lambda_2 t)} \right)\\
    & = & \frac{2 \delta t}{(\lambda_1 - \lambda_2)^2} \left( \lambda_1
    \lambda_2 - \frac{\lambda_1 + \lambda_2}{2} \frac{y}{t} + \sqrt{(\lambda_1
    \lambda_2)^2 - \frac{y}{t} (\lambda_1 \lambda_2^2 + \lambda_1^2 \lambda_2)
    + \lambda_1 \lambda_2 \frac{y^2}{t^2}} \right)\\
    & = & \frac{2 \delta \lambda_1 \lambda_2 t}{(\lambda_1 - \lambda_2)^2}
    \left( 1 - \frac{1}{2} \left( \frac{1}{\lambda_1} + \frac{1}{\lambda_2}
    \right) \frac{y}{t} - \sqrt{1 - \frac{y}{t} \left( \frac{1}{\lambda_1} +
    \frac{1}{\lambda_2} \right) + \frac{1}{\lambda_1 \lambda_2}
    \frac{y^2}{t^2}} \right)\\
    & = & - t F_0 \left( \frac{y}{t} \right) + \frac{- \delta}{(\lambda_1 -
    \lambda_2)^2} \frac{y^2}{t} .
  \end{eqnarray*}
  Remark that with $z = \frac{\lambda_1 + \lambda_2}{\lambda_1 \lambda_2} x -
  \frac{x^2}{\lambda_1 \lambda_2}$ we have $F_0 (x) = \frac{2 \delta \lambda_1
  \lambda_2}{(\lambda_1 - \lambda_2)^2} \left( - 1 + \frac{1}{2} z + \sqrt{1 -
  z} \right)$. This implies that we have $F_0 (x) \geqslant 0$ for all $x \in
  [\lambda_2, \lambda_1]$ and for $x \in [\lambda_2, \lambda_1] \backslash
  \left[ \frac{\lambda_2}{2}, \frac{\lambda_1}{2} \right],$
  \[ F_0 (x) \geqslant C_2 (\lambda_1, \lambda_2, \delta) > 0. \]
  We have the decomposition
  \[ V (y, t) = \frac{e^{\frac{- \delta}{(\lambda_1 - \lambda_2)^2}
     \frac{y^2}{t}}}{\lambda_1 - \lambda_2} e^{- t F_0 \left( \frac{y}{t}
     \right)} H \left( t G_0 \left( \frac{y}{t} \right) \right) . \]
  For $ \frac{y}{t} \in \left[ \frac{\lambda_2}{2}, \frac{\lambda_1}{2}
  \right]$, $\left| e^{- t F_0 \left( \frac{y}{t} \right)} \right| \leqslant
  1$ and $H \left( t G_0 \left( \frac{y}{t} \right) \right) \leqslant H (C_1
  (\lambda_1 \comma \lambda_2, \delta) t) \leqslant \frac{K (\lambda_1,
  \lambda_2, \delta)}{\sqrt{1 + t}}$ thus for $a_0 = \frac{\delta}{(\lambda_1
  - \lambda_2)^2} > 0$, we have
  \[ | V (y, t) | \leqslant \frac{C (\lambda_1, \lambda_2, \delta) e^{\frac{-
     \delta}{(\lambda_1 - \lambda_2)^2} \frac{y^2}{t}}}{(1 + t)^{\frac{1}{2}}}
     \leqslant \frac{K_1 (\lambda_1, \lambda_2, \delta) e^{a_0 \frac{y^2}{1 +
     t}}}{(1 + t)^{\frac{1}{2}}} . \]
  And for $\frac{y}{t} \in [\lambda_2, \lambda_1] \backslash \left[
  \frac{\lambda_2}{2}, \frac{\lambda_1}{2} \right]$, we have $e^{- t F_0
  \left( \frac{y}{t} \right)} \leqslant e^{- C_2 (\lambda_1, \lambda_2,
  \delta) t}, \left| H \left( t G_0 \left( \frac{y}{t} \right) \right) \right|
  \leqslant K$ and thus
  \[ | V (y, t) | \leqslant C (\lambda_1, \lambda_2, \delta) e^{\frac{-
     \delta}{(\lambda_1 - \lambda_2)^2} \frac{y^2}{t}} e^{- C_2 (\lambda_1,
     \lambda_2, \delta) t} \leqslant \frac{K_2 (\lambda_1, \lambda_2, \delta)
     e^{a_0 \frac{y^2}{1 + t}}}{(1 + t)^{\frac{1}{2}}} . \]
  This concludes the estimate for $C_{0, 0} = \max (K_1, K_2)$.
\end{proof}

\subsubsection{Estimates on derivatives of $V$}

\begin{lemma}
  For $\delta > 0, \lambda_1 > 0, \lambda_2 < 0$, the function
  \[ V (y, t) = \frac{e^{- \frac{2 \delta}{(\lambda_1 - \lambda_2)^2} \left( -
     \lambda_1 \lambda_2 t + \frac{\lambda_1 + \lambda_2}{2} y
     \right)}}{\lambda_1 - \lambda_2} I_0 \left( \frac{2 \delta \sqrt{-
     \lambda_1 \lambda_2}}{(\lambda_1 - \lambda_2)^2} \sqrt{- (y - \lambda_1
     t) (y - \lambda_2 t)} \right) \]
  is smooth in both $y$ and $t$ on $t \geqslant 0, y \in [\lambda_2 t,
  \lambda_1 t]$, including at the boundary.
\end{lemma}

\begin{proof}
  We simply check, with Lemma \ref{L24vf}, that
  \begin{eqnarray}
    &  & I_0 \left( \frac{2 \delta \sqrt{- \lambda_1 \lambda_2}}{(\lambda_1 -
    \lambda_2)^2} \sqrt{- (y - \lambda_1 t) (y - \lambda_2 t)} \right)
    \nonumber\\
    & = & \sum_{n = 0}^{+ \infty} \frac{1}{n! \Gamma (n + 1)} \left(
    \frac{\delta^2 \lambda_1 \lambda_2}{(\lambda_1 - \lambda_2)^4} (y -
    \lambda_1 t) (y - \lambda_2 t) \right)^n,  \label{eq23vf}
  \end{eqnarray}
  which leads to the differentiability of $V$ on $t \geqslant 0, y \in
  [\lambda_2 t, \lambda_1 t]$, even at the boundary $y = \lambda_2 t$ or $y =
  \lambda_1 t$. In fact, by this formula,
  \[ V (y, t) = \frac{e^{- \frac{2 \delta}{(\lambda_1 - \lambda_2)^2} \left( -
     \lambda_1 \lambda_2 t + \frac{\lambda_1 + \lambda_2}{2} y
     \right)}}{\lambda_1 - \lambda_2} \sum_{n = 0}^{+ \infty} \frac{1}{n!
     \Gamma (n + 1)} \left( \frac{\delta^2 \lambda_1 \lambda_2}{(\lambda_1 -
     \lambda_2)^4} (y - \lambda_1 t) (y - \lambda_2 t) \right)^n \]
  can be extended to $\mathbb{R}_y \times \mathbb{R}_t$ and is smooth on it.
\end{proof}

We can now conclude the proof of Lemma \ref{Vest}.

\begin{proof}[of Lemma \ref{Vest}]
  We check easily the estimate for $0 \leqslant t \leqslant 1$ and we suppose
  from now on that $t \geqslant 1$. We have, under the notations of the proof
  of Lemma \ref{Vest1}, that
  \[ V (y, t) = \frac{e^{\gamma_1 (y, t)}}{\lambda_1 - \lambda_2} I_0
     (\gamma_2 (y, t)) \]
  with
  \[ \gamma_1 (y, t) = \frac{- 2 \delta}{(\lambda_1 - \lambda_2)^2} \left( -
     \lambda_1 \lambda_2 t + \frac{\lambda_1 + \lambda_2}{2} y \right) \]
  and
  \[ \gamma_2 (y, t) = \frac{2 \delta \sqrt{- \lambda_1 \lambda_2}}{(\lambda_1
     - \lambda_2)^2} \sqrt{- (y - \lambda_1 t) (y - \lambda_2 t)} . \]
  We compute that
  \[ \partial_t V = V \left( \partial_t \gamma_1 + \partial_t (\gamma_2)
     \frac{I_0' (\gamma_2)}{I_0 (\gamma_2)} \right) = V \left( \partial_t
     \gamma_1 + \partial_t \gamma_2 + \partial_t \gamma_2 \left( \frac{I_0'
     (\gamma_2)}{I_0 (\gamma_2)} - 1 \right) \right) . \]
  Under the previous notations, we have
  \[ \gamma_2 (y, t) = t G_0 \left( \frac{y}{t} \right), \]
  where
  \[ G_0 (z) = \frac{2 \delta \sqrt{- \lambda_1 \lambda_2}}{(\lambda_1 -
     \lambda_2)^2} \sqrt{- (z - \lambda_1) (z - \lambda_2)} . \]
  We check that $\partial_t \gamma_1 = - G_0 (0)$ and
  \[ \partial_t \gamma_2 = G_0 \left( \frac{y}{t} \right) - \frac{y}{t} G_0'
     \left( \frac{y}{t} \right) . \]
  Therefore,
  \[ \partial_t \gamma_1 + \partial_t \gamma_2 = - G_0 (0) + G_0 \left(
     \frac{y}{t} \right) - \frac{y}{t} G_0' \left( \frac{y}{t} \right), \]
  We have
  \begin{eqnarray*}
    \partial_t V & = & \frac{e^{\frac{- \delta}{(\lambda_1 - \lambda_2)^2}
    \frac{y^2}{t}}}{\lambda_1 - \lambda_2} e^{- t F_0 \left( \frac{y}{t}
    \right)} H \left( t G_0 \left( \frac{y}{t} \right) \right) \left( - G_0
    (0) + G_0 \left( \frac{y}{t} \right) - \frac{y}{t} G_0' \left( \frac{y}{t}
    \right) \right)\\
    & + & \frac{e^{\frac{- \delta}{(\lambda_1 - \lambda_2)^2}
    \frac{y^2}{t}}}{\lambda_1 - \lambda_2} e^{- t F_0 \left( \frac{y}{t}
    \right)} H \left( t G_0 \left( \frac{y}{t} \right) \right) \partial_t
    \gamma_2 \left( \frac{I_0' (\gamma_2)}{I_0 (\gamma_2)} - 1 \right) .
  \end{eqnarray*}
  We have
  \[ G_0'' (0) = \frac{1}{- \lambda_1 \lambda_2} \left( \sqrt{- \lambda_1
     \lambda_2} + \frac{\lambda_1 + \lambda_2}{2} \right) > 0 \]
  (it is $0$ if and only if $0 < \lambda_1 - \lambda_2 = 0$) and therefore
  \[ - G_0 (0) + G_0 \left( \frac{y}{t} \right) - \frac{y}{t} G_0' \left(
     \frac{y}{t} \right) = O_{\frac{y}{t} \rightarrow 0} \left(
     \frac{y^2}{t^2} \right) . \]
  We compute
  \[ F_0'' (0) = \frac{- \delta (\lambda_1 + \lambda_2)^2}{2 \lambda_1
     \lambda_2 (\lambda_1 - \lambda_2)^2} \geqslant 0 \]
  but it can be $0$ if $\lambda_1 + \lambda_2 = 0$. We therefore estimate
  \begin{eqnarray*}
    &  & | \partial_t V (y, t) |\\
    & \leqslant & \frac{C (\lambda_1, \lambda_2, \delta)}{\sqrt{t}}
    e^{\frac{- \delta}{2 (\lambda_1 - \lambda_2)^2} \frac{y^2}{t}} e^{- t
    \left( F_0 \left( \frac{y}{t} \right) + \frac{\delta y^2}{2 (\lambda_1 -
    \lambda_2)^2 t^2} \right)} \left| - G_0 (0) + G_0 \left( \frac{y}{t}
    \right) - \frac{y}{t} G_0' \left( \frac{y}{t} \right) \right|\\
    & + & \frac{C (\lambda_1, \lambda_2, \delta)}{\sqrt{t}} e^{\frac{-
    \delta}{2 (\lambda_1 - \lambda_2)^2} \frac{y^2}{t}} e^{- t \left( F_0
    \left( \frac{y}{t} \right) + \frac{\delta y^2}{2 (\lambda_1 - \lambda_2)^2
    t^2} \right)} \left| \partial_t \gamma_2 \left( \frac{I_0' (\gamma_2)}{I_0
    (\gamma_2)} - 1 \right) \right| .
  \end{eqnarray*}
  Since $\left( F_0 (x) + \frac{\delta x^2}{2 (\lambda_1 - \lambda_2)^2}
  \right)'' (0) > 0$ and $G_0'' (0) > 0$, we check that for $\frac{y}{t} \in
  \left[ \frac{\lambda_2}{2}, \frac{\lambda_1}{2} \right]$,
  \begin{eqnarray*}
    &  & e^{- t \left( F_0 \left( \frac{y}{t} \right) + \frac{\delta y^2}{2
    (\lambda_1 - \lambda_2)^2 t^2} \right)} \left| - G_0 (0) + G_0 \left(
    \frac{y}{t} \right) - \frac{y}{t} G_0' \left( \frac{y}{t} \right)
    \right|\\
    & \leqslant & \frac{C (\lambda_1, \lambda_2, \delta)}{t} .
  \end{eqnarray*}
  This is a consequence of the fact that for $j \geqslant 1, t > 0$,
  \begin{equation}
    \sup_{y \geqslant 0} \{ e^{- t y^2} y^j \} = \frac{1}{t^{j / 2}} \left(
    \frac{j}{2 e} \right)^{j / 2} . \label{eqq240vf}
  \end{equation}
  Similarly, for $\frac{y}{t} \in \left[ \frac{\lambda_2}{2},
  \frac{\lambda_1}{2} \right]$,
  \begin{eqnarray*}
    &  & e^{- t \left( F_0 \left( \frac{y}{t} \right) + \frac{\delta y^2}{2
    (\lambda_1 - \lambda_2)^2 t^2} \right)} \left| \partial_t \gamma_2 \left(
    \frac{I_0' (\gamma_2)}{I_0 (\gamma_2)} - 1 \right) \right|\\
    & \leqslant & C (\lambda_1, \lambda_2, \delta) e^{- t \left( F_0 \left(
    \frac{y}{t} \right) + \frac{\delta y^2}{2 (\lambda_1 - \lambda_2)^2 t^2}
    \right)} \left| \frac{\partial_t \gamma_2}{\gamma_2} \right|\\
    & \leqslant & \frac{C (\lambda_1, \lambda_2, \delta)}{t} .
  \end{eqnarray*}
  For $\frac{y}{t} \in [\lambda_2, \lambda_1] \backslash \left[
  \frac{\lambda_2}{2}, \frac{\lambda_1}{2} \right]$, we simply go back to
  \[ V (y, t) = \frac{e^{- \frac{2 \delta}{(\lambda_1 - \lambda_2)^2} \left( -
     \lambda_1 \lambda_2 t + \frac{\lambda_1 + \lambda_2}{2} y
     \right)}}{\lambda_1 - \lambda_2} I_0 \left( \frac{2 \delta \sqrt{-
     \lambda_1 \lambda_2}}{(\lambda_1 - \lambda_2)^2} \sqrt{- (y - \lambda_1
     t) (y - \lambda_2 t)} \right), \]
  and remark that there exists $a > 0$ such that
  \[ e^{- \frac{2 \delta}{(\lambda_1 - \lambda_2)^2} \left( - \lambda_1
     \lambda_2 t + \frac{\lambda_1 + \lambda_2}{2} y \right)} \leqslant e^{- a
     t} . \]
  With equation (\ref{eq23vf}) we check that
  \[ \left| \partial_t \left( I_0 \left( \frac{2 \delta \sqrt{- \lambda_1
     \lambda_2}}{(\lambda_1 - \lambda_2)^2} \sqrt{- (y - \lambda_1 t) (y -
     \lambda_2 t)} \right) \right) \right| \leqslant K (\lambda_1, \lambda_2,
     \delta) (1 + t) \]
  if $\frac{y}{t} \in [\lambda_2, \lambda_1]$. We deduce that for $\frac{y}{t}
  \in [\lambda_2, \lambda_1] \backslash \left[ \frac{\lambda_2}{2},
  \frac{\lambda_1}{2} \right], t \geqslant 1$, $t \geqslant \alpha
  \frac{y^2}{1 + t}$ for some $\alpha > 0$, thus
  \begin{equation}
    | \partial_t V (y, t) | \leqslant K (1 + t) e^{- a t} \leqslant K (1 + t)
    e^{- \frac{a}{2} t} e^{- \frac{a}{2 \alpha} \frac{y^2}{1 + t}} \leqslant
    \frac{K}{t^{\frac{1}{2}}} e^{- \frac{a}{2 \alpha} \frac{y^2}{t}} .
    \label{eq24vf}
  \end{equation}
  We now generalize to higher order derivatives in time. We recall that
  \[ \gamma_2 (y, t) = t G_0 \left( \frac{y}{t} \right) . \]
  We show by induction that
  \[ \partial_t^n \gamma_2 (y, t) = t^{1 - n} G_n \left( \frac{y}{t} \right),
  \]
  with
  \[ G_{n + 1} (x) = (1 - n) G_n (x) - x G_n' (x) . \]
  Furthermore, we have
  \[ I_0 (x), I_0' (x) = \frac{e^x}{\sqrt{2 \pi x}} \left( 1 + O_{x
     \rightarrow + \infty} \left( \frac{1}{x} \right) \right), \]
  therefore
  \[ \frac{I_0' (x)}{I_0 (x)} = 1 + O_{x \rightarrow + \infty} \left(
     \frac{1}{x} \right) . \]
  We can show, using the development of $I_0, I_0' = I_1$ from
  {\cite{MR1349110}}, that for $x \geqslant 0, k \geqslant 1$,
  \begin{equation}
    \left| \partial_x^k \left( \frac{I_0' (x)}{I_0 (x)} \right) \right|
    \leqslant \frac{K (k)}{(1 + x)^{k + 1}} . \label{eq26vf}
  \end{equation}
  We recall that
  \[ \partial_t V = V \left( \partial_t \gamma_1 + \partial_t \gamma_2 \left.
     \frac{I_0' (\gamma_2)}{I_0 (\gamma_2)} \right) \right. \]
  and $\partial_t^2 \gamma_1 = 0$. For $n \geqslant 1$ we have
  \[ \partial_t^{n + 1} V = \sum_{k = 0}^n \left(\begin{array}{c}
       n\\
       k
     \end{array}\right) \partial_t^{n - k} V \partial_t^k \left( \partial_t
     \gamma_1 + \partial_t \gamma_2 \frac{I_0' (\gamma_2)}{I_0 (\gamma_2)}
     \right) . \]
  We check that
  \[ \left| \partial_t^k \left( \partial_t \gamma_2 \left. \frac{I_0'
     (\gamma_2)}{I_0 (\gamma_2)} \right) \right) \right| \leqslant \frac{K
     (k)}{(1 + t)^k} . \]
  Indeed, each derivatives hit either a term of the form $\partial_t^j
  \gamma_2$, which yields an additional $\frac{1}{1 + t}$, of it hit a term of
  the form $\partial_x^j \left( \frac{I_0'}{I_0} \right) (\gamma_2)$, and
  \[ \partial_t \left( \partial_x^j \left( \frac{I_0'}{I_0} \right) (\gamma_2)
     \right) = \partial_t \gamma_2 \partial_x^{j + 1} \left( \frac{I_0'}{I_0}
     \right) (\gamma_2), \]
  adding a term in $\frac{1}{(1 + \gamma_2)} \sim \frac{1}{1 + t}$ if
  $\frac{y}{t} \in \left[ \frac{\lambda_2}{2}, \frac{\lambda_1}{2} \right]$.
  Outside of $\left[ \frac{\lambda_2}{2}, \frac{\lambda_1}{2} \right]$, we use
  the exponential decay as for the proof of (\ref{eq24vf}).
  
  For the term $\partial_t^n V \left( \partial_t \gamma_1 + \partial_t
  \gamma_2 \frac{I_0' (\gamma_2)}{I_0 (\gamma_2)} \right)$, we still use the
  fact that
  \[ \left| \partial_t \gamma_1 + \partial_t \gamma_2 \frac{I_0'
     (\gamma_2)}{I_0 (\gamma_2)} \right| \leqslant O_{\frac{y}{t} \rightarrow
     0} \left( \frac{y^2}{t^2} \right) + \frac{K}{t} \]
  and we simply use equation (\ref{eqq240vf}). By induction we deduce the
  estimate on $\partial_t^n V$.
  
  \
  
  Now, we have
  \[ \partial_y V = V \left( \partial_y \gamma_1 + \partial_y \gamma_2 +
     \partial_y \gamma_2 \left( \frac{I_0' (\gamma_2)}{I_0 (\gamma_2)} - 1
     \right) \right) \]
  and we compute that
  \[ \partial_y \gamma_1 = - \frac{\delta (\lambda_1 + \lambda_2)}{(\lambda_1
     - \lambda_2)^2} = - G_0' (0) \]
  and
  \[ \partial_y \gamma_2 = G_0' \left( \frac{y}{t} \right), \]
  therefore
  \[ \partial_y \gamma_1 + \partial_y \gamma_2 = G_0' \left( \frac{y}{t}
     \right) - G_0' (0) = O_{\frac{y}{t} \rightarrow 0} \left( \frac{y}{t}
     \right) . \]
  We conclude that
  \[ | \partial_y V | \leqslant \frac{C (\lambda_1, \lambda_2, \delta)
     e^{\frac{- \delta}{2 (\lambda_1 - \lambda_2)^2} \frac{y^2}{t}}}{t} . \]
  Now, we have
  \begin{equation}
    \partial_y^{n + 1} V = \sum_{k = 0}^n \left(\begin{array}{c}
      n\\
      k
    \end{array}\right) \partial_y^{n - k} V \partial_y^k \left( \partial_y
    \gamma_1 + \partial_y \gamma_2 \frac{I_0' (\gamma_2)}{I_0 (\gamma_2)}
    \right) . \label{eq27vf}
  \end{equation}
  We have $\partial_y^2 \gamma_1 = 0$ and $\gamma_2 (y, t) = t G_0 \left(
  \frac{y}{t} \right)$, therefore
  \[ \partial_y^j \gamma_2 = t^{1 - j} G_0^{(j)} \left( \frac{y}{t} \right) .
  \]
  In the sum, we deal with the case $k = 0$ as in the case of the derivative
  in time. Otherwise, we estimate that for $k \geqslant 1$,
  \[ \left| \partial_y^k \left( \partial_y \gamma_2 \frac{I_0' (\gamma_2)}{I_0
     (\gamma_2)} \right) \right| \leqslant \frac{K (k)}{(1 + t)^k} . \]
  Indeed, derivatives either falls on term of the form $\partial_y^j
  \gamma_2$, adding a factor $\frac{1}{t}$, or it falls on a term of the form
  $\partial^j_y \left( \frac{I_0' (\gamma_2)}{I_0 (\gamma_2)} \right)$, and we
  have $\partial_y \left( \partial^j_y \left( \frac{I_0' (\gamma_2)}{I_0
  (\gamma_2)} \right) \right) = \partial_y \gamma_2 \partial^{j + 1}_y \left(
  \frac{I_0' (\gamma_2)}{I_0 (\gamma_2)} \right) .$ Since $\partial_y
  \gamma_2$ is bounded in time and $\gamma_2 \simeq t$, we get an additional
  factor $\frac{1}{t}$ from (\ref{eq26vf}).
  
  \
  
  We conclude by induction that
  \[ | \partial_y^n V | \leqslant \frac{K (n, \lambda_1, \lambda_2,
     \delta)}{(1 + t)^{1 / 2 + n / 2}} e^{\frac{- \delta}{(\lambda_1 -
     \lambda_2)^2} \frac{y^2}{t}} . \]
  Remark that the only term with the decay in time $\frac{1}{(1 + t)^{1 / 2 +
  n / 2}}$ is the case $k = 0$ in the sum (\ref{eq27vf}).
  
  We complete with the cross derivatives. Starting from
  \[ \partial_t V = V \left( \partial_t \gamma_1 + \partial_t \gamma_2 \left.
     \frac{I_0' (\gamma_2)}{I_0 (\gamma_2)} \right) \right., \]
  we have
  \[ \partial_t^{n + 1} \partial_y^k V = \sum_{i = 0}^n \sum_{j = 0}^k
     \left(\begin{array}{c}
       k\\
       j
     \end{array}\right) \left(\begin{array}{c}
       n\\
       i
     \end{array}\right) \partial_y^{k - j} \partial_t^{n - i} V \partial_y^j
     \partial_t^i \left( \partial_t \gamma_1 + \partial_t \gamma_2 \left.
     \frac{I_0' (\gamma_2)}{I_0 (\gamma_2)} \right) \right. . \]
  By a similar induction estimate, separating the case $i = 0, j = 0$ from the
  other, and by remarking that $\partial_t \partial_y \gamma_1 = 0$, we
  conclude the proof of this lemma.
\end{proof}

\section{Estimation on some convolutions}\label{sect3vf}

We recall the quantities for $\mu_1 > 0, \mu_2 < 0, \mu \in \mathbb{R}, a,
\delta, \gamma > 0, w \in L^1_{\tmop{loc}} (\mathbb{R}, \mathbb{R}^+)$ defined
in the introduction
\[ F_{a, \gamma, w}^{\mu_1, \mu_2} (x, t) = \frac{1}{(1 + t)^{\gamma}}
   \int_{\mu_2 t}^{\mu_1 t} e^{- a \frac{y^2}{1 + t}} w (x - y) d y \]
and
\[ G_{b, \delta, w}^{\mu} (x, t) = e^{- b t} \sup_{z \in [- \delta t, \delta
   t]} w (x - \mu t + z) . \]
Remark that since $w \geqslant 0,$we have $F_{a, \gamma, w}^{\mu_1, \mu_2}
\geqslant 0$ and $G_{b, \delta, w}^{\mu} \geqslant 0$. In this section, we are
going to estimate some convolutions involving these two quantities, that will
appear in the Duhamel formulation of the linear and nonlinear equation. They
are in subsection \ref{31usef} for the ones concerning $F_{a, \gamma,
w}^{\mu_1, \mu_2}$, then in subsection \ref{ss32vf} for the ones concerning
$G_{b, \delta, w}^{\mu}$. Furthermore, in subsection \ref{ss33vf}, we will
give bounds on $F_{a, \gamma, w}^{\mu_1, \mu_2}$ and $G_{b, \delta, w}^{\mu}$
in $L^1 (\mathbb{R})$ and $L^{\infty} (\mathbb{R})$.

\subsection{Convolutions involving $F_{a, \gamma, w}^{\mu_1,
\mu_2}$}\label{31usef}

\begin{lemma}
  \label{L33b4}Consider, for $\mu_1 > 0, \mu_2 < 0, \mu \in [\mu_2, \mu_1], a
  > 0, a_1 > 0, \gamma > 0, w \in L^1_{\tmop{loc}} (\mathbb{R}, \mathbb{R}^+)$
  the function
  \[ H (x, t) = \int_0^t e^{- a_1 (t - s)} F_{a, \gamma, w}^{\mu_1, \mu_2} (x
     - \mu (t - s), s) d s. \]
  Then, there exists $K (a_1, \gamma) > 0$ such that, if $a \leqslant
  \frac{a_1}{2 \mu^2}$, then
  \[ H (x, t) \leqslant K (a_1, \gamma) F_{a, \gamma, w}^{\mu_1, \mu_2} (x, t)
     . \]
\end{lemma}

\begin{proof}
  We have
  \[ H (x, t) = \int_0^t \frac{e^{- a_1 (t - s)}}{(1 + s)^{\gamma}}
     \int_{\mu_2 s}^{\mu_1 s} e^{- a \frac{y^2}{(1 + s)}} w (x - \mu (t - s) -
     y) d y d s. \]
  We first do the change of variable $z = y + \mu (t - s)$ leading to
  \begin{eqnarray*}
    &  & \int_{\mu_2 s}^{\mu_1 s} e^{- a \frac{y^2}{(1 + s)}} w (x - \mu (t -
    s) - y) d y\\
    & = & \int_{\mu_2 s + \mu (t - s)}^{\mu_1 s + \mu (t - s)} e^{- a
    \frac{(z - \mu (t - s))^2}{(1 + s)}} w (x - z) d z.
  \end{eqnarray*}
  Remark that for $s \in [0, t]$ and $\mu \in [\mu_2, \mu_1]$, we have
  \[ \mu_2 s + \mu (t - s) \geqslant \mu_2 t \]
  and
  \[ \mu_1 s + \mu (t - s) \leqslant \mu_1 t. \]
  We deduce that
  \[ \int_{\mu_2 s + \mu (t - s)}^{\mu_1 s + \mu (t - s)} e^{- a \frac{(z -
     \mu (t - s))^2}{(1 + s)}} w (x - z) d z \leqslant \int_{\mu_2 t}^{\mu_1
     t} e^{- a \frac{(z - \mu (t - s))^2}{(1 + s)}} w (x - z) d z \]
  and thus
  \[ H (x, t) \leqslant \int_0^t \frac{1}{(1 + s)^{\gamma}} \int_{\mu_2
     t}^{\mu_1 t} e^{- a \frac{z^2}{(1 + t)}} e^{f (z, s, t)} | w (x - z) | d
     z d s \]
  where
  \[ f (z, s, t) = - a_1 (t - s) - a \frac{(z - \mu (t - s))^2}{(1 + s)} + a
     \frac{z^2}{(1 + t)} . \]
  For $s \in [0, t]$, we deduce that if $a \leqslant \frac{a_1}{2 \mu^2}$,
  then
  \[ f (z, s, t) \leqslant a \left( - \frac{(\mu (t - s))^2}{t - s} - \frac{(z
     - \mu (t - s))^2}{(1 + s)} + \frac{z^2}{(1 + t)} \right) - \frac{a_1}{2}
     (t - s) . \]
  We compute in general that for $x, y \in \mathbb{R}$,
  \begin{eqnarray}
    &  & \frac{y^2}{t - s} + \frac{(x - y)^2}{1 + s} - \frac{x^2}{1 + t}
    \nonumber\\
    & = & \frac{(1 + s) (1 + t) y^2 + (x - y)^2 (t - s) (1 + t) - x^2 (1 + s)
    (t - s)}{(t - s) (1 + s) (1 + t)} \nonumber\\
    & = & \frac{(y^2 (1 + t)^2 + x^2 (t - s)^2 - 2 x y (t - s) (1 + t))}{(t -
    s) (1 + s) (1 + t)} \nonumber\\
    & = & \frac{(x (t - s) - y (1 + t))^2}{(t - s) (1 + s) (1 + t)} . 
    \label{31ca}
  \end{eqnarray}
  Applying it to $x = z, y = \mu (t - s)$, we deduce that for $s \in [0, t]$,
  \begin{eqnarray*}
    &  & - \frac{(\mu (t - s))^2}{t - s} - \frac{(z - \mu (t - s))^2}{(1 +
    s)} + \frac{z^2}{(1 + t)}\\
    & = & - \frac{(t - s) (z - \mu (1 + t))^2}{(1 + s) (1 + t)} \leqslant 0.
  \end{eqnarray*}
  We deduce that
  \begin{eqnarray*}
    H (x, t) & \leqslant & \int_{\mu_2 t}^{\mu_1 t} e^{- a \frac{z^2}{(1 +
    t)}} w (x - z) \int_0^t \frac{1}{(1 + s)^{\gamma}} e^{- \frac{a_1}{2} (t -
    s)} d s d z\\
    & \leqslant & \frac{K (a_1, \gamma)}{(1 + t)^{\gamma}} \int_{\mu_2
    t}^{\mu_1 t} e^{- a \frac{z^2}{(1 + t)}} w (x - z) d z.
  \end{eqnarray*}
  This concludes the proof of this lemma.
\end{proof}

\begin{lemma}
  \label{L34v4}Consider, for $\mu_1 > 0, \mu_2 < 0, a, a_0, \alpha, \beta,
  \delta > 0, w \in L^1_{\tmop{loc}} (\mathbb{R}, \mathbb{R}^+)$, with
  $\delta$ small compared to $\mu_1, | \mu_2 |$, the function
  \[ H (x, t) = \int_0^t \int_{\mu_2 (t - s)}^{\mu_1 (t - s)} \frac{e^{- a_0
     \frac{y^2}{(1 + t - s)}}}{(1 + t - s)^{\alpha}} F_{a, \beta, w}^{\mu_1 +
     \delta, \mu_2 - \delta} (x - y, s) d y d s. \]
  Then, there exists $K > 0$ depending on $\alpha, \beta, a_0, a, \mu_1,
  \mu_2$ and $\varepsilon (a_0) > 0$ such that if $a \leqslant \varepsilon
  (a_0)$, then
  \begin{eqnarray}
    &  & H (x, t) \nonumber\\
    & \leqslant & K \left( \frac{L_{\beta - 1 / 2} (t)}{(1 + t)^{\alpha}} +
    \frac{L_{\alpha - 1 / 2} (t)}{(1 + t)^{\beta}} \right) \int_{(\mu_2 -
    \delta) t}^{(\mu_1 + \delta) t} w (x - r) e^{- a \frac{r^2}{(1 + t)}} d r,
    \label{eq32vf}
  \end{eqnarray}
  where
  \[ L_{\gamma} (t) \assign \int_0^t \frac{d s}{(1 + s)^{\gamma}} . \]
\end{lemma}

Equation (\ref{eq32vf}) is an estimate where the right hand side if $F_{a,
\gamma, w}^{\mu_1 + \delta, \mu_2 - \delta}$ for some $\gamma \in \mathbb{R}$.
In particular, if $\alpha, \beta \neq \frac{3}{2}$, then
\[ \frac{L_{\beta - 1 / 2} (t)}{(1 + t)^{\alpha}} + \frac{L_{\alpha - 1 / 2}
   (t)}{(1 + t)^{\beta}} \leqslant \frac{K}{(1 + t)^{\alpha + \beta -
   \frac{3}{2}}} \]
for $K$ depending on $\alpha, \beta$ and thus
\begin{equation}
  H (x, t) \leqslant K F_{a, \alpha + \beta - \frac{3}{2}, w}^{\mu_1 + \delta,
  \mu_2 - \delta} (x, t) .
\end{equation}
If $\alpha = \frac{3}{2}$ or $\beta = \frac{3}{2}$, then there is an
additional factor $\ln (1 + t)$. This is the reason why in the estimates of
Theorem \ref{NLconstlocTheorem} there is a small loss of decay in time for the
derivatives of $u$ and $\rho$.

\begin{proof}
  We have
  \[ H (x, t) = \int_0^t \int_{\mu_2 (t - s)}^{\mu_1 (t - s)} \frac{e^{- a_0
     \frac{y^2}{(1 + t - s)}}}{(1 + t - s)^{\alpha} (1 + s)^{\beta}}
     \int_{(\mu_2 - \delta) s}^{(\mu_1 + \delta) s} e^{- a \frac{z^2}{(1 +
     s)}} w (x - y - z) d z d y d s. \]
  By the change of variable $r = y + z,$ we have
  \[ \int_{(\mu_2 - \delta) s}^{(\mu_1 + \delta) s} e^{- a \frac{z^2}{(1 +
     s)}} w (x - y - z) d z = \int_{(\mu_2 - \delta) s + y}^{(\mu_1 + \delta)
     s + y} e^{- a \frac{(r - y)^2}{(1 + s)}} w (x - r) d r. \]
  For $y \in [\mu_2 (t - s), \mu_1 (t - s)]$ and $s \in [0, t]$, we have that
  \[ (\mu_2 - \delta) s + y \geqslant (\mu_2 - \delta) t \text{ and } (\mu_1 +
     \delta) s + y \leqslant (\mu_1 + \delta) t, \]
  therefore
  \begin{eqnarray*}
    &  & H (x, t)\\
    & \leqslant & \int_{(\mu_2 - \delta) t}^{(\mu_1 + \delta) t} w (x - r)
    e^{- a \frac{r^2}{(1 + t)}} \int_0^t \int_{\mu_2 (t - s)}^{\mu_1 (t - s)}
    \frac{e^{- \frac{a_0}{2} \frac{y^2}{(1 + t - s)}}}{(1 + t - s)^{\alpha} (1
    + s)^{\beta}} e^{f (r, y, t, s)} d y d s d r
  \end{eqnarray*}
  with
  \[ f (r, y, t, s) = a \frac{r^2}{(1 + t)} - \frac{a_0}{2} \frac{y^2}{(1 + t
     - s)} - a \frac{(r - y)^2}{(1 + s)} . \]
  If $t \leqslant 3$, then the domain of integration of $r, y$ and $s$ are all
  bounded and then for $r \in [(\mu_2 - \delta) t, (\mu_1 + \delta) t]$ we
  have
  \[ \int_0^t \int_{\mu_2 (t - s)}^{\mu_1 (t - s)} \frac{e^{- \frac{a_0}{2}
     \frac{y^2}{(1 + t - s)}}}{(1 + t - s)^{\alpha} (1 + s)^{\beta}} e^{f (r,
     y, t, s)} d y d s \leqslant K (\mu_1, \mu_2, a_0, a), \]
  concluding the proof. Now, if $t \geqslant 3$, we look first at the case $s
  \in [t - 2, t]$. Then,
  \[ f (r, y, t, s) \leqslant a \frac{r^2}{1 + t} - \frac{a_0}{6} y^2 -
     \frac{a}{t - 1} (r - y)^2 \leqslant 0 \]
  given that $t \geqslant 3$ and $a \leqslant \varepsilon (a_0)$ for some
  function $\varepsilon (a_0) > 0$. In that case, $y \in [2 \mu_2, 2 \mu_1]$
  and therefore
  \[ \int_{t - 2}^t \int_{\mu_2 (t - s)}^{\mu_1 (t - s)} \frac{e^{-
     \frac{a_0}{2} \frac{y^2}{(1 + t - s)}}}{(1 + t - s)^{\alpha} (1 +
     s)^{\beta}} e^{f (r, y, t, s)} d y d s \leqslant \frac{K (\mu_1, \mu_2,
     a_0, a)}{(1 + t)^{\beta}} . \]
  Now, if $t \geqslant 3$ and $s \in [0, t - 2]$, which is the main case, we
  have
  \[ f (r, y, t, s) \leqslant a \frac{r^2}{(1 + t)} - \frac{a_0}{6}
     \frac{y^2}{t - s} - a \frac{(r - y)^2}{(1 + s)} \]
  and by equation (\ref{31ca}), we deduce that for $a \leqslant \varepsilon
  (a_0)$, we have
  \[ f (r, y, t, s) \leqslant 0. \]
  Therefore, for the case $t - 2 \geqslant s \geqslant t / 2$,
  \begin{eqnarray*}
    &  & \int_{t / 2}^{t - 2} \int_{\mu_2 (t - s)}^{\mu_1 (t - s)} \frac{e^{-
    \frac{a_0}{2} \frac{y^2}{(1 + t - s)}}}{(1 + t - s)^{\alpha} (1 +
    s)^{\beta}} e^{f (r, y, t, s)} d y d s\\
    & \leqslant & \int_{t / 2}^{t - 2} \int_{\mu_2 (t - s)}^{\mu_1 (t - s)}
    \frac{e^{- \frac{a_0}{2} \frac{y^2}{(1 + t - s)}}}{(1 + t - s)^{\alpha} (1
    + s)^{\beta}} d y d s\\
    & \leqslant & K (a_0) \int_{t / 2}^t \frac{1}{(1 + t - s)^{\alpha - 1 /
    2} (1 + s)^{\beta}} d s\\
    & \leqslant & \frac{K (a_0)}{(1 + t)^{\beta}} L_{\alpha - 1 / 2} (t) .
  \end{eqnarray*}
  Now, remark that if $s \leqslant t / 2$ and $a \leqslant \varepsilon (a_0)$,
  there exists $a_2 (a_0, a) > 0$ such that
  \[ f (r, y, t, s) \leqslant - a_2 \frac{y^2}{(1 + s)} . \]
  We deduce that
  \begin{eqnarray*}
    &  & \int_0^{t / 2} \int_{\mu_2 (t - s)}^{\mu_1 (t - s)} \frac{e^{-
    \frac{a_0}{2} \frac{y^2}{(1 + t - s)}}}{(1 + t - s)^{\alpha} (1 +
    s)^{\beta}} e^{f (r, y, t, s)} d y d s\\
    & \leqslant & \int_0^{t / 2} \int_{\mu_2 (t - s)}^{\mu_1 (t - s)}
    \frac{e^{- a_2 \frac{y^2}{(1 + s)}}}{(1 + t - s)^{\alpha} (1 + s)^{\beta}}
    d y d s\\
    & \leqslant & \int_0^{t / 2} \frac{1}{(1 + t - s)^{\alpha} (1 + s)^{\beta
    - 1 / 2}} d s\\
    & \leqslant & \frac{K (\alpha)}{(1 + t)^{\alpha}} L_{\beta - 1 / 2} (t) .
  \end{eqnarray*}
  This concludes the proof of this lemma.
\end{proof}

\subsection{Convolution with a boundary term}\label{ss32vf}

We recall the notation
\[ F_{a, \gamma, w}^{\mu_1, \mu_2} (x, t) = \frac{1}{(1 + t)^{\gamma}}
   \int_{\mu_2 t}^{\mu_1 t} e^{- a \frac{y^2}{1 + t}} w (x - y) d y. \]
We are interested here in the quantity for $b, \delta > 0,$ $\mu \in
\mathbb{R}$ and $w \in L^1_{\tmop{loc}} (\mathbb{R}, \mathbb{R}^+)$:
\[ G_{b, \delta, w}^{\mu} (x, t) = e^{- b t} \sup_{y \in [- \delta t, \delta
   t]} w (x - \mu t + y) \geqslant 0. \]
\begin{lemma}
  \label{L36v4}Consider $b, \delta > 0, \mu \in \{ \mu_1, \mu_2 \}, w \in
  L^1_{\tmop{loc}} (\mathbb{R}, \mathbb{R}^+)$ and the function
  \[ G_{b, \delta, w}^{\mu} (x, t) = e^{- b t} \sup_{z \in [- \delta t, \delta
     t]} w (x - \mu t + z) . \]
  Then, for any $k, n \in \mathbb{N}$, there exists $C > 0$ depending on
  $\mu_1, \mu_2, b, n, k$ and $a > 0$ depending on $\mu_1, \mu_2, b$ such that
  \begin{eqnarray*}
    &  & \int_0^t \int_{\mu_2 (t - s)}^{\mu_1 (t - s)} | \partial_x^k
    \partial_t^n V (y, t - s) G_{b, \delta, w}^{\mu} (x - y, s) | d y d s\\
    & \leqslant & C F_{a, \frac{1}{2} + \frac{k}{2} + n, w}^{\mu_1 + \delta,
    \mu_2 - \delta} (x, t) .
  \end{eqnarray*}
\end{lemma}

\begin{proof}
  First, by a change of variable on $y$, we compute that
  \begin{eqnarray*}
    &  & \int_0^t \int_{\mu_2 (t - s)}^{\mu_1 (t - s)} | \partial_x^k
    \partial_t^n V (y, t - s) e^{- b s} \sup_{z \in [- \delta s, \delta s]} w
    (x - y - \mu s + z) | d y d s\\
    & \leqslant & \int_0^t \sup_{z \in [- \delta s, \delta s]} \int_{\mu_2 (t
    - s) + \mu s - z}^{\mu_1 (t - s) + \mu s - z} | \partial_x^k \partial_t^n
    V (y - \mu s + z, t - s) | e^{- b s} w (x - y) d y d s.
  \end{eqnarray*}
  Remark that for $\mu \in \{ \mu_1, \mu_2 \}$ and $s \in [0, t]$, we have
  $\mu_1 (t - s) + \mu s \leqslant \mu_1 t$ and $\mu_2 (t - s) + \mu s
  \geqslant \mu_2 t$. We deduce that, denoting
  \[ A = \int_0^t \int_{\mu_2 (t - s)}^{\mu_1 (t - s)} | \partial_x^k
     \partial_t^n V (y, t - s) e^{- b s} \sup_{z \in [- \delta s, \delta s]} w
     (x - y - \mu s + z) | d y d s, \]
  that
  \[ A \leqslant \int_{(\mu_2 - \delta) t}^{(\mu_1 + \delta) t} w (x - y)
     \left( \int_0^t \sup_{z \in [- \delta s, \delta s]} | \partial_x^k
     \partial_t^n V (y - \mu s + z, t - s) | e^{- b s} d s \right) d y. \]
  We recall from Lemma \ref{Vest} that there exists $a_0 > 0$ such that
  \[ | \partial_x^k \partial_t^n V (y - \mu s + z, t - s) | \leqslant \frac{K
     (k, n, \mu_1, \mu_2)}{(1 + t - s)^{\frac{1}{2} + \frac{k}{2} + n}} e^{-
     a_0 \frac{(y - \mu s + z)^2}{1 + t - s}} . \]
  Now, first for the case $s \geqslant t / 2$, we estimate using simply $|
  \partial_x^k \partial_t^n V (y - \mu s, t - s) | \leqslant K (k, n)$ that
  \begin{eqnarray*}
    &  & \int_{\frac{t}{2}}^t \sup_{z \in [- \delta s, \delta s]} |
    \partial_x^k \partial_t^n V (y - \mu s + z, t - s) | e^{- b s} d s\\
    & \leqslant & K (k, n, \mu_1, \mu_2) e^{- \frac{b}{4} t} \int_{t / 2}^t
    e^{- \frac{b}{2} s} d s\\
    & \leqslant & K (k, n, \mu_1, \mu_2, b) e^{- \frac{b}{4} t},
  \end{eqnarray*}
  and thus, using that for $y \in [(\mu_2 - \delta) t, (\mu_1 + \delta) t]$,
  we have $t \geqslant K (\mu_1, \mu_2, \delta) \frac{y^2}{(1 + t)}$, we
  deduce
  \begin{eqnarray*}
    &  & \int_{(\mu_2 - \delta) t}^{(\mu_1 + \delta) t} w (x - y) \left(
    \int_{\frac{t}{2}}^t \sup_{z \in [- \delta s, \delta s]} | \partial_x^k
    \partial_t^n V (y - \mu s + z, t - s) | e^{- b s} d s \right) d y\\
    & \leqslant & K (k, n, \mu_1, \mu_2, b) e^{- \frac{b}{8} t} \int_{(\mu_2
    - \delta) t}^{(\mu_1 + \delta) t} w (x - y) e^{- \frac{b}{4} t} d y\\
    & \leqslant & \frac{K (k, n, \mu_1, \mu_2, b)}{(1 + t)^{\frac{1}{2} +
    \frac{k}{2} + n}} \int_{(\mu_2 - \delta) t}^{(\mu_1 + \delta) t} e^{-
    \tilde{a} \frac{y^2}{1 + t}} w (x - y) d y\\
    & \leqslant & K (k, n, \mu_1, \mu_2, b) F_{\tilde{a}, \frac{1}{2} +
    \frac{k}{2} + n, w}^{\mu_1 + \delta, \mu_2 - \delta} (x, t)
  \end{eqnarray*}
  for some small $\tilde{a} > 0$ depending on $b, \mu_1, \mu_2$. Now, for the
  case $s \leqslant t / 2$, we estimate
  \begin{eqnarray*}
    &  & \int_0^{t / 2} \sup_{z \in [- \delta s, \delta s]} | \partial_x^k
    \partial_t^n V (y - \mu s + z, t - s) | e^{- b s} d s\\
    & \leqslant & \frac{K (k, n, \mu_1, \mu_2)}{(1 + t)^{\frac{1}{2} +
    \frac{k}{2} + n}} \int_0^{t / 2} \sup_{z \in [- \delta s, \delta s]} e^{-
    \frac{a_0}{4} \frac{(y - \mu s + z)^2}{1 + t} - \frac{b}{2} s} e^{-
    \frac{b}{2} s} d s,
  \end{eqnarray*}
  and we check that for $s \in [0, t / 2], y \in [\mu_2 t, \mu_1 t], z \in [-
  \delta s, \delta s]$ there exists $a' > 0$ depending on $\mu_1, \mu_2, b,
  \delta$ such that
  \[ - \frac{a_0}{4} \frac{(y - \mu s + z)^2}{1 + t} - \frac{b}{2} s \leqslant
     - a' \frac{y^2}{1 + t} . \]
  We deduce that
  \begin{eqnarray*}
    &  & \int_0^{t / 2} \sup_{z \in [- \delta s, \delta s]} | \partial_x^k
    \partial_t^n V (y - \mu s + z, t - s) | e^{- b s} d s\\
    & \leqslant & \frac{K (k, n, \mu_1, \mu_2)}{(1 + t)^{\frac{1}{2} +
    \frac{k}{2} + n}} e^{- a' \frac{y^2}{1 + t}} \int_0^{t / 2} e^{-
    \frac{b}{2} s} d s\\
    & \leqslant & \frac{K (k, n, \mu_1, \mu_2, b)}{(1 + t)^{\frac{1}{2} +
    \frac{k}{2} + n}} e^{- a' \frac{y^2}{1 + t}}
  \end{eqnarray*}
  and therefore
  \begin{eqnarray*}
    &  & \int_{(\mu_2 - \delta) t}^{(\mu_1 + \delta) t} w (x - y) \left(
    \int_0^{t / 2} \sup_{z \in [- \delta s, \delta s]} | \partial_x^k
    \partial_t^n V (y - \mu s + z, t - s) | e^{- b s} d s \right) d y\\
    & \leqslant & \frac{K (k, n, \mu_1, \mu_2, b, \delta)}{(1 +
    t)^{\frac{1}{2} + \frac{k}{2} + n}} \int_{(\mu_2 - \delta) t}^{(\mu_1 +
    \delta) t} w (x - y) e^{- a' \frac{y^2}{1 + t}} d y\\
    & \leqslant & K (k, n, \mu_1, \mu_2, b) F_{a', \frac{1}{2} + \frac{k}{2}
    + n, w}^{\mu_1 + \delta, \mu_2 - \delta} (x, t) .
  \end{eqnarray*}
  We conclude that
  \[ A \leqslant K (k, n, \mu_1, \mu_2, b) F_{a, \frac{1}{2} + \frac{k}{2} +
     n, w}^{\mu_1 + \delta, \mu_2 - \delta} (x, t) \]
  for $a = \min (a', \tilde{a})$.
\end{proof}

\begin{lemma}
  \label{L37b4}Consider $b, \delta > 0, \mu, \mu' \in \{ \mu_1, \mu_2 \}$ with
  $b < a_0$ ($a_0$ is defined in Lemma \ref{Vest}), $\delta < \frac{1}{2} \min
  (\mu_1, | \mu_2 |)$ and the function
  \[ G_{b, \delta, w}^{\mu} (x, t) = e^{- b t} \sup_{z \in [- \delta t, \delta
     t]} w (x - \mu t + z) . \]
  Then, there exists $C > 0$ depending on $k, n, \mu_1, \mu_2$ and $a > 0$
  depending on $\mu_1, \mu_2, b$ such that
  \begin{eqnarray*}
    &  & \int_0^t | \partial_x^k \partial_t^n V (\mu' (t - s), t - s) G_{b,
    \delta, w}^{\mu} (x - \mu' (t - s), s) | d s\\
    & \leqslant & C \left( G_{b, \delta, w}^{\mu'} (x, t) + F_{a, \frac{1}{2}
    + \frac{k}{2} + n, w}^{\mu_1 + \delta, \mu_2 - \delta} (x, t) \right) .
  \end{eqnarray*}
\end{lemma}

\begin{proof}
  We recall from Lemma \ref{Vest} that there exists $a_0 > 0$ depending on
  $\mu_1, \mu_2$ such that
  \[ | \partial_x^k \partial_t^n V (\mu' (t - s), t - s) | \leqslant K (k, n,
     \mu_1, \mu_2) e^{- a_0 (t - s)} . \]
  We deduce that
  \begin{eqnarray*}
    &  & \int_0^t | \partial_x^k \partial_t^n V (\mu' (t - s), t - s) G_{b,
    \delta, w}^{\mu} (x - \mu' (t - s), s) | d s\\
    & \leqslant & K (k, n, \mu_1, \mu_2) \int_0^t e^{- a_0 (t - s)} e^{- b s}
    \sup_{z \in [- \delta s, \delta s]} w (x - \mu s - \mu' (t - s) + z) d s.
  \end{eqnarray*}
  If $\mu = \mu'$, then $- \mu s - \mu' (t - s) = - \mu' t$ and
  \begin{eqnarray*}
    &  & \int_0^t | \partial_x^k \partial_t^n V (\mu' (t - s), t - s) G_{b,
    \delta, w}^{\mu} (x - \mu' (t - s), s) | d s\\
    & \leqslant & K (k, n, \mu_1, \mu_2) \sup_{z \in [- \delta t, \delta t]}
    w (x - \mu' t + z) \int_0^t e^{- a_0 (t - s)} e^{- b s} d s\\
    & \leqslant & K (k, n, \mu_1, \mu_2, b) G_{b, \delta, w}^{\mu'} (x, t) .
  \end{eqnarray*}
  Now, if $\mu \neq \mu'$, then we do the change of variable $y = \mu s + \mu'
  (t - s) - z$, leading to
  \begin{eqnarray*}
    &  & \int_0^t | \partial_x^k \partial_t^n V (\mu' (t - s), t - s) G_{b,
    \delta, w}^{\mu} (x - \mu' (t - s), s) | d s\\
    & \leqslant & \frac{K (k, n, \mu_1, \mu_2)}{\mu - \mu'} \sup_{z \in [-
    \delta t, \delta t]} \int_{\mu' t + z}^{\mu t + z} e^{- a_0 (t - s)} e^{-
    b s} w (x - y) d y\\
    & \leqslant & K (k, n, \mu_1, \mu_2) \int_{(\mu_2 - \delta) t}^{(\mu_1 +
    \delta) t} e^{- a_0 (t - s)} e^{- b s} w (x - y) d y
  \end{eqnarray*}
  and there exists $a > 0$ small enough depending on $a_0$ and $b$ such that
  $e^{- a_0 (t - s)} e^{- b s} \leqslant e^{- a t} e^{- a \frac{y^2}{1 + t}}$
  for all $y \in [\mu_2 t, \mu_1 t]$, therefore
  \begin{eqnarray*}
    &  & \int_0^t | \partial_x^k \partial_t^n V (\mu' (t - s), t - s) G_{b,
    \delta, w}^{\mu} (x - \mu' (t - s), s) d s |\\
    & \leqslant & K (k, n, \mu_1, \mu_2) e^{- a t} \int_{(\mu_2 - \delta)
    t}^{(\mu_1 + \delta) t} e^{- a \frac{y^2}{1 + t}} w (x - y) d y\\
    & \leqslant & K (k, n, \mu_1, \mu_2) F_{a, \frac{1}{2} + \frac{k}{2} + n,
    w}^{\mu_1 + \delta, \mu_2 - \delta} (x, t),
  \end{eqnarray*}
  concluding the proof of this lemma.
\end{proof}

\subsection{Estimates on $F_{a, \gamma, w}^{\mu_1, \mu_2}$ and $G_{b, \delta,
w}^{\mu}$}\label{ss33vf}

\begin{lemma}
  Consider for $\gamma, a > 0, \mu_1 > 0, \mu_2 < 0, w \in L^1 (\mathbb{R},
  \mathbb{R}^+)$ the function
  \[ F_{a, \gamma, w}^{\mu_1, \mu_2} (x, t) = \frac{1}{(1 + t)^{\gamma}}
     \int_{\mu_2 t}^{\mu_1 t} e^{- a \frac{y^2}{1 + t}} w (x - y) d y. \]
  Then, there exists $K > 0$ depending on $\mu_1, \mu_2, a$ such that for all
  $x \in \mathbb{R}, t \geqslant 0$,
  \[ \| F_{a, \gamma, w}^{\mu_1, \mu_2} \|_{L^1_x} (t) \leqslant \frac{K}{(1 +
     t)^{\gamma - 1 / 2}} \| w \|_{L^1 (\mathbb{R})} \]
  and
  \[ | F_{a, \gamma, w}^{\mu_1, \mu_2} (x, t) | \leqslant \frac{K}{(1 +
     t)^{\gamma}} \| w \|_{L^1 (\mathbb{R})} . \]
\end{lemma}

\begin{proof}
  We simply compute
  \begin{eqnarray*}
    \int_{\mathbb{R}} F_{a, \gamma, w}^{\mu_1, \mu_2} (x, t) d x & \leqslant &
    \frac{\| w \|_{L^1 (\mathbb{R})}}{(1 + t)^{\gamma}} \int_{\mu_2 t}^{\mu_1
    t} e^{- a \frac{y^2}{1 + t}} d y\\
    & \leqslant & \frac{K \| w \|_{L^1 (\mathbb{R})}}{(1 + t)^{\gamma - 1 /
    2}}
  \end{eqnarray*}
  and
  \begin{eqnarray*}
    F_{a, \gamma, w}^{\mu_1, \mu_2} (x, t) & \leqslant & \frac{\| w \|_{L^1
    (\mathbb{R})}}{(1 + t)^{\gamma}} \left\| e^{- a \frac{y^2}{1 + t}}
    \right\|_{L^{\infty} (\mathbb{R})}\\
    & \leqslant & \frac{\| w \|_{L^1 (\mathbb{R})}}{(1 + t)^{\gamma}},
  \end{eqnarray*}
  concluding the proof.
\end{proof}

For the estimates on
\[ G_{b, \delta, w}^{\mu} (x, t) = e^{- b t} \sup_{z \in [- \delta t, \delta
   t]} w (x - \mu t + z), \]
we check easily that
\begin{equation}
  | G_{b, \delta, w}^{\mu} (x, t) | \leqslant e^{- b t} \| w \|_{L^{\infty}
  (\mathbb{R})}
\end{equation}
and
\begin{equation}
  \| G_{b, \delta, w}^{\mu} (., t) \|_{L^1 (\mathbb{R})} \leqslant e^{- b t}
  \| w \|_{L^1 (\mathbb{R})} .
\end{equation}
\section{Nonlinear change of variable to straighten the light
cone}\label{ss33v4}

For the equation
\[ \partial_t^2 f + (\lambda_1 + \lambda_2) \partial_{x t}^2 f + \lambda_1
   \lambda_2 \partial_x^2 f + \delta \partial_t f = 0 \]
with $\lambda_1 > 0, \lambda_2 < 0, \delta > 0$, we have an explicit solution
thanks to Proposition \ref{dwesol} in the case that these three coefficient
are constants. However, we see that in the nonlinear stability problem (Lemma
\ref{lem32}), these coefficients depends on the unknown. We can try to put
them as a source term. That is, if $\lambda_1 \simeq \lambda_1^0, \lambda_2
\simeq \lambda_2^0, \delta \simeq \delta^0$ where $\lambda_1^0, \lambda_2^0,
\delta^0$ are constants and we write it
\[ \partial_t^2 f + (\lambda_1^0 + \lambda_2^0) \partial_{x t}^2 f +
   \lambda_1^0 \lambda^0_2 \partial_x^2 f + \delta \partial_t f = S \]
with
\[ S = - (\lambda_1 + \lambda_2 - \lambda_1^0 - \lambda_2^0) \partial_{x t}^2
   f - (\lambda_1 \lambda_2 - \lambda_1^0 \lambda^0_2) \partial_x^2 f -
   (\delta - \delta^0) \partial_t f. \]
If we consider $f_L$ the solution to the linear problem
\[ \left\{\begin{array}{l}
     \partial_t^2 f_L + (\lambda_1^0 + \lambda_2^0) \partial_{x t}^2 f_L +
     \lambda_1^0 \lambda^0_2 \partial_x^2 f_L + \delta^0 \partial_t f_L = 0\\
     f_{L | t = 0 \nobracket} = f_0, \partial_t f_{L | t = 0 \nobracket} =
     f_1,
   \end{array}\right. \]
we have from Proposition \ref{dwesol} that
\begin{eqnarray*}
  f (x, t) & = & f_L (x, t)\\
  & + & \int_0^t \int_{\lambda_2^0 (t - s)}^{\lambda^0_1 (t - s)} V (y, t -
  s) S (x - y, s) d y d s,
\end{eqnarray*}
where $V$ is defined with $\lambda^0_1, \lambda^0_2, \delta^0$. If we want to
do an estimate on $f$ with a bootstrap, we see that with this formulation,
provided some estimates on the second derivatives of $f$, we can estimate $f$
and its first derivatives. Indeed, if we differentiate it with respect to $t$,
the derivative never fall on $S$, and if we differentiate with respect to $x$,
when the derivative fall on $S$ we can do an integration by parts.

However, it is not (a priori) possible to estimate the second derivatives of
$f$ without informations on its third derivatives. In this sense, we have a
loss of derivative in this problem. To solve this issue, we will do a
nonlinear change of variable that will change the caracteristic speeds
$\lambda_1, \lambda_2$ to constants $\lambda_1^0, \lambda_2^0$. This will be
use only to estimate second derivatives.

\subsection{Reformulation of the equation}

We consider the equation
\[ \partial_t^2 f + (\lambda_1 + \lambda_2) \partial_{x t}^2 f + \lambda_1
   \lambda_2 \partial_x^2 f + \delta \partial_t f + w \partial_x f = 0, \]
where $\lambda_1, \lambda_2, \delta, w$ depend on $(x, t)$ and $f_{| t = 0
\nobracket} = \partial_t f_{| t = 0 \nobracket} = 0$. In fact, it can depend
on them through $f$, we will precise this later. We suppose that $\lambda_1,
\lambda_2$ are close respectfully to $\lambda_1^0 > 0, \lambda_2^0 < 0$. For
now, we just want to define the change of variable. Specific estimate on the
functions that will appear will be done later.

We write
\[ f (x, t) = g (h_1 (x, t), h_2 (x, t)) . \]
In the case $\lambda_1 = \lambda_1^0, \lambda_2 = \lambda_2^0$ where
$\lambda_1^0, \lambda_2^0$ are constants, we would take $h_1 (x, t) = x$ and
$h_2 (x, t) = t$. Here, we only suppose that
\[ h_1 (x, 0) = x, h_2 (x, 0) = 0. \label{hinit} \]
We denote these new variables by $(y, \nu) \assign (h_1 (x, t), h_2 (x, t))$.
Here and afterward, $g$ is taken in $(y, \nu) = (h_1 (x, t), h_2 (x, t))$ and
$\lambda_1, \lambda_2, \delta, w, h_1, h_2$ are taken in $(x, t)$. We compute
\[ \partial_x f = \partial_x h_1 \partial_y g + \partial_x h_2 \partial_{\nu}
   g, \]
\[ \partial_x^2 f = (\partial_x h_1)^2 \partial_y^2 g + 2 \partial_x h_1
   \partial_x h_2 \partial_{y \nu}^2 g + (\partial_x h_2)^2 \partial_{\nu}^2 g
   + \partial_x^2 h_1 \partial_y g + \partial_x^2 h_2 \partial_{\nu} g, \]
\begin{eqnarray*}
  \partial_{x t}^2 f & = & \partial_x h_1 \partial_t h_1 \partial_y^2 g +
  (\partial_x h_1 \partial_t h_2 + \partial_x h_2 \partial_t h_1) \partial_{y
  \nu}^2 g + \partial_x h_2 \partial_t h_2 \partial_{\nu}^2 g\\
  & + & \partial_{x t}^2 h_1 \partial_y g + \partial_{x t}^2 h_2
  \partial_{\nu} g,
\end{eqnarray*}
\[ \partial_t f = \partial_t h_1 \partial_y g + \partial_t h_2 \partial_{\nu}
   g \]
and
\[ \partial_t^2 f = (\partial_t h_1)^2 \partial_y^2 g + 2 \partial_t h_1
   \partial_t h_2 \partial_{y \nu}^2 g + (\partial_t h_2)^2 \partial_{\nu}^2 g
   + \partial_t^2 h_1 \partial_y g + \partial_t^2 h_2 \partial_{\nu} g. \]
We deduce that $g$ satisfies the equation
\begin{eqnarray}
  &  & \partial_{\nu}^2 g ((\partial_t h_2)^2 + (\lambda_1 + \lambda_2)
  \partial_x h_2 \partial_t h_2 + \lambda_1 \lambda_2 (\partial_x h_2)^2)
  \nonumber\\
  & + & \partial_{y \nu}^2 g (2 \partial_t h_1 \partial_t h_2 + (\lambda_1 +
  \lambda_2) (\partial_x h_1 \partial_t h_2 + \partial_x h_2 \partial_t h_1) +
  2 \lambda_1 \lambda_2 \partial_x h_1 \partial_x h_2) \nonumber\\
  & + & \partial_y^2 g ((\partial_t h_1)^2 + (\lambda_1 + \lambda_2)
  \partial_x h_1 \partial_t h_1 + \lambda_1 \lambda_2 (\partial_x h_1)^2)
  \nonumber\\
  & + & \partial_{\nu} g (\partial_t^2 h_2 + (\lambda_1 + \lambda_2)
  \partial_{x t}^2 h_2 + \lambda_1 \lambda_2 \partial_x^2 h_2 + \delta
  \partial_t h_2 + w \partial_x h_2) \nonumber\\
  & + & \partial_y g (\partial_t^2 h_1 + (\lambda_1 + \lambda_2) \partial_{x
  t}^2 h_1 + \lambda_1 \lambda_2 \partial_x^2 h_1 + \delta \partial_t h_1 + w
  \partial_x h_1) \nonumber\\
  & = & 0.  \label{fatcomp}
\end{eqnarray}
Dividing by $(\partial_x h_2)^2 + (\lambda_1 + \lambda_2) \partial_x h_2
\partial_t h_2 + \lambda_1 \lambda_2 (\partial_x h_2)^2$, we therefore want to
choose $h_1, h_2$ such that
\[ \frac{2 \partial_t h_1 \partial_t h_2 + (\lambda_1 + \lambda_2) (\partial_x
   h_1 \partial_t h_2 + \partial_x h_2 \partial_t h_1) + 2 \lambda_1 \lambda_2
   \partial_x h_1 \partial_x h_2}{(\partial_t h_2)^2 + (\lambda_1 + \lambda_2)
   \partial_x h_2 \partial_t h_2 + \lambda_1 \lambda_2 (\partial_x h_2)^2} =
   \lambda_1^0 + \lambda_2^0 \]
and
\[ \frac{(\partial_t h_1)^2 + (\lambda_1 + \lambda_2) \partial_x h_1
   \partial_t h_1 + \lambda_1 \lambda_2 (\partial_x h_1)^2}{(\partial_t h_2)^2
   + (\lambda_1 + \lambda_2) \partial_x h_2 \partial_t h_2 + \lambda_1
   \lambda_2 (\partial_x h_2)^2} = \lambda_1^0 \lambda_2^0, \]
where $\lambda_1^0, \lambda_2^0$ are constants. We remark that
\begin{eqnarray*}
  &  & (\partial_t h_2)^2 + (\lambda_1 + \lambda_2) \partial_x h_2 \partial_t
  h_2 + \lambda_1 \lambda_2 (\partial_x h_2)^2\\
  & = & (\partial_t h_2 + \lambda_1 \partial_x h_2) (\partial_t h_2 +
  \lambda_2 \partial_x h_2),
\end{eqnarray*}
\begin{eqnarray*}
  &  & (\partial_t h_1)^2 + (\lambda_1 + \lambda_2) \partial_x h_1 \partial_t
  h_1 + \lambda_1 \lambda_2 (\partial_x h_1)^2\\
  & = & (\partial_t h_1 + \lambda_1 \partial_x h_1) (\partial_t h_1 +
  \lambda_2 \partial_x h_1)
\end{eqnarray*}
and
\begin{eqnarray*}
  &  & 2 \partial_t h_1 \partial_t h_2 + (\lambda_1 + \lambda_2) (\partial_x
  h_1 \partial_t h_2 + \partial_x h_2 \partial_t h_1) + 2 \lambda_1 \lambda_2
  \partial_x h_1 \partial_x h_2\\
  & = & (\partial_t h_1 + \lambda_1 \partial_x h_1) (\partial_t h_2 +
  \lambda_2 \partial_x h_2) + (\partial_t h_1 + \lambda_2 \partial_x h_1)
  (\partial_t h_2 + \lambda_1 \partial_x h_2) .
\end{eqnarray*}
Therefore, if we define
\[ A_1 \assign \partial_t h_1 + \lambda_1 \partial_x h_1, A_2 = \partial_t h_1
   + \lambda_2 \partial_x h_1 \]
as well as
\[ B_1 \assign \partial_t h_2 + \lambda_1 \partial_x h_2, B_2 = \partial_t h_2
   + \lambda_2 \partial_x h_2, \]
we have
\begin{equation}
  A_1 A_2 = \lambda_1^0 \lambda_2^0 B_1 B_2, A_1 B_2 + A_2 B_1 = (\lambda_1^0
  + \lambda_2^0) B_1 B_2, \label{420}
\end{equation}
therefore if $B_1, B_2 \neq 0$,
\[ \frac{A_1}{B_1} \frac{A_2}{B_2} = \lambda_1^0 \lambda_2^0, \frac{A_1}{B_1}
   + \frac{A_2}{B_2} = \lambda_1^0 + \lambda_2^0 . \]
Equality (\ref{420}) is satisfied if
\[ A_1 = \lambda_1^0 B_1, A_2 = \lambda_2^0 B_2, \]
that is
\[ (\partial_t + \lambda_1 \partial_x) (h_1 - \lambda_1^0 h_2) = (\partial_t +
   \lambda_2 \partial_x) (h_1 - \lambda_2^0 h_2) = 0. \]
We define $\tilde{h}_1, \tilde{h}_2$ by
\[ h_1 (x, t) = x + \tilde{h}_1 (x, t), h_2 (x, t) = t + \tilde{h}_2 (x, t) \]
and
\[ H_1 \assign \tilde{h}_1 - \lambda_1^0 \tilde{h}_2, H_2 \assign \tilde{h}_1
   - \lambda_2^0 \tilde{h}_2, \]
with $H_{1 | t = 0 \nobracket} = H_{2 | t = 0 \nobracket} = 0$ (by
\ref{hinit}). Then,
\begin{equation}
  (\partial_t + \lambda_1 \partial_x) H_1 = \lambda_1^0 - \lambda_1,
  (\partial_t + \lambda_2 \partial_x) H_2 = \lambda_2^0 - \lambda_2 .
  \label{29eq}
\end{equation}
We find back $\tilde{h}_1, \tilde{h}_2$ by
\begin{equation}
  \tilde{h}_1 = \frac{\lambda_2^0 H_1 - \lambda_1^0 H_2}{\lambda_2^0 -
  \lambda_1^0}, \tilde{h}_2 = \frac{H_1 - H_2}{\lambda_2^0 - \lambda_1^0} .
  \label{htilde27}
\end{equation}
We assume for now that we are able to show that for some small
$\tilde{\varepsilon} > 0$, we have
\[ | \partial_t H_1 (x, t) | + | \partial_x H_1 (x, t) | + | \partial_t H_2
   (x, t) | + | \partial_x H_2 (x, t) | \leqslant \tilde{\varepsilon} . \]
This will be done using equation (\ref{29eq}) and some estimate on
$\lambda_1^0 - \lambda_1, \lambda_2^0 - \lambda_2$ later on. We deduce in that
case that
\[ (x, t) \rightarrow (x + \tilde{h}_1 (x, t), t + \tilde{h}_2 (x, t)) = (y
   (x, t), \nu (x, t)) \]
is an inversible function since
\[ | \partial_t \tilde{h}_1 | + | \partial_x \tilde{h}_1 | + | \partial_t
   \tilde{h}_2 | + | \partial_x \tilde{h}_2 | \leqslant K \tilde{\varepsilon}
   \ll 1. \]
It is in fact close to the identity, in particular this function and its
inverse have a norm close to one. This allows us to define $g$ through $f$.

We check indeed that $t + \tilde{h}_2 (x, t) \geqslant 0$ for all $t \geqslant
0$, using the fact that $\tilde{h}_2 (x, 0) = 0$ and $| \partial_t \tilde{h}_2
| \leqslant K \varepsilon$ which implies that $| \tilde{h}_2 (x, t) |
\leqslant K \tilde{\varepsilon} t \leqslant \frac{t}{2}$ given
$\tilde{\varepsilon}$ small enough. We have then
\begin{equation}
  \partial_t h_2 (x, t) = \partial_t (t + \tilde{h}_2 (x, t)) = 1 + \partial_t
  \tilde{h}_2 (x, t) \geqslant 1 - \tilde{\varepsilon} \label{time} .
\end{equation}

We go back on the equation (\ref{fatcomp}) satisfied by $g$. We have
\begin{eqnarray*}
  &  & \partial_{\nu}^2 g + (\lambda_1^0 + \lambda_2^0) \partial_{y \nu}^2 g
  + \lambda_1^0 \lambda_2^0 \partial_y^2 g\\
  & + & \frac{\partial_t^2 h_2 + (\lambda_1 + \lambda_2) \partial_{x t}^2 h_2
  + \lambda_1 \lambda_2 \partial_x^2 h_2 + \delta \partial_t h_2 + w
  \partial_x h_2}{(\partial_t h_2)^2 + (\lambda_1 + \lambda_2) \partial_x h_2
  \partial_t h_2 + \lambda_1 \lambda_2 (\partial_x h_2)^2} \partial_{\nu} g\\
  & + & \frac{\partial_t^2 h_1 + (\lambda_1 + \lambda_2) \partial_{x t}^2 h_1
  + \lambda_1 \lambda_2 \partial_x^2 h_1 + \delta \partial_t h_1 + w
  \partial_x h_1}{(\partial_t h_2)^2 + (\lambda_1 + \lambda_2) \partial_x h_2
  \partial_t h_2 + \lambda_1 \lambda_2 (\partial_x h_2)^2} \partial_y g\\
  & = & 0.
\end{eqnarray*}
Remark that $g$ is taken in $(h_1 (x, t), h_2 (x, t)) = (y, \nu)$ and $h_1,
h_2, \lambda_1, \lambda_2, \delta, w$ in $(x, t)$. We define $R_1 (x, t), R_2
(x, t)$ by
\begin{equation}
  \frac{\partial_t^2 h_2 + (\lambda_1 + \lambda_2) \partial_{x t}^2 h_2 +
  \lambda_1 \lambda_2 \partial_x^2 h_2 + \delta \partial_t h_2 + w \partial_x
  h_2}{(\partial_t h_2)^2 + (\lambda_1 + \lambda_2) \partial_x h_2 \partial_t
  h_2 + \lambda_1 \lambda_2 (\partial_x h_2)^2} (x, t) = \delta^0 + R_1 (x, t)
  \label{link3}
\end{equation}
and
\[ \frac{\partial_t^2 h_1 + (\lambda_1 + \lambda_2) \partial_{x t}^2 h_1 +
   \lambda_1 \lambda_2 \partial_x^2 h_1 + \delta \partial_t h_1 + w \partial_x
   h_1}{(\partial_t h_2)^2 + (\lambda_1 + \lambda_2) \partial_x h_2 \partial_t
   h_2 + \lambda_1 \lambda_2 (\partial_x h_2)^2} = R_2 (x, t) \]
We recall that $(x, t) \rightarrow (h_1 (x, t), h_2 (x, t))$ is smooth,
invertible and close to the identity for $\tilde{\varepsilon}$ small enough.
We define $\pi$ its inverse (that is, $\pi (y, \nu) = (\pi_1 (y, \nu), \pi_2
(y, \nu))$ is such that $y = h_1 (\pi_1 (y, \nu), \pi_2 (y, \nu)), \nu = h_2
(\pi_1 (y \comma \nu), \pi_2 (y, \nu))$). Then,
\[ g (y, \nu) = (f o \pi) (y, \nu) . \]
In particular,
\[ \partial_y g = \partial_y \pi_1 (\partial_x f o \pi) + \partial_y \pi_2
   (\partial_t f o \pi) \]
and
\[ \partial_{\nu} g = \partial_{\nu} \pi_1 (\partial_x f o \pi) +
   \partial_{\nu} \pi_2 (\partial_t f o \pi) . \]
Now, letting the variables be $(y, \nu) = (h_1 (x, t), h_2 (x, t))$, the
equation satisfied by $g$ is
\begin{eqnarray}
  &  & (\partial_{\nu}^2 g + (\lambda_1^0 + \lambda_2^0) \partial_{y \nu}^2 g
  + \lambda_1^0 \lambda_2^0 \partial_y^2 g + \delta^0 \partial_{\nu} g) (y,
  \nu) \nonumber\\
  & = & - (R_1 o \pi) \partial_{\nu} g - R_2 o \pi \partial_y
  g.  \label{eqg}
\end{eqnarray}
At $\nu = 0$, we have $t + h_2 (x, t) = 0$ and from (\ref{time}) this implies
that $t = 0$, hence $g (y, 0) = f (y, 0), \partial_{\nu} g (y, 0) =
\partial_{\nu} \pi_1 (y, 0) \partial_x f (y, 0) + \partial_{\nu} \pi_2 (y, 0)
\partial_t f (y, 0) .$ From Proposition \ref{dwesol}, this implies that for
any $y \in \mathbb{R}, \nu \geqslant 0$,
\begin{eqnarray*}
  g (y, \nu) & = g_L (y, \nu) + & \int_0^{\nu} \int_{\lambda_2^0 (\nu -
  s)}^{\lambda^0_1 (\nu - s)} V (z, \nu - s) S (y - z, s) d z d s,
\end{eqnarray*}
where $S (y - z, s) = (- R_1 o \pi \partial_{\nu} g - R_2 o \pi
\partial_y g) (y - z, s)$ and $g_L$ is solution to $\partial_{\nu}^2 g_L +
(\lambda_1^0 + \lambda_2^0) \partial_{y \nu}^2 g_L + \lambda_1^0 \lambda_2^0
\partial_y^2 g_L + \frac{1}{\tau} \partial_{\nu} g_L = 0$ with the initial
condition $g_L (y, 0) = f_0$, $\partial_{\nu} g_L (y, 0) = \partial_{\nu}
\pi_1 (y, 0) \partial_x f (y, 0) + \partial_{\nu} \pi_2 (y, 0) \partial_t f
(y, 0)$.

\

Our goal is to estimate $R_1, R_2$ with estimates on $\lambda_1, \lambda_2,
\delta, w$ that require the less possible amount of derivatives on these
functions. This is because they will depend on $f$, and to do a bootstrap we
will have information only on second derivatives of $f$.

First, we remark that for any function $h$, we have
\[ \partial_t^2 h + (\lambda_1 + \lambda_2) \partial_{x t}^2 h + \lambda_1
   \lambda_2 \partial_x^2 h = (\partial_t + \lambda_1 \partial_x) (\partial_t
   + \lambda_2 \partial_x) h - (\partial_t \lambda_2 + \lambda_1 \partial_x
   \lambda_2) \partial_x h \]
and
\[ \partial_t^2 h + (\lambda_1 + \lambda_2) \partial_{x t}^2 h + \lambda_1
   \lambda_2 \partial_x^2 h = (\partial_t + \lambda_2 \partial_x) (\partial_t
   + \lambda_1 \partial_x) h - (\partial_t \lambda_1 + \lambda_2 \partial_x
   \lambda_1) \partial_x h. \]
Therefore, since $h_1 = x + \frac{\lambda_2^0 H_1 - \lambda_1^0
H_2}{\lambda_2^0 - \lambda_1^0}$, we compute
\begin{eqnarray*}
  &  & \partial_t^2 h_1 + (\lambda_1 + \lambda_2) \partial_{x t}^2 h_1 +
  \lambda_1 \lambda_2 \partial_x^2 h_1\\
  & = & \frac{\lambda_2^0}{\lambda_2^0 - \lambda_1^0} (\partial_t^2 +
  (\lambda_1 + \lambda_2) \partial_{x t}^2 + \lambda_1 \lambda_2 \partial_x^2)
  H_1\\
  & - & \frac{\lambda_1^0}{\lambda_2^0 - \lambda_1^0} (\partial_t^2 +
  (\lambda_1 + \lambda_2) \partial_{x t}^2 + \lambda_1 \lambda_2 \partial_x^2)
  H_2\\
  & = & \frac{\lambda_2^0}{\lambda_2^0 - \lambda_1^0} ((\partial_t +
  \lambda_2 \partial_x) (\partial_t + \lambda_1 \partial_x) H_1 - (\partial_t
  \lambda_1 + \lambda_2 \partial_x \lambda_1) \partial_x H_1)\\
  & - & \frac{\lambda_1^0}{\lambda_2^0 - \lambda_1^0} ((\partial_t +
  \lambda_1 \partial_x) (\partial_t + \lambda_2 \partial_x) H_2 - (\partial_t
  \lambda_2 + \lambda_1 \partial_x \lambda_2) \partial_x H_2)\\
  & = & \frac{\lambda_2^0}{\lambda_2^0 - \lambda_1^0} (- (\partial_t
  \lambda_1 + \lambda_2 \partial_x \lambda_1) - (\partial_t \lambda_1 +
  \lambda_2 \partial_x \lambda_1) \partial_x H_1)\\
  & - & \frac{\lambda_1^0}{\lambda_2^0 - \lambda_1^0} (- (\partial_t
  \lambda_2 + \lambda_1 \partial_x \lambda_2) - (\partial_t \lambda_2 +
  \lambda_1 \partial_x \lambda_2) \partial_x H_2) .
\end{eqnarray*}
A similar estimate can be made for $\partial_t^2 h_2 + (\lambda_1 + \lambda_2)
\partial_{x t}^2 h_2 + \lambda_1 \lambda_2 \partial_x^2 h_2$. Remark that
$\partial_t^2 h_1 + (\lambda_1 + \lambda_2) \partial_{x t}^2 h_1 + \lambda_1
\lambda_2 \partial_x^2 h_1$ was the only component of $R_1$ that has two
derivatives in $h_1$ (see (\ref{link3})). The situation is similar for
$\partial_t^2 h_2 + (\lambda_1 + \lambda_2) \partial_{x t}^2 h_2 + \lambda_1
\lambda_2 \partial_x^2 h_2$ and $R_2$.

\subsection{Summary}\label{sumsum}

Consider the equation
\[ \partial_t^2 f + (\lambda_1 + \lambda_2) \partial_{x t}^2 f + \lambda_1
   \lambda_2 \partial_x^2 f + \delta \partial_t f + w \partial_x f = 0, \]
where $\lambda_1, \lambda_2, \delta, w$ depend on $(x, t)$. We suppose that
there exists constants $\lambda_1^0 > 0, \lambda_2^0 < 0, \delta^0 > 0$ that
are close respectfully to $\lambda_1, \lambda_2, \delta$. Then, we define
$H_1, H_2$ the solutions to the problems
\[ (\partial_t + \lambda_1 \partial_x) H_1 = \lambda_1^0 - \lambda_1,
   (\partial_t + \lambda_2 \partial_x) H_2 = \lambda_2^0 - \lambda_2 \]
with $H_{1 | t = 0 \nobracket} = H_{2 | t = 0 \nobracket} = 0$, as well as
\[ \tilde{h}_1 = \frac{\lambda_2^0 H_1 - \lambda_1^0 H_2}{\lambda_2^0 -
   \lambda_1^0}, \tilde{h}_2 = \frac{H_1 - H_2}{\lambda_2^0 - \lambda_1^0} .
\]
We suppose that we can show that
\[ | \partial_t \tilde{h}_1 | + | \partial_x \tilde{h}_1 | + | \partial_t
   \tilde{h}_2 | + | \partial_x \tilde{h}_2 | \ll 1, \]
in particular, the change of variable
\[ (y, \nu) = (x + \tilde{h}_1 (x, t), t + \tilde{h}_2 (x, t)) \]
is invertible and close to the identity. We then define the function $g$ by
\[ g (y, \nu) = f (\pi (y, \nu)) = f (x, t), \]
where $\pi (y, \nu) = (x, t)$ are obtain by solving $y = x + \tilde{h}_1 (x,
t), \nu = t + \tilde{h}_2 (x, t)$. Then, the function $g$ satisfies the
equation
\begin{eqnarray*}
  &  & (\partial_{\nu}^2 g + (\lambda_1^0 + \lambda_2^0) \partial_{y \nu}^2 g
  + \lambda_1^0 \lambda_2^0 \partial_y^2 g + \delta^0 \partial_{\nu} g) (y,
  \nu)\\
  & = & - (R_1 o \pi) \partial_{\nu} g - R_2 o \pi \partial_y
  g,
\end{eqnarray*}
with $g (y, 0) = f (y, 0), \partial_{\nu} g (y, 0) = \partial_{\nu} \pi_1 (y,
0) \partial_x f (y, 0) + \partial_{\nu} \pi_2 (y, 0) \partial_t f (y, 0)$, and
where $R_1, R_2$ are small functions depending on $H_1, H_2, \partial_x H_1,
\partial_x H_2, \partial_t H_1, \partial_t H_2, \partial_x \lambda_1,
\partial_t \lambda_1, \partial_x \lambda_2, \partial_t \lambda_2, \delta -
\delta^0, w$. In particular, they do not depend on second derivatives of $H_1,
H_2, \lambda_1$ or $\lambda_2$. It means that $R_1, R_2, \partial_t R_1,
\partial_x R_1, \partial_t R_2$ and $\partial_x R_2$ are small if we control
$H_1, H_2, \lambda_1, \lambda_2$ and their first and second derivatives, but
we do not need any informations on their third or more derivatives. Now, we
can use Proposition \ref{dwesol} on $g$, and since we have only first
derivatives in the right hand side (we are going to be able to control $n$
derivatives of $H_1, H_2$ with $n$ derivatives on our original variables),
there is no longer loss of derivatives. However, if we try to go back to $f$
by this method, all the derivatives are ``mixed'' together. We can not expect
to get a better decay than the worst derivative, that is $\partial_x f$, on
which the decay in time is $\frac{1}{(1 + t)^{1 - \nu}}$, the $\nu$-loss being
caused by convolutions (see the remark below Lemma \ref{L34v4}).

\section{Proof of the linear stability (Proposition
\ref{constlinstab})}\label{sec5vf}

\subsection{The equations of the linear problem}

\begin{lemma}
  \label{lineq31}The linear problem around the constant flow is
  \[ \partial_t \left(\begin{array}{c}
       \rho\\
       u
     \end{array}\right) + \left(\begin{array}{cc}
       U (\rho_0) & \rho_0\\
       0 & U (\rho_0) - \rho_0 h' (\rho_0)
     \end{array}\right) \partial_x \left(\begin{array}{c}
       \rho\\
       u
     \end{array}\right) + \frac{1}{\tau} \left(\begin{array}{cc}
       0 & 0\\
       U_f & 1
     \end{array}\right) \left(\begin{array}{c}
       \rho\\
       u
     \end{array}\right) = 0. \]
  In particular, with $\lambda_{\ast} = U (\rho_0) - \rho_0 U_f$, the
  functions
  \[ q (x, t) \assign \rho (x - \lambda_{\ast} t, t), v (x, t) \assign u (x -
     \lambda_{\ast} t, t) \]
  satisfies the equation
  \[ \partial_t^2 q + (\lambda^0_1 + \lambda^0_2) \partial_{x t}^2 q +
     \lambda^0_1 \lambda^0_2 \partial_x^2 q + \frac{1}{\tau} \partial_t q = 0
  \]
  and
  \[ \partial_t^2 v + (\lambda^0_1 + \lambda^0_2) \partial_{x t}^2 v +
     \lambda^0_1 \lambda^0_2 \partial_x^2 v + \frac{1}{\tau} \partial_t v = 0,
  \]
  where $\lambda^0_1 = \rho_0 U_f, \lambda^0_2 = \rho_0 U_f - \rho_0 h'
  (\rho_0)$.
\end{lemma}

\begin{proof}
  The linear equations satisfied by $\rho$ and $u$ are
  \[ \partial_t \rho + U (\rho_0) \partial_x \rho + \rho_0 \partial_x u = 0 \]
  and
  \[ \partial_t u + (U (\rho_0) - \rho_0 h' (\rho_0)) \partial_x u +
     \frac{1}{\tau} (U_f \rho + u) = 0. \]
  Taking $\partial_t + U (\rho_0) \partial_x$ of the second equation, we have
  \begin{eqnarray*}
    &  & \partial_t^2 u + (2 U (\rho_0) - \rho_0 h' (\rho_0)) \partial_{x
    t}^2 u + U (\rho_0) (U (\rho_0) - \rho_0 h' (\rho_0)) \partial_x^2 u\\
    & + & \frac{1}{\tau} (\partial_t u + (U (\rho_0) - \rho_0 U_f) \partial_x
    u)\\
    & = & 0.
  \end{eqnarray*}
  The equation on $v$ follows from the change of variable $ v (x, t) = u (x -
  \lambda_{\ast} t, t)$. To get the equation on $\rho$, we take $\partial_t +
  (U (\rho_0) - \rho_0 h' (\rho_0)) \partial_x + \frac{1}{\tau}$ of the first
  equation and the result follows.
\end{proof}

\subsection{Proof of Proposition \ref{constlinstab}}

\begin{proof}
  From Lemma \ref{lineq31}, the function
  \[ v (x, t) = u (x - \lambda_{\ast} t, t) \]
  satisfies
  \[ \partial_t^2 v + (\lambda_1^0 + \lambda^0_2) \partial_{x t}^2 v +
     \lambda_1^0 \lambda^0_2 \partial_x^2 v + \frac{1}{\tau} \partial_t v = 0,
  \]
  where $\lambda_1^0 = \rho_0 U_f > 0, \lambda^0_2 = \rho_0 U_f - \rho_0 h'
  (\rho_0) < 0$ since $s_{\tmop{cc}} (\rho_0) = h' (\rho_0) - U_f > 0$. We
  have
  \[ v_0 = v_{| t = 0 \nobracket} = u_{| t = 0 \nobracket}, v_1 = \partial_t
     v_{| t = 0 \nobracket} = (\partial_t u + (U (\rho_0) - \rho_0 U_f)
     \partial_x u)_{| t = 0 \nobracket} . \]
  From Proposition \ref{dwesol}, we deduce that
  \begin{eqnarray*}
    v (x, t) & = & \int_{\lambda^0_2 t}^{\lambda_1^0 t} V (y, t) \left(
    \frac{v_0}{\tau} + (\lambda_1 + \lambda_2) v_0' + v_1 \right) (x - y) d
    y\\
    & + & \int_{\lambda^0_2 t}^{\lambda_1^0 t} \partial_t V (y, t) v_0 (x -
    y) d y.\\
    & + & \lambda_1^0 \frac{e^{\frac{- \lambda_1^0 t}{\tau (\lambda_1^0 -
    \lambda^0_2)}}}{\lambda_1^0 - \lambda^0_2} v_0 (x - \lambda_1^0 t) -
    \lambda^0_2 \frac{e^{\frac{\lambda^0_2 t}{\tau (\lambda_1^0 -
    \lambda^0_2)}}}{\lambda_1^0 - \lambda^0_2} v_0 (x - \lambda^0_2 t),
  \end{eqnarray*}
  where
  \[ V (y, t) = \frac{e^{- \frac{2}{\tau (\lambda_1^0 - \lambda^0_2)^2} \left(
     - \lambda_1^0 \lambda^0_2 t + \frac{\lambda_1^0 + \lambda^0_2}{2} y
     \right)}}{\lambda_1^0 - \lambda^0_2} I_0 \left( \frac{2 \sqrt{-
     \lambda_1^0 \lambda^0_2}}{\tau (\lambda_1^0 - \lambda^0_2)^2} \sqrt{- (y
     - \lambda_1^0 t) (y - \lambda^0_2 t)} \right) . \]
  We check easily that this quantity is well defined and $C^2$ with respect
  both to the position and derivative, and thus is the classical solution of
  (\ref{linconst}).
  
  We recall from Lemma \ref{Vest} that there exists $a_0 > 0$ (depending on
  $\lambda_1^0, \lambda^0_2, \tau$) such that for all $k, n \in \mathbb{N}$,
  there exists $C_{k, n} (\lambda_1^0, \lambda^0_2, \tau) > 0$ such that for
  $t \geqslant 0, y \in [\lambda_2 t, \lambda_1 t]$,
  \begin{equation}
    | \partial_x^k \partial_t^n V (y, t) | \leqslant \frac{C_{k, n}
    (\lambda_1^0, \lambda^0_2, \tau) e^{- a_0 \frac{y^2}{1 + t}}}{(1 +
    t)^{\frac{1}{2} + \frac{k}{2} + n}} . \label{eq318v4}
  \end{equation}
  We recall also the notation
  \[ F_{a, \gamma, w}^{\lambda_1^0, \lambda^0_2} (x, t) = \frac{1}{(1 +
     t)^{\gamma}} \int_{\lambda^0_2 t}^{\lambda_1^0 t} e^{- a \frac{y^2}{1 +
     t}} w (x - y) d y. \]
  We first estimate with (\ref{eq318v4}) that
  \begin{eqnarray}
    &  & \left| \int_{\lambda^0_2 t}^{\lambda_1^0 t} V (y, t) v_0 (x - y) d y
    \right| \nonumber\\
    & \leqslant & \frac{K (\lambda_1^0, \lambda^0_2, \tau)}{(1 +
    t)^{\frac{1}{2}}} \int_{\lambda^0_2 t}^{\lambda_1^0 t} e^{- a_0
    \frac{y^2}{1 + t}} | v_0 (x - y) | d y \nonumber\\
    & \leqslant & K (\lambda_1^0, \lambda^0_2, \tau) F_{a_0, \frac{1}{2}, |
    v_0 |}^{\lambda_1^0, \lambda^0_2} (x, t) . 
  \end{eqnarray}
  We check easily that there exists $b_1 > 0$ depending on $\lambda_1^0,
  \lambda^0_2, \tau$ such that
  \begin{eqnarray*}
    &  & \left| \lambda_1^0 \frac{e^{\frac{- \lambda_1^0 t}{\tau (\lambda_1^0
    - \lambda^0_2)}}}{\lambda_1^0 - \lambda^0_2} v_0 (x - \lambda_1^0 t) -
    \lambda_2 \frac{e^{\frac{\lambda^0_2 t}{\tau (\lambda_1^0 -
    \lambda^0_2)}}}{\lambda_1^0 - \lambda^0_2} v_0 (x - \lambda^0_2 t)
    \right|\\
    & \leqslant & C (\lambda_1^0, \lambda^0_2) e^{- b_1 t} (| v_0 (x -
    \lambda_1^0 t) | + | v_0 (x - \lambda^0_2 t) |) .
  \end{eqnarray*}
  Furthermore, by integration by parts,
  \begin{eqnarray}
    &  & \int_{\lambda^0_2 t}^{\lambda_1^0 t} V (y, t) v_0' (x - y) d y
    \nonumber\\
    & = & - \int_{\lambda^0_2 t}^{\lambda_1^0 t} V (y, t) \partial_y (v_0 (x
    - y)) d y \nonumber\\
    & = & - V (\lambda_1^0 t, t) v_0 (x - \lambda_1^0 t) + V (\lambda^0_2 t,
    t) v_0 (x - \lambda^0_2 t) \nonumber\\
    & + & \int_{\lambda^0_2 t}^{\lambda_1^0 t} \partial_y V (y, t) v_0 (x -
    y) d y.  \label{320v4}
  \end{eqnarray}
  From (\ref{eq318v4}) and Lemma \ref{Vest}, we check that there exists $b_2 >
  0$ depending on $\lambda_1^0, \lambda^0_2, \tau$ such that
  \begin{eqnarray*}
    &  & \left| \int_{\lambda^0_2 t}^{\lambda_1^0 t} V (y, t) v_0' (x - y) d
    y \right|\\
    & \leqslant & K (\lambda_1^0, \lambda^0_2, \tau) F_{a_0, \frac{1}{2}, |
    v_0 |}^{\lambda_1^0, \lambda_2} (x, t)\\
    & + & C (\lambda_1^0, \lambda^0_2) e^{- b_2 t} (| v_0 (x - \lambda_1^0 t)
    | + | v_0 (x - \lambda^0_2 t) |)
  \end{eqnarray*}
  for the same $a_0$ as equation (\ref{eq318v4}). With these estimate and
  integration by parts, we can check that for the initial datas
  \[ u_i = u_{| t = 0 \nobracket}, \rho_i = \rho_{| t = 0 \nobracket} \in
     C^j_{\tmop{loc}} (\mathbb{R}), \]
  we have from Lemma \ref{lineq31}, $\partial_t u + (U (\rho_0) - \rho_0 h'
  (\rho_0)) \partial_x u + \frac{1}{\tau} (U_f \rho + u) = 0$ (that we take at
  $t = 0$ to write $\partial_t u$ in terms of $\partial_x u, \rho, u$ at $t =
  0$), therefore there exists $b_0 > 0$ depending on $\lambda_1^0,
  \lambda^0_2, \tau$ such that we have
  \begin{eqnarray*}
    | v (x, t) | & \leqslant & K (\lambda_1^0, \lambda^0_2, \tau) F_{a_0,
    \frac{1}{2}, | u_i | + | \rho_i |}^{\lambda_1^0, \lambda^0_2} (x, t)\\
    & + & C (\lambda_1^0, \lambda^0_2) e^{- b_0 t} (| u_i (x - \lambda_1^0 t)
    | + | u_i (x - \lambda^0_2 t) |)\\
    & + & C (\lambda_1^0, \lambda^0_2) e^{- b_0 t} (| \rho_i (x - \lambda_1^0
    t) | + | \rho_i (x - \lambda^0_2 t) |) .
  \end{eqnarray*}
  For $n, k \in \mathbb{N}$, we now want to estimate $| \partial_t^n
  \partial_x^k v (x, t) |$. First, we check that
  \begin{eqnarray}
    &  & \left| \partial_t^n \partial_x^k \left( \lambda_1^0 \frac{e^{\frac{-
    \lambda_1^0 t}{\tau (\lambda_1^0 - \lambda^0_2)}}}{\lambda_1^0 -
    \lambda^0_2} u_i (x - \lambda_1^0 t) - \lambda_2 \frac{e^{\frac{\lambda_2
    t}{\tau (\lambda_1^0 - \lambda^0_2)}}}{\lambda_1^0 - \lambda^0_2} u_i (x -
    \lambda^0_2 t) \right) \right| \nonumber\\
    & \leqslant & K (n, \lambda_1^0, \lambda^0_2, \tau) e^{- b_1 t} \sum_{j =
    k}^{n + k} (| u_i^{(j)} (x - \lambda_1^0 t) | + | u_i^{(j)} (x -
    \lambda^0_2 t) |) .  \label{319}
  \end{eqnarray}
  Let us now look at
  \[ \partial_t^n \partial_x^k \left( \int_{\lambda^0_2 t}^{\lambda_1^0 t} V
     (y, t) u_i (x - y) d y \right) = \partial_t^n \left( \int_{\lambda^0_2
     t}^{\lambda_1^0 t} V (y, t) u^{(k)}_i (x - y) d y \right) . \]
  Either at least one derivative in time fall on the boundaries of the
  integral. In that case, the term coming from it can be estimated as
  (\ref{319}). If none does, then we are left with the term
  \[ \int_{\lambda^0_2 t}^{\lambda_1^0 t} \partial_t^n V (y, t) u^{(k)}_i (x -
     y) d y. \]
  On this term, we do integration by parts, in the spirit of (\ref{320v4}), to
  move the $k$ derivatives from $v_0$ to $V$. In the process, we create
  boundary terms that can be estimated as follow for some $b_2 > 0$:
  \begin{eqnarray}
    &  & \left| \int_{\lambda^0_2 t}^{\lambda_1^0 t} \partial_t^n V (y, t)
    u^{(k)}_i (x - y) d y - \int_{\lambda^0_2 t}^{\lambda_1^0 t} \partial_t^n
    \partial_y^k V (y, t) u^{}_i (x - y) d y \right| \nonumber\\
    & \leqslant & K (n, k, \lambda_1^0, \lambda^0_2, \tau) e^{- b_2 t}
    \sum_{j = 0}^{k - 1} (| u_i^{(j)} (x - \lambda_1^0 t) | + | u_i^{(j)} (x -
    \lambda_2 t) |) .  \label{320}
  \end{eqnarray}
  Finally, we estimate with Lemma \ref{Vest} that
  \begin{eqnarray*}
    &  & \left| \int_{\lambda^0_2 t}^{\lambda_1^0 t} \partial_t^n
    \partial_y^k V (y, t) u^{}_i (x - y) d y \right|\\
    & \leqslant & \frac{C_{k, n} (\lambda_1^0, \lambda^0_2, \tau)}{(1 +
    t)^{\frac{1}{2} + \frac{k}{2} + n}} \int_{\lambda^0_2 t}^{\lambda_1^0 t}
    e^{- a_0 \frac{y^2}{1 + t}} | u_i (x - y) | d y\\
    & \leqslant & C_{k, n} (\lambda_1^0, \lambda^0_2, \tau) F_{a_0,
    \frac{1}{2} + \frac{k}{2} + n, | u_i |}^{\lambda_1^0, \lambda^0_2} (x, t)
    .
  \end{eqnarray*}
  We can then estimate every other term by this methods. We put all
  derivatives on $V$, and the boundary terms this create can be estimated as
  in (\ref{320}). At the end, on $V$ we will have at least $k$ derivatives in
  position and $n$ in time, leading to the estimate on $\partial_t^n
  \partial_x^k v$.
  
  The estimate on $\partial_t^n \partial_x^k q$ follows from a similar proof,
  and this concludes the proof of Proposition \ref{constlinstab}.
\end{proof}

\section{Proof of the nonlinear stability (Theorem
\ref{NLconstlocTheorem})}\label{sec6vf}

\subsection{The equations of the nonlinear problem}

\begin{lemma}
  \label{lem32}Taking the problem (ARZ) for the variable $(\rho_0 + \rho, U
  (\rho_0) + u)$ yield the equations
  \begin{equation}
    \partial_t \left(\begin{array}{c}
      \rho\\
      u
    \end{array}\right) + A \partial_x \left(\begin{array}{c}
      \rho\\
      u
    \end{array}\right) + B \left(\begin{array}{c}
      \rho\\
      u
    \end{array}\right) = 0, \label{nleqconst}
  \end{equation}
  where
  \[ A = \left(\begin{array}{cc}
       U (\rho_0) + u & \rho_0 + \rho\\
       0 & U (\rho_0) + u - h' (\rho_0 + \rho) (\rho_0 + \rho)
     \end{array}\right) \]
  and
  \[ B = \frac{1}{\tau} \left(\begin{array}{cc}
       0 & 0\\
       U_f & 1
     \end{array}\right) . \]
  In particular, with $\lambda_{\ast} = U (\rho_0) - \rho_0 U_f$, the function
  \[ q (x, t) \assign \rho (x - \lambda_{\ast} t, t), v (x, t) \assign u (x -
     \lambda_{\ast} t, t) \]
  satisfies, with the nonlinear caracteristic speeds
  \[ \lambda_1 \assign \rho_0 U_f + v, \lambda_2 \assign \rho_0 U_f - (\rho_0
     + q) h' (\rho_0 + q) + v, \]
  the equations
  \begin{equation}
    \partial_t^2 v + (\lambda_1 + \lambda_2) \partial_{x t}^2 v + \lambda_1
    \lambda_2 \partial_x^2 v + \frac{1}{\tau} \partial_t v + \Omega_1
    \partial_x v = 0
  \end{equation}
  with
  \[ \Omega_1 \assign \partial_t \lambda_2 + \lambda_1 \partial_x \lambda_2 +
     \frac{1}{\tau} (v - U_f q) \]
  and
  \begin{equation}
    \partial_t^2 q + (\lambda_1 + \lambda_2) \partial_{x t}^2 q + \lambda_1
    \lambda_2 \partial_x^2 q + \left( \frac{1}{\tau} + \Omega_2 \right)
    \partial_t q + \Omega_3 \partial_x q = 0,
  \end{equation}
  with
  \[ \Omega_2 \assign \partial_x v, \Omega_3 \assign \frac{v - U_f q}{\tau} +
     \lambda_2 \partial_x v. \]
\end{lemma}

\

\begin{proof}
  We decompose $P = \rho_0 + \rho, V = U (\rho_0) + u$ in
  \[ \left\{\begin{array}{l}
       \partial_t P + \partial_x (P V) = 0\\
       \partial_t (V + h (P)) + V \partial_x (V + h (P)) = \frac{1}{\tau} (U
       (P) - V),
     \end{array}\right. \]
  where $(\rho_0, U (\rho_0))$ is a constant solution to $(\tmop{ARZ})$. We
  have
  \begin{eqnarray*}
    0 & = & \partial_t P + \partial_x (P V)\\
    & = & \partial_t \rho + (\rho_0 + \rho) \partial_x u + (U (\rho_0) + u)
    \partial_x \rho .
  \end{eqnarray*}
  Furthermore,
  \begin{eqnarray*}
    &  & \partial_t (V + h (P))\\
    & = & \partial_t u + h' (P) \partial_t \rho\\
    & = & \partial_t u + h' (P) (- (\rho_0 + \rho) \partial_x u - (U (\rho_0)
    + u) \partial_x \rho)
  \end{eqnarray*}
  and
  \[ V \partial_x (V + h (P)) = V \partial_x u + V h' (P) \partial_x \rho . \]
  Also, since $U (\rho) = U_f (1 - \rho)$,
  \begin{eqnarray*}
    - \frac{1}{\tau} (U (P) - V) & = & \frac{- 1}{\tau} (U (\rho_0 + \rho) - U
    (\rho_0) - u)\\
    & = & \frac{- 1}{\tau} (- U_f \rho - u),
  \end{eqnarray*}
  hence
  \begin{eqnarray}
    0 & = & \partial_t (V + h (P)) + V \partial_x (V + h (P)) - \frac{1}{\tau}
    (U (P) - V) \nonumber\\
    & = & \partial_t u + h' (P) (- (\rho_0 + \rho) \partial_x u - (U (\rho_0)
    + u) \partial_x \rho) \nonumber\\
    & + & V \partial_x u + V h' (P) \partial_x \rho \nonumber\\
    & + & \frac{1}{\tau} (U_f \rho + u) .  \label{311}
  \end{eqnarray}
  We conclude the proof of (\ref{nleqconst}) by remarking that $V h' (P)
  \partial_x \rho - (U (\rho_0) + u) h' (P) \partial_x \rho = 0.$ We deduce
  that
  \[ \left\{\begin{array}{l}
       \partial_t \rho + (U (\rho_0) + u) \partial_x \rho + (\rho_0 + \rho)
       \partial_x u = 0\\
       \partial_t u + (U (\rho_0) + u - h' (\rho_0 + \rho) (\rho_0 + \rho))
       \partial_x u + \frac{1}{\tau} (U_f \rho + u) = 0
     \end{array}\right. \]
  and therefore, if we define
  \[ v (x, t) \assign u (x - \lambda_{\ast} t, t) \]
  and
  \[ q (x, t) \assign \rho (x - \lambda_{\ast} t, t), \]
  we have the equations
  \begin{equation}
    \partial_t q + \lambda_1 \partial_x q + (\rho_0 + q) \partial_x v = 0
    \label{eqv1}
  \end{equation}
  and
  \begin{equation}
    \partial_t v + \lambda_2 \partial_x v + \frac{1}{\tau} (U_f q + v) = 0,
    \label{eqq1}
  \end{equation}
  where
  \[ \lambda_1 = \rho_0 U_f + v, \lambda_2 = \rho_0 U_f - (\rho_0 + q) h'
     (\rho_0 + q) + v. \]
  Taking $\partial_t + \lambda_1 \partial_x$ of equation (\ref{eqq1}), we
  prove equation (\ref{feqv}) with
  \[ \Omega_1 = \partial_t \lambda_2 + \lambda_1 \partial_x \lambda_2 +
     \frac{1}{\tau} (\lambda_1 - U_f (\rho_0 + q)) . \]
  Since $\lambda_1 = \rho_0 U_f + v$, we have
  \[ \Omega_1 = \partial_t \lambda_2 + \lambda_1 \partial_x \lambda_2 +
     \frac{1}{\tau} (v - U_f q) . \]
  Now, taking $\partial_t + \lambda_2 \partial_x$ of equation (\ref{eqv1}), we
  show equation (\ref{feqq}).
\end{proof}

\subsection{Proof of Theorem \ref{NLconstlocTheorem}}

\begin{proof}
  Let us give here an overview of the proof. In step 1 and 2 we recall briefly
  the equation satisfied by $q, v$ and the notations. In step 3 we set up the
  bootstrap. Step 4 to 6 are devoted respectfully to estimates on $v$, its
  first derivatives, then its second derivatives. We conclude in step 7 by
  explaining how to do similar estimates on $q$.
  
  \
  
  {\tmstrong{Step 1.}} Equation satisfied by $v (x, t) = u (x + (U (\rho_0) -
  \rho_0 U_f) t, t)$ and $q (x, t) = \rho (x + (U (\rho_0) - \rho_0 U_f) t,
  t)$.
  
  \
  
  We recall that here, $\rho_0$ is a constant. By Cauchy theory, we have the
  existence at least for small times of a solution of (\ref{maineq}). It
  exists in particular as long as $\rho, u$ and its derivatives in time are
  bounded.
  
  We recall that (\ref{maineq}) is equivalent to
  \[ \left\{\begin{array}{l}
       \partial_t \rho + (U (\rho_0) + u) \partial_x \rho + (\rho_0 + \rho)
       \partial_x u = 0\\
       \partial_t u + (U (\rho_0) + u - h' (\rho_0 + \rho) (\rho_0 + \rho))
       \partial_x u + \frac{1}{\tau} (U_f \rho + u) = 0
     \end{array}\right. \]
  if we write the solution $(\rho_0 + \rho, U (\rho_0) + u)$, and therefore,
  if we define
  \[ v (x, t) = u (x - \lambda_{\ast} t, t), q (x, t) = \rho (x -
     \lambda_{\ast} t, t), \]
  we have the equations
  \begin{equation}
    \partial_t q + \lambda_1 \partial_x q + (\rho_0 + q) \partial_x v = 0
    \label{eqv1bis}
  \end{equation}
  and
  \begin{equation}
    \partial_t v + \lambda_2 \partial_x v + \frac{1}{\tau} (U_f q + v) = 0,
    \label{eqq1bis}
  \end{equation}
  where
  \[ \lambda_1 = \rho_0 U_f + v, \lambda_2 = \rho_0 U_f - (\rho_0 + q) h'
     (\rho_0 + q) + v. \]
  We define
  \[ \lambda_1^0 = \rho_0 U_f > 0, \lambda_2^0 = \rho_0 U_f - \rho_0 h'
     (\rho_0) < 0. \]
  By Lemma \ref{lem32}, we have
  \begin{equation}
    \partial_t^2 v + (\lambda_1 + \lambda_2) \partial_{x t}^2 v + \lambda_1
    \lambda_2 \partial_x^2 v + \frac{1}{\tau} \partial_t v + \Omega_1
    \partial_x v = 0, \label{feqv}
  \end{equation}
  where
  \[ \Omega_1 = \partial_t \lambda_2 + \lambda_1 \partial_x \lambda_2 +
     \frac{1}{\tau} (\lambda_1 - U_f (\rho_0 + q)) . \]
  Similarly,
  \begin{equation}
    \partial_t^2 q + (\lambda_1 + \lambda_2) \partial_{x t}^2 q + \lambda_1
    \lambda_2 \partial_x^2 q + \left( \frac{1}{\tau} + \Omega_2 \right)
    \partial_t q + \Omega_3 \partial_x q = 0, \label{feqq}
  \end{equation}
  where
  \[ \Omega_2 = \partial_x v, \Omega_3 = \frac{v - U_f q}{\tau} + \lambda_2
     \partial_x v. \]

  {\tmstrong{{\tmstrong{Step 2}}.}} Linear and nonlinear parts of the equation
  on $v$.
  
  \
  
  We decompose equation (\ref{feqv}) in
  \begin{eqnarray*}
    &  & \partial_t^2 v + (\lambda^0_1 + \lambda^0_2) \partial_{x t}^2 v +
    \lambda^0_1 \lambda^0_2 \partial_x^2 v + \frac{1}{\tau} \partial_t v\\
    & = & - \Omega_1 \partial_x v - (\lambda_1 - \lambda_1^0 + \lambda_2 -
    \lambda_2^0) \partial_{x t}^2 v - (\lambda_1 \lambda_2 - \lambda_1^0
    \lambda_2^0) \partial_x^2 v
  \end{eqnarray*}
  with $\lambda_1^0 = \rho_0 U_f, \lambda_2^0 = \rho_0 U_f - \rho_0 h'
  (\rho_0)$ which are constants. We note $v_0 = v_{| t = 0 \nobracket} = u_{|
  t = 0 \nobracket}$ and $v_1 = \partial_t v_{| t = 0 \nobracket} = \left(
  (\nobracket - U_f \rho_0 - u + (\rho + \rho_0) h' \nobracket (\rho_0 +
  \rho)) \partial_x u - \frac{1}{\tau} (U_f \rho + u) \right)_{| t = 0
  \nobracket}$.
  
  \
  
  We define the linear part of the solution, $v_L$, as the solution to
  \[ \left\{\begin{array}{l}
       \partial_t^2 v_L + (\lambda^0_1 + \lambda^0_2) \partial_{x t}^2 v_L +
       \lambda^0_1 \lambda^0_2 \partial_x^2 v_L + \frac{1}{\tau} \partial_t
       v_L = 0\\
       v_{L | t = 0 \nobracket} = v_0, \partial_t v_{L | t = 0 \nobracket} =
       v_1 .
     \end{array}\right. \]
  Estimates on $v_L$ follows from Proposition \ref{constlinstab}: there exists
  $C_0, a_1, b_1 > 0$ (depending on $\lambda^0_1, \lambda_2^0, \tau$) such
  that for $n + k \leqslant 2$, we have
  \begin{eqnarray*}
    &  & | \partial_t^n \partial_x^k v_L (x, t) |\\
    & \leqslant & \frac{C_0}{(1 + t)^{\frac{1}{2} + \frac{k}{2} + n}}
    \int_{\lambda_2^0 t}^{\lambda_1^0 t} e^{- a_1 \frac{y^2}{1 + t}} (| u_i (x
    - y) | + | \rho_i (x - y) |) d y\\
    & + & C_0 e^{- b_1 t} \left( \sum_{j = 0}^{n + k} (| \rho_i^{(j)} | + |
    u_i^{(j)} |) (x - \lambda^0_1 t) + (| \rho_i^{(j)} | + | u_i^{(j)} |) (x -
    \lambda_2^0 t) \right)
  \end{eqnarray*}
  We define
  \[ S \assign - \Omega_1 \partial_x v - (\lambda_1 - \lambda_1^0 + \lambda_2
     - \lambda_2^0) \partial_{x t}^2 v - (\lambda_1 \lambda_2 - \lambda_1^0
     \lambda_2^0) \partial_x^2 v. \]
  From Proposition \ref{dwesol}, we can write $v = v_L + v_{\tmop{NL}}$ where
  \[ v_{\tmop{NL}} (x, t) \assign \int_0^t \int_{\lambda_2^0 (t -
     s)}^{\lambda^0_1 (t - s)} V (y, t - s) S (x - y, s) d y d s. \]

  {\tmstrong{{\tmstrong{Step 3}}.}} Notations and setting up the bootstrap.

  We recall the notation
  \[ F_{a, \gamma, w}^{\lambda^{}_1, \lambda_2} (x, t) = \frac{1}{(1 +
     t)^{\gamma}} \int_{\lambda_2 t}^{\lambda_1 t} e^{- a \frac{y^2}{1 + t}} w
     (x - y) d y. \]
  we define the function
  \[ w_i \assign \sum_{j = 0}^2 | \rho_i^{(j)} | + | u_i^{(j)} | \]
  and, to simplify the notations, we define for $a > 0, \gamma \in \mathbb{R}$
  the quantity
  \[ F_{a, \gamma, \delta} (x, t) \assign F_{a, \gamma, w_i}^{\lambda_1^0 +
     \delta, \lambda_2^0 - \delta} (x, t) = \frac{1}{(1 + t)^{\gamma}}
     \int_{(\lambda_2^0 - \delta) t}^{(\lambda_1^0 + \delta) t} e^{- a
     \frac{y^2}{1 + t}} w_i (x - y) d y \]
  for some small $\delta > 0$ that will be fixed later on. We also define for
  some $b, \delta > 0$ the quantity
  \[ G_{b, \delta, w}^{\mu} (x, t) = e^{- b t} \sup_{y \in [- \delta t, \delta
     t]} w (x - \mu t + y) \]
  and
  \[ G_{b, \delta} (x, t) \assign G_{b, \delta, w_i}^{\lambda_1^0} (x, t) +
     G_{b, \delta, w_i}^{\lambda_2^0} (x, t) . \]
  Our initial estimate on $v_L$ can be translated using $F_{a, \gamma,
  \delta}$ and $G_{b, \delta}$: There exists $C_L, a_L, b_L > 0$ depending on
  $\lambda_1^0, \lambda_2^0, \tau$ such that for $n + k \leqslant 2$,
  \[ | \partial_t^n \partial_x^k v_L (x, t) | \leqslant C_L \left( F_{a_L,
     \frac{1}{2} + \frac{k}{2} + n, 0} + G_{b_L, 0} \right) (x, t) \]
  for all $x \in \mathbb{R}, t \geqslant 0$. Finally, for small $\nu > 0$
  ($\nu < \frac{1}{8}$ will be needed in the proof) and $k + n \leqslant 2$,
  we define $\gamma_{\nu} (k, n)$ by
  \[ \gamma_{\nu} (0, 0) = \frac{1}{2}, \gamma_{\nu} (1, 0) = 1 - \nu,
     \gamma_{\nu} (0, 1) = \frac{3}{2} - \nu, \gamma_{\nu} (k, n) = 1 - \nu \]
  for $k + n = 2$. Remark that in all cases,
  \[ \gamma_{\nu} (k, n) \leqslant \frac{1}{2} + \frac{k}{2} + n. \]
  We want to bootstrap the following results: there exists $a, b, \delta > 0$
  such that for any fixed small $\nu > 0$, there exists $C_0 (\nu) > 0$ such
  that for $k + n \leqslant 1$,
  \[ | \partial_t^n \partial_x^k v (x, t) | + | \partial_t^n \partial_x^k q
     (x, t) | \leqslant C_0 (\nu) (F_{a, \gamma_{\nu} (k, n), \delta} + G_{b,
     \delta}) (x, t), \]
  and for $k + n = 2$, for some $a_2, b_2, \delta_2 > 0$,
  \[ | \partial_t^n \partial_x^k v (x, t) | + | \partial_t^n \partial_x^k q
     (x, t) | \leqslant C_0 (\nu) (F_{a_2, \gamma_{\nu} (k, n), \delta_2} +
     G_{b_2, \delta_2}) (x, t) . \]
  In the bootstrap, the coefficient $a_2, b_2, \delta_2$ are not the same as
  $a, b, \delta$ (in face $a_2, b_2$ will be respecfully smaller than $a, b$).
  In the computations, we will first take $b$ small enough to beat some
  constants depending on $\rho_0, U_f, h$, then $a, a_2, b_2$ will be taken
  small enough to beat constants depending on $b, \rho_0, U_f, h$.
  
  We denote by $T^{\ast} \geqslant 0$ the maximum time such that the solution
  exists in the classical sense (that is it is a $C^2$ function) and these
  estimates hold for given $C_0 (\nu), a, b, a_2, b_2, \delta$ and $\delta_2$.
  
  \
  
  First, from standard local Cauchy theory (see for instance
  {\cite{MR1816648}}), there exists $\tau_0^{\ast} > 0$ such that the solution
  exists and is $C^2$ on $[0, \tau_0^{\ast}]$. Furthermore, still from
  {\cite{MR1816648}}, if the estimates hold on $[0, T^{\ast}]$, then there
  exists $\tau^{\ast} > 0$ such that the solution exists and is $C^2$ on $[0,
  T^{\ast} + \tau^{\ast}]$. The fact that $T^{\ast} > 0$ if we take $C_0
  (\nu)$ large enough is a consequence of the standard estimates from Kato
  (see {\cite{MR390516}}).
  
  \
  
  Our goal is to show that $T^{\ast} = + \infty$. We suppose that \ $T^{\ast}
  < + \infty$. From subsection \ref{ss33vf}, we check that on $[0, T^{\ast}]$,
  for $C_0 (\nu)$ large enough,

  \begin{equation}
    \| \partial_t^n \partial_x^k v (., t) \|_{L^{\infty} (\mathbb{R})} + \|
    \partial_t^n \partial_x^k q (., t) \|_{L^{\infty} (\mathbb{R})} \leqslant
    C_0 (\nu) \frac{\| w_i \|_{L^1 (\mathbb{R})} + \| w_i \|_{L^{\infty}
    (\mathbb{R})}}{(1 + t)^{\gamma_{\nu} (k, n)}} \label{617vf}
  \end{equation}
  and
  \begin{equation}
    \| \partial_t^n \partial_x^k v (., t) \|_{L^1 (\mathbb{R})} + \|
    \partial_t^n \partial_x^k q (., t) \|_{L^1 (\mathbb{R})} \leqslant C_0
    (\nu) \frac{\| w_i \|_{L^1 (\mathbb{R})}}{(1 + t)^{\gamma_{\nu} (k, n) -
    \frac{1}{2}}} . \label{618vf}
  \end{equation}
  We recall that $\| w_i \|_{L^1 (\mathbb{R})} + \| w_i \|_{L^{\infty}
  (\mathbb{R})} \leqslant \varepsilon \ll 1$ by hypothesis. In particular, for
  $\varepsilon$ small enough, we have $\rho_0 + \rho \in] 0, 1 [, \lambda_1 >
  0, \lambda_2 < 0$ on $[0, T^{\ast}]$.
  
  \
  
  {\tmstrong{Step 4.}} Estimates on $v$
  
  \
  
  We recall from step 2. that $v = v_L + v_{\tmop{NL}}$ with
  \[ v_{\tmop{NL}} (x, t) = \int_0^t \int_{\lambda_2^0 (t - s)}^{\lambda^0_1
     (t - s)} V (y, t - s) S (x - y, s) d y d s \]
  and $v_L$ satisfies
  \[ | v_L (x, t) | \leqslant C_L \left( F_{a_1, \frac{1}{2}, 0} + G_{b_1, 0}
     \right) (x, t) \]
  for some fixed values of $C_L, a_1, b_1$ for all positive times. We will
  decompose $S$ in a particular form for this estimate. We recall that
  \[ S = - \Omega_1 \partial_x v - (\lambda_1 - \lambda_1^0 + \lambda_2 -
     \lambda_2^0) \partial_{x t}^2 v - (\lambda_1 \lambda_2 - \lambda_1^0
     \lambda_2^0) \partial_x^2 v. \]
  Here, the difficulty is to keep the decay in time and have an estimate with
  $F_{a_1, \frac{1}{2}, \delta}$. It would be easy to have the result with
  $F_{a_1, \frac{1}{2} - \nu, \delta}$ instead (by applying Lemma
  \ref{L34v4}), but we need some specific computation to not lose the factor
  $t^{\nu}$.
  
  \
  
  We start with the estimate of
  \[ E_1 \assign \int_0^t \int_{\lambda_2^0 (t - s)}^{\lambda^0_1 (t - s)} -
     V (y, t - s) ((\lambda_1 \lambda_2 - \lambda_1^0 \lambda_2^0)
     \partial_x^2 v) (x - y, s) d y d s, \]
  one of the terms appearing in $v_{\tmop{NL}}$. We have that
  \[ \partial_x^2 v (x - y, s) = - \partial_y (\partial_x v (x - y, s)), \]
  therefore, by inegration by parts,
  \begin{eqnarray*}
    &  & E_1 (x, t)\\
    & = & \int_0^t V (\lambda^0_1 (t - s), t - s) ((\lambda_1 \lambda_2 -
    \lambda_1^0 \lambda_2^0) \partial_x v) (x - \lambda^0_1 (t - s), s) d s\\
    & - & \int_0^t V (\lambda^0_2 (t - s), t - s) ((\lambda_1 \lambda_2 -
    \lambda_1^0 \lambda_2^0) \partial_x v) (x - \lambda^0_2 (t - s), s) d s\\
    & - & \int_0^t \int_{\lambda_2^0 (t - s)}^{\lambda^0_1 (t - s)}
    \partial_y V (y, t - s) ((\lambda_1 \lambda_2 - \lambda_1^0 \lambda_2^0)
    \partial_x v) (x - y, s) d y d s\\
    & - & \int_0^t \int_{\lambda_2^0 (t - s)}^{\lambda^0_1 (t - s)} V (y, t -
    s) (\partial_y (\lambda_1 \lambda_2 - \lambda_1^0 \lambda_2^0) \partial_x
    v) (x - y, s) d y d s.
  \end{eqnarray*}
  We have
  \[ \lambda_1 \lambda_2 - \lambda_1^0 \lambda_2^0 = (\lambda_1 - \lambda_1^0)
     \lambda_2 + \lambda_1^0 (\lambda_2 - \lambda_2^0), \]
  and
  \[ \lambda_1 - \lambda_1^0 = v, \lambda_2 - \lambda_2^0 = \rho_0 h' (\rho_0)
     - (\rho_0 + q) h' (\rho_0 + q) + v. \]
  Therefore on $[0, T^{\ast}]$, by equation (\ref{617vf}), we have
  \begin{equation}
    | (\lambda_1 \lambda_2 - \lambda_1^0 \lambda_2^0) (x - \lambda^0_1 (t -
    s), s) | \leqslant K C_0 (\nu) \varepsilon . \label{eq3288}
  \end{equation}
  This implies that on $[0, T^{\ast}]$,
  \begin{eqnarray*}
    &  & \left| \int_0^t V (\lambda^0_1 (t - s), t - s) ((\lambda_1 \lambda_2
    - \lambda_1^0 \lambda_2^0) \partial_x v) (x - \lambda^0_1 (t - s), s) d s
    \right|\\
    & \leqslant & K C_0 (\nu) \varepsilon \int_0^t | V (\lambda^0_1 (t - s),
    t - s) \partial_x v (x - \lambda^0_1 (t - s), s) | d s\\
    & \leqslant & K C_0 (\nu)^2 \varepsilon \int_0^t | V (\lambda^0_1 (t -
    s), t - s) F_{a, \gamma_{\nu} (1, 0), \delta} (x - \lambda^0_1 (t - s), s)
    | d s\\
    & + & K C_0 (\nu)^2 \varepsilon \int_0^t | V (\lambda^0_1 (t - s), t - s)
    G_{b, \delta} (x - \lambda^0_1 (t - s), s) | d s.
  \end{eqnarray*}
  From Lemmas \ref{Vest} and \ref{L33b4}, we have (if $a > 0$ is small enough)
  \begin{eqnarray*}
    &  & \int_0^t | V (\lambda^0_1 (t - s), t - s) F_{a, \gamma_{\nu} (1, 0),
    \delta} (x - \lambda^0_1 (t - s), s) | d s\\
    & \leqslant & K F_{a, \gamma_{\nu} (1, 0), \delta} (x, t)\\
    & \leqslant & K F_{a, \frac{1}{2}, \delta} (x, t) .
  \end{eqnarray*}
  From Lemma \ref{L37b4}, we have (if $a$ is small enough, depending on $b$)
  \begin{eqnarray*}
    &  & \int_0^t | V (\lambda^0_1 (t - s), t - s) G_{b, \delta} (x -
    \lambda^0_1 (t - s), s) | d s\\
    & \leqslant & K \left( G_{b, \delta} (x, t) + F_{a, \frac{1}{2}, \delta}
    (x, t) \right) .
  \end{eqnarray*}
  We deduce that
  \begin{eqnarray*}
    &  & \left| \int_0^t V (\lambda^0_1 (t - s), t - s) ((\lambda_1 \lambda_2
    - \lambda_1^0 \lambda_2^0) \partial_x v) (x - \lambda^0_1 (t - s), s) d s
    \right|\\
    & \leqslant & K C_0 (\nu)^2 \varepsilon \left( G_{b, \delta} (x, t) +
    F_{a, \frac{1}{2}, \delta} (x, t) \right) .
  \end{eqnarray*}
  The same estimate holds for the second line of $E_1$. For the third line,
  using for $s \in [0, T^{\ast}]$ that
  \[ | (\lambda_1 \lambda_2 - \lambda_1^0 \lambda_2^0) (x - y, s) | \leqslant
     \frac{K C_0 (\nu) \varepsilon}{(1 + s)^{1 / 2}}, \]
  we have (with Lemma \ref{Vest})
  \begin{eqnarray*}
    &  & \left| \int_0^t \int_{\lambda_2^0 (t - s)}^{\lambda^0_1 (t - s)}
    \partial_y V (y, t - s) ((\lambda_1 \lambda_2 - \lambda_1^0 \lambda_2^0)
    \partial_x v) (x - y, s) d y d s \right|\\
    & \leqslant & K C_0 (\nu)^2 \varepsilon \int_0^t \int_{\lambda_2^0 (t -
    s)}^{\lambda^0_1 (t - s)} \frac{e^{- a_0 \frac{y^2}{1 + t - s}}}{(1 + t -
    s)} F_{a, \gamma_{\nu} (1, 0) + \frac{1}{2}, \delta} (x - y, s) d y d s\\
    & + & K C_0 (\nu)^2 \varepsilon \int_0^t \int_{\lambda_2^0 (t -
    s)}^{\lambda^0_1 (t - s)} \frac{e^{- a_0 \frac{y^2}{1 + t - s}}}{(1 + t -
    s)} G_{b, \delta} (x - y, s) d y d s.
  \end{eqnarray*}
  From Lemma \ref{L34v4}, given that $a$ is small enough, we have
  \begin{eqnarray*}
    &  & \int_0^t \int_{\lambda_2^0 (t - s)}^{\lambda^0_1 (t - s)} \frac{e^{-
    a_0 \frac{y^2}{1 + t - s}}}{(1 + t - s)} F_{a, \gamma_{\nu} (1, 0) +
    \frac{1}{2}, \delta} (x - y, s) d y d s\\
    & \leqslant & K F_{a, 1 - \nu, \delta} (x, t)\\
    & \leqslant & K F_{a, \frac{1}{2}, \delta} (x, t),
  \end{eqnarray*}
  and using Lemma \ref{L36v4}, for $a$ small enough
  \begin{eqnarray*}
    &  & \int_0^t \int_{\lambda_2^0 (t - s)}^{\lambda^0_1 (t - s)} \frac{e^{-
    a_0 \frac{y^2}{1 + t - s}}}{(1 + t - s)} G_{b, \delta} (x, t) d y d s\\
    & \leqslant & K F_{a, 1, \delta} (x, t)\\
    & \leqslant & K F_{a, \frac{1}{2}, \delta} (x, t) .
  \end{eqnarray*}
  We deduce that
  \begin{eqnarray*}
    &  & \left| \int_0^t \int_{\lambda_2^0 (t - s)}^{\lambda^0_1 (t - s)}
    \partial_y V (y, t - s) ((\lambda_1 \lambda_2 - \lambda_1^0 \lambda_2^0)
    \partial_x v) (x - y, s) d y d s \right|\\
    & \leqslant & K C_0 (\nu)^2 \varepsilon F_{a, \frac{1}{2}, \delta} (x, t)
    .
  \end{eqnarray*}
  A similar computation give the same estimate on the fourth line of $E_1$.
  This concludes the proof of
  \[ | E_1 (x, t) | \leqslant K C_0 (\nu)^2 \varepsilon \left( F_{a,
     \frac{1}{2}, \delta} (x, t) + G_{b, \delta} (x, t) \right) . \]
  We can do a similar decomposition and estimation for
  \[ E_2 (x, t) = \int_0^t \int_{\lambda_2^0 (t - s)}^{\lambda^0_1 (t - s)} -
     V (y, t - s) ((\lambda_1 - \lambda_1^0 + \lambda_2 - \lambda_2^0)
     \partial_{x t}^2 v) (x - y, s) d y d s, \]
  leading to
  \[ | E_2 (x, t) | \leqslant K C_0 (\nu)^2 \varepsilon \left( F_{a,
     \frac{1}{2}, \delta} (x, t) + G_{b, \delta} (x, t) \right) . \]
  Finally, we look at
  \[ E_3 (x, t) = \int_0^t \int_{\lambda_2^0 (t - s)}^{\lambda^0_1 (t - s)} V
     (y, t - s) (\Omega_1 \partial_x v) (x - y, s) d y d s, \]
  where (from step 1)
  \[ \Omega_1 = \partial_t \lambda_2 + \lambda_1 \partial_x \lambda_2 +
     \frac{1}{\tau} (v - U_f q) . \]
  Simply using Lemma \ref{L34v4} at this stage would lead to an estimate with
  $F_{a, \frac{1}{2} - \nu, \delta}$ instead of $F_{a, \frac{1}{2}, \delta}$.
  It is critical to get $F_{a, \frac{1}{2}, \delta}$ here to make the
  bootstrap work. We had some margin in $E_1$ and $E_2$ but it will not be the
  case here. From step 1, we recall that
  \[ \partial_t v + \lambda_2 \partial_x v + \frac{1}{\tau} (U_f q + v) = 0,
  \]
  which implies that on $[0, T^{\ast}]$ we have
  \[ | (U_f q + v) (x, t) | \leqslant K C_0 (\nu) (G_{b, \delta} (x, t) +
     F_{a, 1 - \nu, \delta} (x, t)) . \]
  Remark that this is better than estimates on $q$ or $v$ individually. We
  decompose
  \[ \Omega_1 = \frac{2}{\tau} v + \partial_t \lambda_2 + \lambda_1 \partial_x
     \lambda_2 - \frac{1}{\tau} (U_f q + v) . \]
  Remark that
  \[ \left| \partial_t \lambda_2 + \lambda_1 \partial_x \lambda_2 -
     \frac{1}{\tau} (U_f q + v) \right| (x, t) \leqslant \frac{K C_0 (\nu)
     \varepsilon}{(1 + t)^{\gamma_{\nu} (1, 0)}} \]
  on $[0, T^{\ast}]$. Therefore, using Lemmas \ref{L34v4} and \ref{L36v4}, we
  estimate
  \begin{eqnarray*}
    &  & \left| \int_0^t \int_{\lambda_2^0 (t - s)}^{\lambda^0_1 (t - s)} V
    (y, t - s) \left( \left( \partial_t \lambda_2 + \lambda_1 \partial_x
    \lambda_2 - \frac{1}{\tau} (U_f q + v) \right) \partial_x v \right) (x -
    y, s) d y d s \right|\\
    & \leqslant & K C_0 (\nu) \varepsilon \int_0^t \int_{\lambda_2^0 (t -
    s)}^{\lambda^0_1 (t - s)} | V (y, t - s) | (F_{a, 2 \gamma_{\nu} (1, 0),
    \delta} + G_{b, \delta}) (x - y, s) d y d s\\
    & \leqslant & K C_0 (\nu) \varepsilon \left( F_{a, \frac{1}{2}, \delta}
    (x, t) + G_{b, \delta} (x, t) \right) .
  \end{eqnarray*}
  Finally, by integration by parts,
  \begin{eqnarray*}
    &  & \int_0^t \int_{\lambda_2^0 (t - s)}^{\lambda^0_1 (t - s)} V (y, t -
    s) \left( \frac{2}{\tau} v \partial_x v \right) (x - y, s) d y d s\\
    & = & \frac{- 1}{\tau} \int_0^t \int_{\lambda_2^0 (t - s)}^{\lambda^0_1
    (t - s)} V (y, t - s) \partial_y (v^2 (x - y, s)) d y d s\\
    & = & \frac{- 1}{\tau} \int_0^t V (\lambda^0_1 (t - s), t - s) v^2 (x -
    \lambda^0_1 (t - s), s) d s\\
    & + & \frac{1}{\tau} \int_0^t V (\lambda_2^0 (t - s), t - s) v^2 (x -
    \lambda_2^0 (t - s), s) d s\\
    & + & \frac{1}{\tau} \int_0^t \int_{\lambda_2^0 (t - s)}^{\lambda^0_1 (t
    - s)} \partial_y V (y, t - s) v^2 (x - y, s) d y d s.
  \end{eqnarray*}
  These three terms can be estimated as previously. We deduce that
  \[ | v_{\tmop{NL}} (x, t) | \leqslant | E_1 (x, t) | + | E_2 (x, t) | + |
     E_3 (x, t) | \leqslant K C_0 (\nu)^2 \varepsilon \left( F_{a,
     \frac{1}{2}, \delta} (x, t) + G_{b, \delta} (x, t) \right), \]
  and therefore, for $a < a_L, b < b_L$,
  \[ | v (x, t) | \leqslant | v_L (x, t) | + | v_{\tmop{NL}} (x, t) |
     \leqslant (C_L + K C_0 (\nu)^2 \varepsilon) \left( F_{a, \frac{1}{2},
     \delta} + G_{b, \delta} \right) (x, t) . \]
  Taking $\varepsilon$ small enough (depending on $\nu$), we deduce that on
  $[0, T^{\ast}]$, we have
  \[ | v (x, t) | \leqslant \frac{C_0 (\nu)}{2} \left( F_{a, \frac{1}{2},
     \delta} + G_{b, \delta} \right) (x, t) . \]
  Therefore $v$ is not the one responsible for the fact that $T^{\ast} \neq +
  \infty$.
  
  \
  
  {\tmstrong{Step 5.}} Estimates on $\partial_t v$ and $\partial_x v$
  
  \
  
  We compute that
  \begin{eqnarray*}
    \partial_x v_{\tmop{NL}} (x, t) & = & \int_0^t \int_{\lambda_2^0 (t -
    s)}^{\lambda^0_1 (t - s)} V (y, t - s) (- \partial_y) (S (x - y, s)) d y d
    s\\
    & = & - \int_0^t V (\lambda^0_1 (t - s), t - s) S (x - \lambda^0_1 (t -
    s), s) d s\\
    & + & \int_0^t V (\lambda^0_2 (t - s), t - s) S (x - \lambda^0_2 (t - s),
    s) d s\\
    & + & \int_0^t \int_{\lambda_2^0 (t - s)}^{\lambda^0_1 (t - s)}
    \partial_y V (y, t - s) S (x - y, s) d y d s.
  \end{eqnarray*}
  We estimate as previously that on $[0, T^{\ast}]$, estimating the second
  derivatives of $v$ and $q$ in $L^{\infty}$ by (\ref{617vf}), that
  \begin{eqnarray*}
    | S (x, t) | & \leqslant & | \Omega_1 \partial_x v | + | (\lambda_1 -
    \lambda_1^0 + \lambda_2 - \lambda_2^0) \partial_{x t}^2 v | + | (\lambda_1
    \lambda_2 - \lambda_1^0 \lambda_2^0) \partial_x^2 v |\\
    & \leqslant & \| \partial_x v \|_{L^{\infty}_x} | \Omega_1 | + \|
    \partial^2_{x t} v \|_{L^{\infty}_x} | (\lambda_1 - \lambda_1^0 +
    \lambda_2 - \lambda_2^0) |\\
    & + & \| \partial^2_x v \|_{L^{\infty}_x} | (\lambda_1 \lambda_2 -
    \lambda_1^0 \lambda_2^0) |\\
    & \leqslant & K C_0 (\nu)^2 \varepsilon \left( F_{a, \frac{3}{2} - \nu,
    \delta} + G_{b, \delta} \right) (x, t) .
  \end{eqnarray*}
  We deduce that on $[0, T^{\ast}]$,
  \[ | \partial_x v_{\tmop{NL}} (x, t) | \leqslant K C_0 (\nu)^2 \varepsilon
     (F_{a, \gamma_{\nu} (1, 0), \delta} + G_{b, \delta}) (x, t) . \]
  Since
  \[ | \partial_x v_L (x, t) | \leqslant C_L (F_{a_L, 1, 0} + G_{b_L, 0}) (x,
     t) \leqslant C_L (F_{a, \gamma_{\nu} (1, 0), \delta} + G_{b, \delta}) (x,
     t), \]
  we deduce that, taking $\varepsilon$ small enough (depending on $\nu$), on
  $[0, T^{\ast}]$ we have
  \[ | \partial_x v (x, t) | \leqslant \frac{C_0 (\nu)}{2} (F_{a, \gamma_{\nu}
     (1, 0), \delta} + G_{b, \delta}) (x, t) . \]
  Also,
  \begin{eqnarray*}
    \partial_t v_{\tmop{NL}} (x, t) & = & \lambda_1^0 \int_0^t V (\lambda^0_1
    (t - s), t - s) S (x - \lambda^0_1 (t - s), s) d s\\
    & - & \lambda_2^0 \int_0^t V (\lambda^0_2 (t - s), t - s) S (x -
    \lambda^0_2 (t - s), s) d s\\
    & + & \int_0^t \int_{\lambda_2^0 (t - s)}^{\lambda^0_1 (t - s)}
    \partial_t V (y, t - s) S (x - y, s) d y d s,
  \end{eqnarray*}
  and with similar estimates, we deduce that for $\varepsilon$ small enough,
  \[ | \partial_t v (x, t) | \leqslant \frac{C_0 (\nu)}{2} (F_{a, \gamma_{\nu}
     (0, 1), \delta} + G_{b, \delta}) (x, t) . \]

  {\tmstrong{Step 6.}} Estimates on the second derivatives of $v$
  
  \
  
  The method used in the previous steps does not work, as some terms will
  have three derivatives That is, if we compute $\partial_t^2 v_{\tmop{NL}}$
  for instance, a derivative will fall on $S$ in a boundary term, and it can
  not be moved by integration by parts. This difficulty comes from the fact
  that $\lambda_1$ and $\lambda_2$ are not constants, and they affect the cone
  of light. We are going to use section \ref{ss33v4} to solve this issue.
  
  We recall that
  \[ \partial_t^2 v + (\lambda_1 + \lambda_2) \partial_{x t}^2 v + \lambda_1
     \lambda_2 \partial_x^2 v + \frac{1}{\tau} \partial_t v + \Omega_1
     \partial_x v = 0, \]
  and we define $\delta^0 = \frac{1}{\tau}, w = \Omega_1$ to be consistent
  with the notations of section \ref{ss33v4}.
  
  \
  
  Following subsection \ref{sumsum}, we define the functions $H_1, H_2$ by
  \begin{equation}
    \left\{\begin{array}{l}
      \partial_t H_1 + \lambda_1 \partial_x H_1 = \lambda_1^0 - \lambda_1\\
      H_{1 | t = 0 \nobracket} = 0
    \end{array}\right. \label{eq329v4}
  \end{equation}
  and
  \[ \left\{\begin{array}{l}
       \partial_t H_2 + \lambda_2 \partial_x H_2 = \lambda_2^0 - \lambda_2\\
       H_{2 | t = 0 \nobracket} = 0.
     \end{array}\right. \]

  We define the new variables
  \[ y = x + \tilde{h}_1 (x, t), \nu = t + \tilde{h}_2 (x, t) \]
  where
  \[ \tilde{h}_1 = \frac{\lambda_2^0 H_1 - \lambda_1^0 H_2}{\lambda_2^0 -
     \lambda_1^0}, \tilde{h}_2 = \frac{H_1 - H_2}{\lambda_2^0 - \lambda_1^0} .
  \]
  Let us now compute some estimates on $H_1, H_2$. First, for $x \in
  \mathbb{R}$, we define $X_x$ on $[0, T^{\ast}]$ by
  \[ \left\{\begin{array}{l}
       X'_x (t) = \lambda_1 (X_x (t), t)\\
       X_x (0) = x.
     \end{array}\right. \]
  Now, we infer that for any $y \in \mathbb{R}, t \leqslant T^{\ast}$, there
  exists $x \in \mathbb{R}$ such that $X_x (t) = y$. Indeed, by (\ref{617vf})
  we have with some margin
  \[ | \lambda_1 - \lambda_1^0 | \leqslant \varepsilon \]
  and we take $\varepsilon$ small enough such that $\varepsilon <
  \frac{\lambda_1^0}{2}$. we deduce that $| X_x' (t) - \lambda_1^0 | \leqslant
  \frac{\lambda_1^0}{2}$, thus $\frac{3 \lambda_1^0}{2} \geqslant X_x' (t)
  \geqslant \frac{\lambda_1^0}{2} .$ In particular, $X_{y -
  \frac{\lambda_1^0}{2} t} (t) \geqslant y - \frac{\lambda_1^0}{2} t +
  \int_0^t X_x' (s) d s \geqslant y$ and $X_{y - \frac{3 \lambda_1^0}{2} t}
  (t) \leqslant y$. By continuity of $x \rightarrow X_x (t)$, we deduce that
  there exists $x \in \left[ y - \frac{3 \lambda_1^0}{2} t, y -
  \frac{\lambda_1^0}{2} t \right]$ such that $X_x (t) = y$.
  
  \
  
  Now, we have by differentiating the equation on $H_1$ that
  \[ (\partial_t + \lambda_1 \partial_x) (\partial_t H_1) = - \partial_t
     \lambda_1 (1 + \partial_x H_1) . \]
  Furthermore, by (\ref{eq329v4}) we have $\partial_x H_1 = \frac{\lambda_1^0
  - \lambda_1}{\lambda_1} - \frac{\partial_t H_1}{\lambda_1}$ hence
  \[ (\partial_t + \lambda_1 \partial_x) (\partial_t H_1) = - \partial_t
     \lambda_1 \left( 1 + \frac{\lambda_1^0 - \lambda_1}{\lambda_1} -
     \frac{\partial_t H_1}{\lambda_1} \right) . \]
  We deduce that for any $x \in \mathbb{R}$,
  \[ \partial_t (\partial_t H_1 (X_x (t), t)) = - \left( \partial_t \lambda_1
     \left( 1 + \frac{\lambda_1^0 - \lambda_1}{\lambda_1} - \frac{\partial_t
     H_1}{\lambda_1} \right) \right) (X_x (t), t) . \]
  This implies that
  \begin{eqnarray}
    \partial_t H_1 (X_x (t), t) & = & \partial_t H_1 (X_x (0), 0) \nonumber\\
    & + & \int_0^t - \left( \partial_t \lambda_1 \left( 1 + \frac{\lambda_1^0
    - \lambda_1}{\lambda_1} - \frac{\partial_t H_1}{\lambda_1} \right) \right)
    (X_x (s), s) d s.  \label{est215}
  \end{eqnarray}
  Let us show that for any $t \in [0, T^{\ast}], y \in \mathbb{R}$, we have
  \[ | \partial_t H_1 (y, t) | \leqslant 2 \varepsilon . \]
  We denote by $T \geqslant 0$ the maximum time such that this holds. Since
  $H_{1 | t = 0 \nobracket} = 0$ and $H_1$ satisfies the equation $(\partial_t
  + \lambda_1 \partial_x) H_1 = \lambda_1^0 - \lambda_1$, we have
  \[ | \partial_t H_1 (y, 0) | = | \lambda_1^0 - \lambda_1 (y, 0) | \leqslant
     \varepsilon . \]
  By continuity, we have $T > 0$. Suppose that $T < T^{\ast}$. Then for any $x
  \in \mathbb{R}$, using (\ref{617vf}) and (\ref{est215}), we estimate
  \begin{eqnarray*}
    | \partial_t H_1 (X_x (T^{\ast}), T^{\ast}) | & \leqslant & \varepsilon\\
    & + & \int_0^{T^{\ast}} \left| \left( \partial_t \lambda_1 \left( 1 +
    \frac{\lambda_1^0 - \lambda_1}{\lambda_1} - \frac{\partial_t
    H_1}{\lambda_1} \right) \right) (X_x (s), s) \right| d s\\
    & \leqslant & \varepsilon + C_1 \int_0^{T^{\ast}} \frac{\varepsilon^2}{(1
    + s)^{1 + \nu}} d s\\
    & \leqslant & \varepsilon + C_2 (\nu) \varepsilon^2\\
    & \leqslant & \frac{3}{2} \varepsilon
  \end{eqnarray*}
  given that $\varepsilon > 0$ is small enough (here appears the fact that
  $\varepsilon$ depends on $\nu$). by surjectivity of $x \rightarrow X_x
  (T^{\ast})$ we have a contradiction. This implies that $T = T^{\ast}$ and we
  have the estimate $| \partial_t H_1 (y, t) | \leqslant 2 \varepsilon .$ From
  (\ref{eq329v4}) we also estimate that $| \partial_y H_1 (y, t) | \leqslant 3
  \varepsilon$. We have the same estimates on $H_2$, and by (\ref{htilde27}),
  we have similar estimates on the first derivatives of $\tilde{h}_1,
  \tilde{h}_2$, which is what was needed in subsection \ref{ss33v4} to
  complete the change of variable. In particular,
  \[ (x, t) \rightarrow (x + \tilde{h}_1 (x, t), t + \tilde{h}_2 (x, t)) = (y
     (x, t), \nu (x, t)) \]
  is an inversible function since
  \[ | \partial_t \tilde{h}_1 | + | \partial_x \tilde{h}_1 | + | \partial_t
     \tilde{h}_2 | + | \partial_x \tilde{h}_2 | \leqslant K \varepsilon \ll 1.
  \]
  It is in fact close to the identity, in particular this function and its
  inverse have a norm close to one.
  
  We recall that $t + \tilde{h}_2 (x, t) \geqslant 0$ for all $t \geqslant 0$,
  using the fact that $\tilde{h}_2 (x, 0) = 0$ and $| \partial_t \tilde{h}_2 |
  \leqslant K \varepsilon$ which implies that $| \tilde{h}_2 (x, t) |
  \leqslant K \varepsilon t \leqslant \frac{t}{2}$ given $\varepsilon$ small
  enough. We have then
  \begin{equation}
    \partial_t h_2 (x, t) = \partial_t (t + \tilde{h}_2 (x, t)) = 1 +
    \partial_t \tilde{h}_2 (x, t) \geqslant 1 - 2 \varepsilon \label{timebis}
    .
  \end{equation}
  We then write
  \[ v (x, t) = w (x + h_1 (x, t), t + h_2 (x, t)) = w (y, \nu) \]
  and
  \[ q (x, t) = p (x + h_1 (x, t), t + h_2 (x, t)) = p (y, \nu) . \]
  we recall that the change of variable $\pi (y, \nu) \rightarrow (x, t)$ is
  invertible and close to the identity for $\varepsilon$ small enough. Then,
  by section \ref{ss33v4}, $w$ satisfies
  \begin{eqnarray}
    &  & (\partial_{\nu}^2 w + (\lambda_1^0 + \lambda_2^0) \partial_{y \nu}^2
    w + \lambda_1^0 \lambda_2^0 \partial_y^2 w + \delta^0 \partial_{\nu} w)
    (y, \nu) \nonumber\\
    & = & - (R_1 o \pi) \partial_{\nu} w - R_2 o \pi \partial_y
    w,  \label{623vf}
  \end{eqnarray}
  where $R_1, R_2$ are small functions depending on $H_1, H_2, \partial_x H_1,
  \partial_x H_2, \partial_t H_1, \partial_t H_2, \partial_x \lambda_1,
  \partial_t \lambda_1, \partial_x \lambda_2, \partial_t \lambda_2, w$ but not
  depending on their derivatives. The same can be said for $\pi$, that depends
  on $H_1, H_2$ but not its derivatives.
  
  At $\nu = 0$, we have $t = 0$ from section \ref{ss33v4}, therefore
  \[ w (y, 0) = v (y, 0) = 0, \partial_{\nu} w (y, 0) = \partial_{\nu} \pi_1
     (y, 0) \partial_x v (y, 0) + \partial_{\nu} \pi_2 (y, 0) \partial_t v (y,
     0) = 0. \]
  From Proposition \ref{dwesol}, with $w_0 = w_{| \nu = 0 \nobracket}$ and
  $w_1 = \partial_{\nu} w_{| \nu = 0 \nobracket}$, as well as
  \[ S = - (R_1 o \pi) \partial_{\nu} w - R_2 o \pi \partial_y
     w, \]
  we have
  \begin{eqnarray}
    w (y, \nu) & = & w_L (y, \nu) \nonumber\\
    & + & \int_0^{\nu} \int_{\lambda_2^0 (\nu - s)}^{\lambda^0_1 (\nu - s)} V
    (z, \nu - s) S (y - z, s) d z d s.  \label{eq618vf}
  \end{eqnarray}
  Let us make a few remarks on the domain of definition of $w$, that is
  \[ \Pi \assign \pi^{- 1} (\mathbb{R}_x \times [0, T^{\ast}]) . \]
  Since $\pi$ is $C^1$ and close to the identity, we have $\Pi = \{ (z,
  T_{\pi} (z), z \in \mathbb{R}) \}$ for some smooth function $T_{\pi}$ that
  satisfies $| T_{\pi}' | \leqslant K \varepsilon$. In particular, if $(y,
  \nu) \in \Pi$ and $s \in [0, \nu], z \in [\lambda_2^0 (\nu - s), \lambda_1^0
  (\nu - s)]$, then $(y - z, s) \in \Pi$ if $\varepsilon$ is small enough.
  Therefore in equation (\ref{eq618vf}) the function $S$ will always be taken
  in a point where we have our bootstrap estimate.
  
  \
  
  Now, let us estimate $w, \partial_y w, \partial_{\nu} w$ by using
  \[ w (y, \nu) = v (\pi^{- 1} (y, \nu)) . \]
  We recall that $\pi^{- 1}_1 (y, \nu) = x, \pi_2^{- 1} (y, \nu) = t$. With
  section \ref{ss33v4}, we estimate that
  \[ (1 + \varepsilon) \nu \geqslant \pi_2^{- 1} (y, \nu) \geqslant (1 -
     \varepsilon) \nu, \]
  and that
  \[ | y - \pi^{- 1}_1 (y, \nu) | \leqslant 2 \varepsilon \nu . \]
  Then using the results of steps 4 and 5, we deduce that
  \[ | w (y, \nu) | \leqslant C_0 (\nu) \left( F_{a, \frac{1}{2}, \delta} +
     G_{b, \delta} \right) (\pi^{- 1}_1 (y, \nu), \pi_2^{- 1} (y, \nu)) . \]
  First, we compute
  \begin{eqnarray}
    &  & G_{b, \delta} (\pi^{- 1}_1 (y, \nu), \pi_2^{- 1} (y, \nu))
    \nonumber\\
    & \leqslant & e^{- b \pi_2^{- 1} (y, \nu)} \sup_{z \in [- \delta \pi_2^{-
    1} (y, \nu), \delta \pi_2^{- 1} (y, \nu)]} w_i (\pi^{- 1}_1 (y, \nu) - z)
    \nonumber\\
    & \leqslant & e^{- b (1 - \varepsilon) \nu} \sup_{z \in [- \delta (1 +
    \varepsilon) \nu, \delta (1 + \varepsilon) \nu]} w_i (\pi^{- 1}_1 (y, \nu)
    - z) \nonumber\\
    & \leqslant & e^{- b (1 - \varepsilon) \nu} \sup_{z \in [- \delta (1 + 2
    \varepsilon) \nu, \delta (1 + 2 \varepsilon) \nu]} w_i (y - z) \nonumber\\
    & \leqslant & G_{b (1 - \varepsilon), \delta (1 + 2 \varepsilon)} (y,
    \nu) .  \label{vf624}
  \end{eqnarray}
  We also compute that
  \begin{eqnarray}
    &  & F_{a, \frac{1}{2}, \delta} (\pi^{- 1}_1 (y, \nu), \pi_2^{- 1} (y,
    \nu)) \nonumber\\
    & = & \frac{1}{\sqrt{1 + \pi_2^{- 1} (y, \nu)}} \int_{(\lambda_2^0 -
    \delta) \pi_2^{- 1} (y, \nu)}^{(\lambda_1^0 + \delta) \pi_2^{- 1} (y,
    \nu)} e^{- a \frac{z^2}{1 + \pi_2^{- 1} (y, \nu)}} w_i (\pi^{- 1}_1 (y,
    \nu) - z) d z \nonumber\\
    & \leqslant & \frac{2}{\sqrt{1 + \nu}} \int_{(\lambda_2^0 - \delta')
    \nu}^{(\lambda_1^0 + \delta') \nu} e^{- \frac{a}{(1 + \varepsilon)}
    \frac{(\varepsilon \nu + z)^2}{1 + \nu}} w_i (y - z) d z \nonumber\\
    & \leqslant & \frac{2}{\sqrt{1 + \nu}} \int_{(\lambda_2^0 - \delta')
    \nu}^{(\lambda_1^0 + \delta') \nu} e^{- \frac{a}{2 (1 + \varepsilon)}
    \frac{z^2}{1 + \nu}} w_i (y - z) d z \nonumber\\
    & \leqslant & 2 F_{a_3, \frac{1}{2}, \delta_3} (y, \nu)  \label{vf625}
  \end{eqnarray}
  given that $a_3, \delta_3$ are small enough (compared to $a, \delta_1$). We
  deduce that on $\Pi$ for some $b_3 > 0$,
  \[ | w (y, \nu) | \leqslant 2 C_0 (\nu) \left( F_{a_3, \frac{1}{2},
     \delta_3} + G_{b_3, \delta_3} \right) (y, \nu) . \]
  A similar estimate can be done for $\partial_y w$ and $\partial_{\nu} w$,
  leading to (taking $a_3, b_3, \delta_3$ small enough)
  \[ | \partial_y w (y, \nu) | \leqslant 2 C_0 (\nu) (F_{a_3, 1 - \nu,
     \delta_3} + G_{b_3, \delta_3}) (y, \nu) \]
  and
  \[ | \partial_{\nu} w (y, \nu) | \leqslant 2 C_0 (\nu) (F_{a_3, 1 - \nu,
     \delta_3} + G_{b_3, \delta_3}) (y, \nu) . \]
  Now we estimate the second derivatives with
  \begin{eqnarray*}
    w (y, \nu) & = & w_L (y, \nu)\\
    & + & \int_0^{\nu} \int_{\lambda_2^0 (\nu - s)}^{\lambda^0_1 (\nu - s)} V
    (z, \nu - s) S (y - z, s) d z d s.
  \end{eqnarray*}
  The estimate on $\partial_{\nu}^n \partial_y^k w_L$ for $n + k = 2$ follows
  from initial conditions and Proposition \ref{constlinstab}. For the
  nonlinear part, we first compute
  \begin{eqnarray*}
    &  & \partial_{\nu y}^2 \left( \int_0^{\nu} \int_{\lambda_2^0 (\nu -
    s)}^{\lambda^0_1 (\nu - s)} V (z, \nu - s) S (y - z, s) d z d s \right)\\
    & = & \lambda_1^0 \int_0^{\nu} V (\lambda^0_1 (\nu - s), \nu - s)
    \partial_y S (y - \lambda^0_1 (\nu - s), s) d s\\
    & - & \lambda_2^0 \int_0^{\nu} V (\lambda^0_2 (\nu - s), \nu - s)
    \partial_y S (y - \lambda^0_2 (\nu - s), s) d s\\
    & + & \int_0^{\nu} \int_{\lambda_2^0 (\nu - s)}^{\lambda^0_1 (\nu - s)}
    \partial_{\nu} V (z, \nu - s) \partial_y S (y - z, s) d z d s.
  \end{eqnarray*}
  We recall that $S = - (R_1 o \pi) \partial_{\nu} w - R_2 o \pi
  \partial_y w$. We check that on $\Pi$, with our bootstrap estimates, we have
  \[ | \partial_y S | \leqslant K (\nu) \varepsilon (F_{a_3, 1 - \nu,
     \delta_3} + G_{b_3, \delta_3}) . \]
  Using Lemmas \ref{L33b4}, \ref{L34v4}, \ref{L36v4} and \ref{L37b4}, we can
  thus show by a bootstrap argument that if we take $\varepsilon$ small
  enough, then on $\Pi$ we have
  \[ | \partial_{\nu y}^2 w (y, \nu) | \leqslant C_0 (\nu) (F_{a_3, 1 - \nu,
     \delta_3} + G_{b_3, \delta_3}) (y, \nu) . \]
  Similarly, we can estimate $\partial_{\nu}^2 w$ and get the same estimate.
  By (\ref{623vf}), the equation satisfied by $w$, we deduce the same estimate
  on $\partial_y^2 w$.
  
  Finally, we infer that these estimates holds also in the variables $(x, t)
  \in \mathbb{R} \times [0, T^{\ast}]$ if we consider $\delta_2, a_2, b_2$
  small enough instead of $\delta_3, a_3, b_3$ by doing the reversed estimates
  of (\ref{vf624}) and (\ref{vf625}) (the computations are similar). This
  concludes the estimates on $v$ and its first and second derivatives.
  
  \
  
  {\tmstrong{Step 7.}} Conclusion
  
  \
  
  To conclude the proof, we need to do the same estimates on $q$. The
  equation satisfied by $q$ is (\ref{feqq}) which is of the same form than the
  one on $v$. We infer that a similar proof, with the same steps 4-6 give the
  same estimate on $q$. This concludes the proof of the bootstrap and the fact
  that $T^{\ast} = + \infty$.
\end{proof}

\end{document}